\documentclass[a4paper]{article}

\usepackage{authblk}
\usepackage{maths_symbols}

\newcommand{\pr}[0]{\ensuremath{\textup{pr}}}
\newcommand{\pos}[0]{\ensuremath{\textup{pos}}}
\newcommand{\obs}[0]{\ensuremath{\textup{obs}}}
\newcommand{\var}[0]{\ensuremath{\textup{Var}}}
\newcommand{\opt}[0]{\ensuremath{\textup{opt}}}
\newcommand{\y}[0]{\ensuremath{\textup{y}}}

\begin{document}

\title{Optimal low-rank posterior mean and distribution approximation in linear Gaussian inverse problems on Hilbert spaces}

\author[]{Giuseppe Carere\thanks{~giuseppe.carere@uni-potsdam.de, ORCID ID: 0000-0001-9955-4115} }
\author[]{Han Cheng Lie\thanks{~han.lie@uni-potsdam.de, ORCID ID: 0000-0002-6905-9903}}
\affil[]{~Institut f\"ur Mathematik, Universit\"at Potsdam, Potsdam OT Golm 14476, Germany}
\renewcommand\Affilfont{\small}

\date{}
\maketitle

\begin{abstract}
	We construct optimal low-rank approximations for the Gaussian posterior distribution in linear Gaussian inverse problems with possibly infinite-dimensional separable Hilbert parameter spaces and finite-dimensional data spaces. We first consider approximate posteriors in which the means vary and the posterior covariance is kept fixed, for all possible realisations of the data simultaneously. We give necessary and sufficient conditions for these approximating posteriors to be equivalent to the exact posterior. For such approximations, we measure the data-averaged approximation error with the Kullback--Leibler, R\'enyi and Amari $\alpha$-divergences for $\alpha\in(0,1)$, and the Hellinger distance. With the loss in Kullback--Leibler and R\'enyi divergences, we find the optimal approximations and formulate an equivalent condition for their uniqueness, extending the work in finite dimensions of Spantini et al.\ (SIAM J.\ Sci.\ Comput.\ 2015). We then consider joint low-rank approximation of the mean and covariance. For the reverse Kullback--Leibler divergence, the optimal approximations of the mean and of the covariance yield an optimal joint approximation of the mean and covariance. We interpret one such joint approximation in terms of an optimal projector in parameter space, and show that this approximation amounts to solving a Bayesian inverse problem with projected forward model. Extensive numerical examples demonstrate some of our theoretical findings.
\end{abstract}


\vskip2ex
\textbf{Keywords}: Nonparametric linear Bayesian inverse problems, Gaussian measure, low-rank operator approximation, equivalent measure approximation, projected inverse problem
\vskip1ex
\noindent
\textbf{MSC codes}: Primary: 60G15, 62F15, 62G05; Secondary: 28C20, 47A58


\section{Introduction}
\label{sec:introduction}

Linear inverse problems are characterised by a linear map $G$ that encodes the underlying model and the observation process of the problem at hand. That is, $G$ describes the known relationship between the unknown parameter $x^\dagger$ to be inferred and the data, which is a noisy observation of $Gx^\dagger$. The parameter $x^\dagger$ is often a function, such as a diffusivity field in a partial differential equation. 

Inference on $x^\dagger$ essentially amounts to inverting the operator $G$. Such inversion is typically an ill-posed operation, due to the smoothing nature of $G$. For example, if $G$ involves application of an elliptic partial differential equation, then $G$ typically has quickly decaying spectrum, since the inverse Laplacian has quickly decaying spectrum. Furthermore, inference of a function $x^\dagger$ based on a finite amount of observations need not be uniquely possible. For these reasons, regularisation is required.
Bayesian methods can be seen as a way to regularise the inverse problem, and also naturally allow for uncertainty quantification. To quantify the uncertainty, the posterior covariance operator is essential.

The Bayesian method for inferring $x^\dagger$ involves considering $x^\dagger$ as a random variable $X$ with specified distribution and finding the conditional distribution of $X$ given the data. The prior distribution is the chosen distribution of $X$ and the posterior distribution is the resulting conditional distribution of $X$ given the data. The spread of the posterior distribution can then be interpreted as a quantification of uncertainty.

For linear inverse problems, a Gaussian prior is a convenient choice because in this case the posterior is also Gaussian with explicit expressions for its mean and covariance. We choose a nondegenerate prior distribution $X\sim\mathcal{N}(m_\pr,\mathcal{C}_\pr)$ and assume the data $y$ is obtained via the linear observation model
\begin{equation}
	\label{eqn:observation_model}
	 Y=GX+\zeta,\quad \zeta\sim \mathcal{N}(0,\mathcal{C}_\obs),
\end{equation}
where $\mathcal{N}(0,\mathcal{C}_\obs)$ is nondegenerate observation noise with known covariance $\mathcal{C}_\obs$ and zero mean, and $Y$ takes values in $\R^n$. For a given realisation $y\in\R^n$ of $Y$, the posterior distribution then is $\mathcal{N}(m_\pos,\mathcal{C}_\pos)$, where
 \begin{align*}
	 m_\pos = m_\pr+ \mathcal{C}_{\pos}G^*\mathcal{C}_{\obs}^{-1}(y-Gm_\pr),\quad	 \mathcal{C}_{\pos} = \mathcal{C}_{\pr} - \mathcal{C}_{\pr} G^*(\mathcal{C}_{\obs}+G\mathcal{C}_{\pr} G^*)^{-1}G\mathcal{C}_{\pr},
\end{align*}
see \cite[Example 6.23]{Stuart2010}. The posterior covariance $\mathcal{C}_\pos$ is independent of $y$; only the posterior mean $m_\pos$ depends on the realisation of the data.

These explicit expressions hold both in the case that $X$ is an element of a finite-dimensional or infinite-dimensional Hilbert space.
In the latter case, however, a computational solution of the problem requires its discretisation, after which the resulting finite-dimensional Bayesian inverse problem can be solved numerically.

For such finite-dimensional posterior distributions, various works have studied its approximation, which for tractability in terms of computation and storage may be essential. The update from prior to posterior distribution is determined by the choice of prior, by the structure \eqref{eqn:observation_model} of the inverse problem and by the observed data $y$. Low dimensionality of this update lies at the core of approximation procedures considered in \cite{Flath2011,Cui2014,Zahm2022,Li2024a,Li2024b,Cui2016b,Spantini2015}. In \cite{Flath2011}, low-rank approximation for Gaussian linear inverse problems is considered, while \cite{Spantini2015} proves optimality for low-rank approximations of posterior mean and covariance. Low-rank approximation for nonlinear and non-Gaussian problems is studied in \cite{Cui2014,Cui2016b,Zahm2022,Li2024a,Li2024b}. The work of \cite{Cui2016b} describes an algorithm which exploits the low-rank structure of the prior-to-posterior update for certain nonlinear problems based on the ideas developed in finite dimensions, but which can also target infinite-dimensional posteriors. A common feature of these approximations is that they exploit the low-rank structure of the Bayesian prior-to-posterior update, and not just low-rank structure of the prior or forward model. Also other approximation methods exist, such as variational methods, e.g.\ \cite{Pinski2015}.

The optimality of specific low-rank approximations of the posterior mean in finite-dimensional linear Gaussian inverse problems is studied in \cite{Spantini2015}. Such an approximation may prove useful in a many-query setting, in which the posterior mean has to be recomputed for many different realisations of the data. In \cite[Section 4]{Spantini2015}, the approximation error is quantified by considering a Bayes risk, which averages over the data. A goal-oriented version is constructed in \cite{Spantini2017}. The approximation method developed in \cite{Li2024a} also targets approximation of the posterior distribution, and hence the posterior mean, but does so for a specific realisation of $y$.

Instead of discretising the problem, optimal low-rank approximations can also be studied directly for the infinite-dimensional posterior. In order to show consistency of the optimal low-rank approximations constructed for discretised versions of the inverse problem, an optimal low-rank approximation problem in the infinite-dimensional setting is required. Then, once a specific approximation scheme is chosen for a given inverse problem, this infinite-dimensional optimal approximation can be used to show discretisation independence of the approximation method. This is similar in spirit to how \cite{Bui-Thanh2016} shows dimension independence of a sampling scheme for a finite element based discretisation of certain partial differential equations, using the infinite-dimensional results on sampling methods established in \cite{Cotter2013,Beskos2011}. Numerical evidence that discretisation independence should hold in a specific setting was found in \cite{Bui-Thanh2013}.

In this work we aim to analyse and provide such optimal low-rank approximations for the posterior mean directly in the Hilbert space formulation. Furthermore, using the results of \cite{PartI} on optimal low-rank posterior covariance approximations in Hilbert spaces, we also identify low-rank joint approximations of the posterior mean and covariance. This allows us to obtain discretisation-independent and dimension-independent optimal low-rank posterior approximations.

\subsection{Challenges of posterior mean approximation in infinite dimensions}

Technical difficulties arise for posterior mean approximation in infinite dimensions. As for posterior covariance approximations, these are in part due to the fact that the Cameron--Martin space $\ran{\mathcal{C}_\pr^{1/2}}$ is a proper subspace of $\H$. That is, $\mathcal{C}_\pr^{1/2}$ is not surjective, and neither is $\mathcal{C}_\pr$ since $\ran{\mathcal{C}_\pr}\subset\ran{\mathcal{C}_\pr^{1/2}}$. Furthermore, if $\mathcal{C}_\pr$ and $\mathcal{C}_\pr^{1/2}$ are injective, then we can define the inverses as unbounded operators which are only defined on a dense subspace, c.f.\ \Cref{lemma:covariance_properties}\ref{item:covariance_properties_2}. This is in contrast with the finite-dimensional setting, in which all the operators involved are bounded and defined everywhere.

Even if the posterior covariance is kept fixed, an approximation of the posterior mean can result in an approximate posterior distribution which need not be equivalent to the exact posterior distribution, in the sense that the approximate distribution is not absolutely continuous with respect to the exact posterior distribution. 
In fact, when the approximate and exact posterior are not equivalent, they are mutually singular by the Feldman--Hajek theorem. If the approximate posterior is mutually singular with respect to the exact posterior measure, then the approximate posterior assigns positive probability only to events that have zero probability under the exact posterior, and events which have positive posterior probability have zero probability under the approximate measure. 
The issue of equivalence to the exact posterior for almost every realisation of the data is also present in the case of joint approximation of the mean and covariance.

In the finite-dimensional setting of [31, Section 4], the Bayes risk is used to measure the error  of the approximate posterior mean. Since the same Bayes risk is infinite in the infinite-dimensional setting, an alternative measurement of the error of the approximate posterior mean is required.

\subsection{Contributions}

We formulate two types of low-rank posterior mean approximations: structure-preserving and structure-ignoring approximations. One type preserves the structure of the prior-to-posterior mean update as a function of the data, while the other does not. Keeping the exact posterior covariance fixed, the posterior mean approximations lead to approximate posterior distributions. Not every low-rank posterior mean update retains equivalence between the corresponding approximate posterior distribution and the exact posterior. In fact, direct generalisation to infinite dimensions of the low-rank updates of \cite[Section 4]{Spantini2015} leads to nonequivalent approximations in general. In \Cref{prop:mean_approximation_problem_equivalence}, we characterise, for both the structure-preserving and structure-ignoring posterior mean approximations, which approximations satisfy this equivalence property. Here, equivalence holds not only for one realisation of the data $y$, but for all realisations in a set of probability 1. This is the first main contribution of the paper.

The second main contribution is to solve the Gaussian measure approximation problems for approximating the posterior mean using the low-rank update classes mentioned in the previous paragraph. We keep the exact posterior covariance fixed and quantify the accuracy of an approximation using the R\'enyi, Amari, Hellinger, and forward and reverse Kullback-Leibler divergences, averaged over the data distribution. That is, we consider approximations of the mean that are accurate on average, rather than for a specific realisation of $y$. These losses are related to the weighted Bayes risk considered in the finite-dimensional case of \cite{Spantini2015} and are a natural generalisation to infinite dimensions. The approximation problems rely on a generalisation of the result on reduced-rank matrix approximation by \cite{Sondermann1986} and \cite{Friedland2007} to infinite dimensions, which can be found in \cite{CarereLie2024}. The solutions and the corresponding minimal losses in Kullback--Leibler and R\'enyi divergences are identified in \Cref{thm:optimal_low_rank_mean_approx,thm:optimal_low_rank_update_mean_approx}, and upper bounds for the Hellinger distance and Amari $\alpha$-divergences are obtained in \Cref{cor:optimal_mean_for_amari_and_hellinger}. The resulting optimal approximations share the property with $m_\pos$ that they lie in $\ran{\mathcal{C}_\pos}$ with probability 1, and hence in $\ran{\mathcal{C}_\pr^{}}$ with probability 1, since $\ran{\mathcal{C}_\pos}=\ran{\mathcal{C}_\pr}$ for Gaussian linear inverse problems, see \cite[eq.\ (6.13a)]{Stuart2010}. \Cref{thm:optimal_low_rank_mean_approx,thm:optimal_low_rank_update_mean_approx} and \Cref{cor:optimal_mean_for_amari_and_hellinger} thus extend the results of \cite[Section 4]{Spantini2015} to an infinite-dimensional setting, and also give necessary and sufficient conditions for uniqueness of the optimal approximations. 

The third main contribution is to consider the family of measure approximation problems where both the posterior mean and posterior covariance are jointly approximated. We construct approximations of the posterior which are equivalent to the exact posterior, for all realisations of $Y$ in a set of probability 1. We measure the error in terms of the reverse Kullback--Leibler divergence, averaged over $Y$. 
The reverse Kullback--Leibler divergence is given by $\int \log(\frac{\d{\widetilde{\mu}_\pos}}{\d{\mu_\pos}})\d{\widetilde{\mu}_\pos}$, where $\widetilde{\mu}_\pos$ and $\mu_\pos=\mathcal{N}(m_\pos,\mathcal{C}_\pos)$ denote the approximate posterior and exact posterior respectively.
This divergence is important in variational approximation methods, see e.g.\  \cite[Theorem 5]{Ray2022}. 
In \Cref{prop:optimal_joint_approximation}, we exploit the Pythagorean structure of the expression of the Kullback--Leibler divergence between Gaussians. This allows us to show that the problem of finding an optimal low-rank joint approximation of the mean and covariance can be solved by combining an optimal solution of the low-rank covariance approximation problem in \cite[Theorem 4.21]{PartI} with an optimal solution of the low-rank mean approximation problem given in \Cref{thm:optimal_low_rank_mean_approx,thm:optimal_low_rank_update_mean_approx} below. The mean, covariance and joint approximation problems have the same necessary and sufficient condition for uniqueness of solutions. 
The optimal joint approximation result of \Cref{prop:optimal_joint_approximation} and its interpretation via optimal projection given in \Cref{prop:optimal_projector} provide a perspective on low-rank posterior Gaussian measure approximation which combines the insights obtained in the separate mean and covariance approximation procedures.

As shown in \cite{PartI} and recalled below in \Cref{prop:bayesian_feldman_hajek}, the Bayesian prior-to-posterior update occurs only on a finite-dimensional subspace of the parameter space. The optimal joint approximation to the posterior only differs significantly from the prior in \emph{a few} directions of the parameter space, if the optimal approximation is accurate. This follows from \Cref{prop:optimal_projector}, which shows that the optimal approximate posterior that results from the structure-ignoring posterior mean approximation can be obtained as the exact posterior corresponding to a projected version of the Bayesian inverse problem \eqref{eqn:observation_model}, in which $G$ is precomposed by a low-rank projector in parameter space. Thus, if the low-rank approximation is accurate, the prior-to-posterior update on the infinite-dimensional parameter space essentially occurs on a low-dimensional subspace of the parameter space.

\subsection{Outline}

Background concepts and key notation are summarised in \Cref{sec:notation}. \Cref{sec:formulation} presents the linear Bayesian inverse problem and introduces the approximation families we consider for posterior mean approximation. In \Cref{sec:equivalence_and_divergences_between_gaussian_measures} we describe the divergences which are used to measure approximation errors. This section also describes the notion of equivalence of Gaussian measures and expands on the relevant operators for the analysis of the Bayesian update. Certain aspects of low-rank posterior covariance approximation are briefly recalled in \Cref{sec:optimal_approximation_covariance}. In this section we also interpret the prior-to-posterior update in terms of variance reduction. Optimal low-rank posterior mean approximation is considered in \Cref{sec:optimal_approximation_mean}. Joint posterior mean and covariance approximation is discussed in \Cref{sec:optimal_joint_approximation}, and in \Cref{sec:optimal_projector} we interpret the results of the previous section in terms of an optimal projection of the likelihood function on a low-dimensional subspace in parameter space. In \Cref{sec:examples_short}, we consider two examples of linear Gaussian inverse problems, namely, deconvolution and inferring the initial condition of a heat equation, for which we identify the operators relevant for the low-rank approximations. An example involving the heat equation on a two-dimensional spatial domain is implemented in \Cref{sec:numerical_example}, in which we verify numerically several aspects of our theoretical findings. We conclude in \Cref{sec:conclusion}. Auxiliary results required in the analysis are summarised in \Cref{sec:theoretical_facts}. Proofs can be found in \Cref{sec:proofs_of_results}. \Cref{sec:examples} provides detailed calculations for the examples in \Cref{sec:examples_short}.

\subsection{Notation}
\label{sec:notation}

To introduce the notation, we let $\mathcal{H}$ and $\mathcal{K}$ be separable Hilbert spaces, that is, complete inner product spaces with a countable orthonormal basis (ONB). We denote the linear spaces of linear operators defined with domain $\H$ and codomain $\mathcal{K}$ which are bounded, compact and finite-rank by, respectively, $\B(\H,\mathcal{K})$, $\B_0(\mathcal{H},\mathcal{K})$ and $\B_{00}(\mathcal{H},\mathcal{K})$. A linear operator is said to have `finite rank' if it is bounded and its range is finite-dimensional. The set of finite-rank operators which have rank at most $r\in\N$ is denoted by $\B_{00,r}(\H,\mathcal{K})$. The above sets are all endowed with the operator norm $\norm{\cdot}$ defined by $\norm{T}\coloneqq \sup\{\norm{Th}:\norm{h}\leq 1\}$. The trace-class and Hilbert--Schmidt operators are compact operators with summable and square-summable eigenvalue sequence respectively, and are denoted by $L_1(\mathcal{H},\mathcal{K})$ and $L_2(\mathcal{H},\mathcal{K})$ respectively. Their respective norms are denoted by $\norm{\cdot}_{L_1(\H,\mathcal{K})}$ and $\norm{\cdot}_{L_2(\H,\mathcal{K})}$. Thus, $\norm{T}_{L_1(\H,\mathcal{K})}$ and $\norm{T}_{L_2(\H,\mathcal{K})}^2$ are computed by summing respectively the absolute values and squares of the singular values of $T$. If $\H=\mathcal{K}$, then we write $\B(\H)$ instead of $\B(\H,\mathcal{K})$, and similarly for the other sets above. We have the inclusion of sets $\B_{00,r}(\H)\subset\B_{00}(\H)\subset L_1(\H)\subset L_2(\H)\subset\B_0(\H)\subset\B(\H)$.

The operator $T^*\in\B(\mathcal{K},\H)$ denotes the adjoint of $T\in\B(\H,\mathcal{K})$. By $\B(\mathcal{H})_\R$ we denote the subspace of $\B(\H)$ that consists of self-adjoint operators. We similarly define the spaces $\mathcal{B}_0(\H)_\R$, $\mathcal{B}_{00}(\H)_\R$, $L_1(\H)_\R$ and $L_2(\H)_\R$, and the set $\B_{00,r}(\H)_\R$. 

If $T\in\B(\H)$, then we call $T$ `nonnegative' or `positive' if $\langle Th,h\rangle \geq 0$ or $\langle Th,h\rangle>0$ for all nonzero $h\in\H$ respectively, and write $T\geq 0$ and $T> 0$ respectively. For self-adjoint and nonnegative $T$, there exists a self-adjoint and nonnegative square root $T^{1/2}\in\B(\H)_\R$. If $T>0$, then $T^{1/2}>0$.

For $h\in\H$ and $k\in\mathcal{K}$, the tensor product $h\otimes k\in\B_{00,1}(\H,\mathcal{K})$ denotes the rank-1 operator $(k\otimes h)(z)=\langle h,z\rangle k$, $z\in\H$. Any $T\in\B_0(\H,\mathcal{K})$ has a singular value decomposition (SVD) $T=\sum_{i}^{}\sigma_i k_i\otimes h_i$, where $(\sigma_i)_i$ is a nonnegative and nonincreasing sequence converging to zero and $(h_i)_i$ and $(k_i)_i$ are orthonormal sequences in $\H$ and $\mathcal{K}$ respectively, c.f.\ \Cref{lemma:operator_svd}.

For $T\in\B(\H)$, we denote by $T^\dagger$ the Moore--Penrose inverse of $T$, also known as the generalised inverse and pseudo-inverse of $T$, c.f.\ \cite[Definition 2.2]{engl_regularization_1996}, \cite[Section B.2]{da_prato_stochastic_2014} or \cite[Definition 3.5.7]{Hsing2015}. It holds that $T^\dagger$ is bounded if and only if $\ran{T}$ is closed, c.f.\ \cite[Proposition 2.4]{engl_regularization_1996}. If $T$ is injective, then $T^\dagger=T^{-1}$ on $\ran{T}$.

We also briefly introduce the notion of an unbounded operator $T$ between $\H$ and $\mathcal{K}$. Such an operator is defined on a dense, possibly proper subspace $\dom{T}$ of $\H$, and is not necessarily bounded. We write $T:\H\rightarrow\mathcal{K}$ or $T:\dom{T}\subset\H\rightarrow\mathcal{K}$ or $T:\dom{T}\rightarrow\mathcal{K}$ for such unbounded operators $T$. Note that the term `unbounded operator' encompasses the bounded operators as well.
Sums and compositions of unbounded operators are defined as follows. If $T:\mathcal{H}\rightarrow\mathcal{K}$, $S:\mathcal{H}\rightarrow\mathcal{K}$ and $U:\mathcal{K}\rightarrow\mathcal{Z}$ for some separable Hilbert space $\mathcal{Z}$, then $T+S:\dom{T+S}\subset\mathcal{H}\rightarrow\mathcal{K}$ with $\dom{T+S}\coloneqq\dom{T}\cap\dom{S}$ and $UT:\dom{UT}\subset\mathcal{H}\rightarrow\mathcal{Z}$ with $\dom{UT}\coloneqq T^{-1}(\dom{U})$.

If $T\in\B(\mathcal{H})$ is positive and self-adjoint, then the norm $\norm{\cdot}_{T^{-1}}$ on $\ran{T}$ is defined by $\norm{h}_{T^{-1}}=\norm{T^{-1/2}h}$, for $h\in\ran{T}$. Here $T^{-1/2}:\ran{T^{1/2}}\subset\H\rightarrow\H$ is the unbounded inverse of $T^{1/2}$.

Two measures $\mu$ and $\nu$ are equivalent, i.e.\ $\mu\sim\nu$, if they are absolutely continuous with respect to each other. That is $\mu(A)=0$ implies $\nu(A)=0$ for every measurable set $A$, and vice versa. Thus, $\mu$ has a density with respect to $\nu$ and vice versa. We denote the support of a measure $\mu$ by $\supp{\mu}$.

If a random variable $X$ has distribution $\mu$, we write $X\sim \mu$. We write $X\sim\mathcal{N}(m,\mathcal{C})$ if $\langle X,h\rangle\sim\mathcal{N}(\langle m,h\rangle,\langle \mathcal{C}h,h\rangle)$ for every $h\in\H$. In this case, we say that $X$ has a Gaussian distribution on $\H$ with mean $m$, covariance $\mathcal{C}$, and precision $\mathcal{C}^{-1}$, where $m=\mathbb{E}X$ and $\langle\mathcal{C}h,k\rangle=\mathbb{E}\langle h,X-m\rangle\langle X-m,k\rangle$ for all $h,k\in\H$. 

By $\ell^2(I)$ we denote the space of square-summable sequences on a non-empty interval $I\subset\R$. That is, $\ell^2(I)\coloneqq\{(x_i)_{i\in\N}\subset I\ :\ \sum_{i\in\N} \abs{x_i}^2<\infty\}$.

A statement that depends on a random variable is said to hold `almost surely', or `a.s.', if it holds with probability 1 with respect to the distribution of that random variable.

We indicate the replacement of $a$ with $b$ by `$a\leftarrow b$'.

\section{Low-rank posterior mean approximations}
\label{sec:formulation}

We consider a possibly infinite-dimensional parameter space $\H$, which is assumed to be a separable Hilbert space. In the Bayesian framework, the unknown parameter $X$ is an $\H$-valued random variable. We assume that the prior distribution $\mu_\pr$ of $X$ satisfies the following.

\begin{assumption}
	\label{ass:nondegeneracy}
	We assume $\mu_\pr$ is a nondegenerate and centered Gaussian measure on $\H$.
\end{assumption}

Hence, $X$ is distributed according to $X\sim\mu_\pr=\mathcal{N}(0,\mathcal{C}_\pr)$, where the prior covariance $\mathcal{C}_\pr$ is a self-adjoint operator. The data constitutes a finite amount of noisy observations of linear functions of $X$. That is, there exists an $n\in\N$, a linear and continuous map $G\in\B(\H,\R^n)$ known as the `forward model', and a multivariate normal random variable $\zeta$ on $\R^n$ such that the model \eqref{eqn:observation_model} is satisfied. Here, $n$, $G$, and the noise covariance $\mathcal{C}_\obs$ are all assumed to be known. We assume that $\mathcal{C}_\obs$ is invertible, so that $\zeta$ has a probability density on $\R^n$. We also assume that $\zeta$ and $X$ are statistically independent. In practice, only one realisation $y\in\R^n$ of $Y$ is observed,  and the Bayesian inverse problem amounts to finding the distribution of $X\vert Y=y$ on $\H$. This is called the posterior distribution and is indicated by $\mu_\pos(y)$. 

We have thus specified the distribution of the random variable $(X,Y)$ by prescribing the marginal distribution of $X$, i.e.\ the prior distribution, and by prescribing the distribution of $Y\vert X=x$ for any $x\in\H$ via \eqref{eqn:observation_model}. The latter distribution admits a probability density function on $\R^n$, known as the `likelihood', which is proportional to $y\mapsto\exp{(-\frac{1}{2}\norm{\mathcal{C}_\obs^{-1/2}(Gx-y)}^2)}$. As a function of $x$, the negative log-likelihood has a Hessian $H$ given by
\begin{align}
	\label{eqn:hessian}
	H = G^*\mathcal{C}_\obs^{-1}G\in\B_{00,n}(\H)_\R.
\end{align}
In statistics, $H$ is also known as the Fisher information operator, but we shall refer to it as ``the Hessian''. We have $H=(\mathcal{C}_\obs^{-1/2}G)^*(\mathcal{C}_\obs^{-1/2}G)$ and hence $H$ is self-adjoint and nonnegative. Furthermore, by \Cref{lemma:range_of_square_of_finite_rank} and the invertibility of $\mathcal{C}_\obs^{-1/2}$, $\rank{H}=\rank{(\mathcal{C}_\obs^{-1/2}G)^*}=\rank{\mathcal{C}_\obs^{-1/2}G}=\rank{G}$. 

With the chosen distributions of $X$ and $Y\vert X$, we have also specified the distributions of $Y$ and $X\vert Y=y$, i.e.\ the data distribution and the posterior distribution. They are both Gaussian: $Y\sim\mathcal{N}(0,\mathcal{C}_\obs+G\mathcal{C}_\pr G^*)$ and $X\vert Y=y\sim\mathcal{N}(m_\pos,\mathcal{C}_\pos)$, where by \cite[Example 6.23]{Stuart2010},
\begin{subequations}
	\label{eqn:update_equations}
	\begin{align}
		\label{eqn:pos_mean}
		m_\pos = m_{\pos}(y) &= \mathcal{C}_{\pos}G^*\mathcal{C}_{\obs}^{-1}y\in\ran{\mathcal{C}_\pos},\\
		\label{eqn:pos_covariance}
		\mathcal{C}_{\pos} &= \mathcal{C}_{\pr} - \mathcal{C}_{\pr} G^*(\mathcal{C}_{\obs}+G\mathcal{C}_{\pr} G^*)^{-1}G\mathcal{C}_{\pr},\\
		\label{eqn:pos_precision}
		\mathcal{C}_\pos^{-1} &= \mathcal{C}_\pr^{-1} + G^*\mathcal{C}_\obs^{-1}G = \mathcal{C}_\pr^{-1}+H.
	\end{align}
\end{subequations}
The posterior mean depends on $y$ and lies in $\ran{\mathcal{C}_\pos}$, by \eqref{eqn:pos_mean}. The posterior covariance is independent of $y$, as \eqref{eqn:pos_covariance} shows. 

Equation \eqref{eqn:pos_precision} requires some interpretation. Since $\mu_\pr$ is nondegenerate by \Cref{ass:nondegeneracy}, $\supp{\mu_\pr}=\H$, c.f.\ \cite[Definition 3.6.2]{bogachev_gaussian_1998} and $\mathcal{C}_\pr$ is positive, hence injective, c.f.\ \Cref{lemma:covariance_properties,lemma:positive_is_injective_nonnegative}. Therefore, we can invert $\mathcal{C}_\pr$ on its range $\ran{\mathcal{C}_\pr}$. Also $\mathcal{C}_\pr^{1/2}$ is injective, and hence $\ran{\mathcal{C}_\pr^{1/2}}$ is dense in $\H$, see \Cref{lemma:kernel_of_square,lemma:kernel_range}. For a fixed $y$, the measures $\mu_\pr$ and $\mu_\pos(y)$ are equivalent, see \cite[Theorem 6.31]{Stuart2010}. Thus, by the Feldman--Hajek theorem, which is recalled in \Cref{thm:feldman--hajek}, also $\ran{\mathcal{C}_\pos^{1/2}}$ is dense in $\H$. We conclude that also $\mathcal{C}_\pos$ and $\mathcal{C}_\pos^{1/2}$ are injective, and $\mathcal{C}_\pos^{-1}$ is a densely-defined operator with $\dom{\mathcal{C}_\pos^{-1}}=\ran{\mathcal{C}_\pos}$. Equation \eqref{eqn:pos_precision} now states that $\dom{\mathcal{C}_\pos^{-1}}=\dom{\mathcal{C}_\pr^{-1}+H}$. Since $H=G^*\mathcal{C}_\obs^{-1}G\in\B(\H)$, c.f.\ \eqref{eqn:hessian}, is defined on all of $\H$, this shows $\dom{\mathcal{C}_\pos^{-1}}=\dom{\mathcal{C}_\pr^{-1}}$. Hence, $\ran{\mathcal{C}_\pos}=\ran{\mathcal{C}_\pr}$, and this subspace forms the domain of definition of \eqref{eqn:pos_precision}.

In infinite dimensions, $\mathcal{C}_\pr^{-1}:\ran{\mathcal{C}_\pr}\rightarrow\H$ and $\mathcal{C}_\pr^{-1/2}:\ran{\mathcal{C}_\pr^{1/2}}\rightarrow\H$ are unbounded operators, i.e.\ discontinuous linear functions. We have the range inclusion $\ran{\mathcal{C}_\pr}\subset\ran{\mathcal{C}_\pr^{1/2}}$. Furthermore, the ranges $\ran{\mathcal{C}_\pr^{1/2}}$ and $\ran{\mathcal{C}_\pr}$ take a central role in the Bayesian inverse problem. They are called the `Cameron--Martin space' and `pre-Cameron--Martin space' of the prior respectively, and are both proper subspaces of $\H$. These spaces are endowed with the Cameron--Martin norm $\norm{\cdot}_{\mathcal{C}_\pr^{-1}}$ defined by $\norm{h}_{\mathcal{C}_\pr^{-1}}=\norm{\mathcal{C}_\pr^{-1/2}h}$. Since the Cameron--Martin space characterises a Gaussian measure, equivalence between Gaussian measures depends on their Cameron--Martin spaces. Furthermore, as discussed in the previous paragraph, these spaces are also involved in the update equations \eqref{eqn:update_equations}. For both reasons, the analysis of posterior approximations will therefore make use of these spaces.

In this work we mostly focus on the approximation of the posterior mean in \eqref{eqn:pos_mean}. We shall construct approximations $\widetilde{m}_\pos(y)$ of the exact posterior mean $m_\pos(y)$, such that the resulting approximate posterior $\mathcal{N}(\widetilde{m}_\pos(y),\mathcal{C}_\pos)$ and the exact posterior $\mathcal{N}(m_\pos(y),\mathcal{C}_\pos)$ are equivalent.  This equivalence should not only hold for one fixed $y$, but for every possible realisation $y$ of $Y$ in a set of probability 1 with respect to the distribution of $Y$, so that equivalence is guaranteed prior to observing the data.

For approximations of the posterior mean, we observe from \eqref{eqn:pos_mean} that the posterior mean is the result of applying an operator to the data $y$.
This motivates the following classes of operators:
\begin{subequations}
	\label{eqn:class_approx_means}
	\begin{align}
		\label{eqn:class_approx_mean_with_covariance_update}
		\mathscr{M}^{(1)}_r \coloneqq& \{(\mathcal{C}_\pr-B)G^*\mathcal{C}_\obs^{-1}:\ B\in\B_{00,r}(\H),\ \mathcal{N}( (\mathcal{C}_\pr-B)G^*\mathcal{C}_\obs^{-1}Y,\mathcal{C}_\pos)\sim\mu_\pos(Y)\quad\text{a.s.}\},
		\\
		\label{eqn:class_approx_mean_without_covariance_update}
		\mathscr{M}^{(2)}_r \coloneqq& \{A\in\B_{00,r}(\R^n,\H):\ \mathcal{N}( AY,\mathcal{C}_\pos)\sim\mu_\pos(Y)\quad \text{a.s.}\}.
	\end{align}
\end{subequations}
In this way, we ensure that by approximating the posterior mean by $Ay$ for $A\in\mathscr{M}^{(i)}_r$, the equivalence with $\mu_{\pos}(y)$ is maintained for all $y$ in a set of probability 1 with respect to the distribution of $Y$. We stress that $A$ is constructed before a specific realisation $y$ of $Y$ is observed.
The structure-preserving class in \eqref{eqn:class_approx_mean_with_covariance_update} takes into account properties of the posterior mean and covariance that are implied by \eqref{eqn:pos_mean}-\eqref{eqn:pos_covariance}.
In contrast, the structure-ignoring class in \eqref{eqn:class_approx_mean_without_covariance_update} ignores these properties and only requires that the posterior mean is a linear transformation of the data and that the resulting approximate posterior approximation is equivalent to the exact posterior.
We note that the rank-$r$ update $-B$ of $\mathcal{C}_\pr$ in \eqref{eqn:class_approx_mean_with_covariance_update} is not required to be self-adjoint. However, as we shall see in \Cref{sec:optimal_approximation_mean}, the posterior mean approximations of the form \eqref{eqn:class_approx_mean_with_covariance_update} which are optimal, in the sense specified in \Cref{sec:optimal_approximation_mean}, do in fact correspond to self-adjoint updates $-B$.

By \eqref{eqn:pos_mean}, it follows that there exists $r_0\leq n$ such that $m_\pos\in\mathscr{M}^{(1)}_r \cap \mathscr{M}^{(2)}_r$ for all $r\geq r_0$.
Indeed, if $r\geq \rank{G^*}=\rank{G}$, then $(\mathcal{C}_\pr-B)G^*\mathcal{C}_\obs^{-1}\in \B_{00,r}(\R^n,\H)$ for every $B\in\B_{00,r}(\H)_\R$. Thus, $\mathscr{M}^{(1)}_r\subset \mathscr{M}^{(2)}_r$ for $r\geq \rank{G}$. Since $\mathcal{C}_{\pr} G^*(\mathcal{C}_{\obs}+G\mathcal{C}_{\pr} G^*)^{-1}G\mathcal{C}_{\pr}$ has rank at most $\rank{G}$, \eqref{eqn:pos_mean}-\eqref{eqn:pos_covariance} show $m_\pos\in\mathscr{M}^{(1)}_r\subset\mathscr{M}^{(2)}_r$ for $r\geq \rank{G}$.

Because the rank of $A$ and $B$ in \eqref{eqn:class_approx_mean_with_covariance_update} and \eqref{eqn:class_approx_mean_without_covariance_update} are restricted and may be much smaller than $n$, we refer to $Ay$ for $A\in\mathscr{M}^{(i)}_r$, $i=1,2$, as a `low-rank' approximation of $m_\pos(y)$. If $\dim{\H}<\infty$, then $\mathscr{M}^{(i)}_r$ coincides with the approximation classes considered in \cite[Section 4]{Spantini2015}.

\section{Equivalent Gaussian measures and Bayesian inference}
\label{sec:equivalence_and_divergences_between_gaussian_measures}

%

We quantify posterior approximation errors using various divergences. Let $\nu_2$ be a target measure on $\H$ and $\nu_1$ an approximation of $\nu_2$ satisfying $\nu_1\sim\nu_2$. We use the $\rho$-R\'enyi divergence, the forward Kullback-Leibler (KL) divergence, the Amari $\alpha$-divergence for $\alpha\in(0,1)$ and the Hellinger distance, defined respectively by,
\begin{align*}
	D_{\kl}(\nu_2\Vert\nu_1) &\coloneqq  \int_\mathcal{H}\log{\frac{\d \nu_2}{\d \nu_1}}\d\nu_2,\\
	D_{\ren,\rho}(\nu_2\Vert\nu_1) &\coloneqq -\frac{1}{\rho(1-\rho)}\log{\int_\mathcal{H}\left(\frac{\d\nu_2}{\d\nu_1}\right)^\rho\d\nu_1},\\
	D_{\am,\alpha}(\nu_2\Vert\nu_1) &\coloneqq\frac{-1}{\alpha(1-\alpha)}\left( \int_{\mathcal{H}}^{}\left( \frac{\d\nu_2}{\d\nu_1} \right)^\alpha\d\nu_1-1 \right),\\
	D_\hel(\nu_2,\nu_1)^2 &\coloneqq \int_{\mathcal{H}}^{}\left( 1-\sqrt{\frac{\d\nu_2}{\d\nu_1}} \right)^2\d\nu_1 = 2-2\int_{\mathcal{H}}^{}\sqrt{\frac{\d\nu_2}{\d\nu_1}}\d\nu_1.
\end{align*}
We refer to $D_\kl(\nu_1\Vert\nu_2)$ as the `reverse KL divergence'.
We do not distinguish between forward R\'enyi divergences $D_{\ren,\rho}(\nu_2\Vert\nu_1)$ and reverse R\'enyi divergences $D_{\ren,\rho}(\nu_1\Vert\nu_2)$, because of the `skew symmetry' of the R\'enyi divergence: $D_{\ren,\rho}(\nu_1\Vert \nu_2) = D_{\ren,1-\rho}(\nu_2\Vert \nu_1)$, c.f.\ \cite[Proposition 2]{vanErven2014}.  
\begin{remark}[Hellinger and Amari divergences]
	\label{rmk:hellinger_and_amari_divergence}
	We note that
	\begin{align}
		\label{eqn:amari_renyi_relation}
		D_{\am,\alpha}(\nu_2\Vert\nu_1) =& \frac{-1}{\alpha(1-\alpha)}\left(\exp(-\alpha(1-\alpha)D_{\ren,\alpha}(\nu_2\Vert\nu_1))-1\right)\\
		\label{eqn:hellinger_renyi_relation}
		D_\hel(\nu_2,\nu_1)^2=&-2\left(1-\exp\left(\frac{1}{4}D_{\ren,1/2}(\nu_2\Vert\nu_1)\right)\right),
	\end{align}
	where \eqref{eqn:hellinger_renyi_relation} follows by \cite[eqs.\ (134)--(135)]{minh_regularized_2021}. It is then straightforward to show, c.f.\ \cite[Remarks 3.10 and 3.11]{PartI} that minimising the Amari-$\alpha$ divergence over $\nu_1$ is equivalent to minimising the $\alpha$-R\'enyi divergence over $\nu_1$. Furthermore, minimising the Hellinger distance over $\nu_1$ is equivalent to minimising the $\frac{1}{2}$-R\'enyi divergence over $\nu_1$. The divergence $\frac{1}{4}D_{\ren,\frac{1}{2}}$ is also known as the Bhattacharyya distance, and is a metric.
\end{remark}

If a divergence between Gaussian measures $\nu_1$ and $\nu_2$ requires access to the density $\frac{\d\nu_2}{\d\nu_1}$, then $\nu_1$ and $\nu_2$ must be equivalent. This is shown by the Feldman--Hajek theorem below. The Feldman--Hajek theorem also characterises which Gaussian measures are equivalent in terms of their means and covariance. For statistical inference, it is often important that the posterior has a density with respect to the prior. This further motivates the need to construct approximate posterior measures that are equivalent to $\mu_\pos$ and $\mu_\pr$. 


\begin{theorem}[Feldman--Hajek]
	\label{thm:feldman--hajek}
	Let $\H$ be a Hilbert space and $\mu=\mathcal{N}(m_1,\mathcal{C}_1)$ and $\nu=\mathcal{N}(m_2,\mathcal{C}_2)$ be Gaussian measures on $\H$.
	Then $\mu$ and $\nu$ are either singular or equivalent, and $\mu$ and $\nu$ are equivalent if and only if the following conditions hold:
	\begin{enumerate}
		\item 
			\label{item:fh_ranges}
			$\ran{\mathcal{C}_1^{1/2}} = \ran{\mathcal{C}_2^{1/2}}$,
		\item 
			\label{item:fh_means}
			$m_1-m_2\in\ran{\mathcal{C}_1^{1/2}}$ , and
		\item 
			\label{item:fh_covariance}
			$(\mathcal{C}_1^{-1/2}\mathcal{C}_2^{1/2})(\mathcal{C}_1^{-1/2}\mathcal{C}_2^{1/2})^*-I\in L_2(\H)$.
	\end{enumerate}
\end{theorem}
For a proof, see e.g.\ \cite[Corollary 6.4.11]{bogachev_gaussian_1998} or \cite[Theorem 2.25]{da_prato_stochastic_2014}.
For injective covariances $\mathcal{C}_1$ and $\mathcal{C}_2$ such that \cref{item:fh_ranges,item:fh_covariance} in \Cref{thm:feldman--hajek} hold, we define
\begin{align}
	\label{eqn:feldman_hajek_operator}
	R(\mathcal{C}_2\Vert \mathcal{C}_1) \coloneqq \mathcal{C}_1^{-1/2}\mathcal{C}_2^{1/2}(\mathcal{C}_1^{-1/2}\mathcal{C}_2^{1/2})^* - I.
\end{align}
Note that two Gaussian measures $\mathcal{N}(m,\mathcal{C}_1)$ and $\mathcal{N}(m,\mathcal{C}_2)$ are equal if $R(\mathcal{C}_2\Vert\mathcal{C}_1)=0$. On the other hand, if these measures are mutually singular, then $R(\mathcal{C}_2\Vert\mathcal{C}_1)$ does not have a square-summable eigenvalue sequence. If the eigenvalues are square-summable, then a faster decay implies the Gaussian measures are more similar. Hence, $R(\cdot\Vert\cdot)$ describes the amount of similarity between Gaussian measures with the same means.

%

If $\nu_1$ and $\nu_2$ are Gaussian measures, then the above divergences can be expressed explicitly in terms of the means and covariances of $\nu_1$ and $\nu_2$ using $R(\cdot\Vert\cdot)$ defined in \eqref{eqn:feldman_hajek_operator}. These formulations rely on a generalisation of the determinant to infinite-dimensional Hilbert spaces. For $A\in L_1(\H)$, the so-called `Fredholm determinant' $\det(I+A)$ can be defined, and if only $A\in L_2(\H)$, then the notion of `Hilbert--Carleman determinant' $\det_2(I+A)$ can be used. The Fredholm and Hilbert--Carleman determinants are defined on respectively trace-class and Hilbert--Schmidt perturbations of the identity. In finite dimensions, every operator is a trace-class and Hilbert--Schmidt perturbation of the identity, and hence these generalised determinants are defined everywhere in this case. In fact, the Fredholm determinant agrees with the standard determinant in this case. We refer to \cite[Theorem 3.2, Theorem 6.2]{simon_notes_1977} or \cite[Lemma 3.3, Theorem 9.2]{Simon2005} for details.

The following result holds when $\H$ is a separable Hilbert space of finite or infinite dimension. The proof is a direct application of \cite[Theorems 14 and 15]{minh_regularized_2021}.

\begin{theorem}[{\cite[Theorem 3.8]{PartI}}]
	\label{thm:gaussian_divergences}
	Let $m_1,m_2\in\mathcal{H}$ and $\mathcal{C}_1,\mathcal{C}_2\in L_2(\mathcal{H})_\R$ be positive. If $\mathcal{N}(m_1,\mathcal{C}_1)\sim\mathcal{N}(m_2,\mathcal{C}_2)$, then
	\begin{subequations}
		\label{eqn:divergences}
		\begin{align}
			\label{eqn:kullback_leibler_divergence}
			D_{\kl}(\mathcal{N}(m_2,\mathcal{C}_2)\Vert \mathcal{N}(m_1,\mathcal{C}_1)) \coloneqq &\frac{1}{2}\Norm{\mathcal{C}_1^{-1/2}(m_2-m_1)}^2-\frac{1}{2}\log\det_2(I+R(\mathcal{C}_2\Vert\mathcal{C}_1)),\\
			\label{eqn:renyi_divergence_order_rho}
			\begin{split}
				D_{\ren,\rho}(\mathcal{N}(m_2,\mathcal{C}_2)\Vert \mathcal{N}(m_1,\mathcal{C}_1)) \coloneqq &\frac{1}{2}\Norm{\bigr(\rho I+(1-\rho)(I+R(\mathcal{C}_2\Vert\mathcal{C}_1))\bigr)^{-1/2}\mathcal{C}_1^{-1/2}(m_2-m_1)}^2\\
				&+\frac{\log\det\left[\bigl(I+R(\mathcal{C}_2\Vert\mathcal{C}_1)\bigr)^{\rho-1}\bigl(\rho I+(1-\rho)(I+R(\mathcal{C}_2\Vert\mathcal{C}_1))\bigr)\right]}{2\rho(1-\rho)}.
			\end{split}
		\end{align}
	\end{subequations}
	Furthermore, 
	\begin{align*}
		\lim_{\rho\rightarrow 1} D_{\ren,\rho}(\mathcal{N}(m_2,\mathcal{C}_2)\Vert \mathcal{N}(m_1,\mathcal{C}_1))= D_{\kl}(\mathcal{N}(m_2,\mathcal{C}_2)\Vert \mathcal{N}(m_1,\mathcal{C}_1)),\\
		\lim_{\rho\rightarrow 0} D_{\ren,\rho}(\mathcal{N}(m_2,\mathcal{C}_2)\Vert \mathcal{N}(m_1,\mathcal{C}_1))= D_{\kl}(\mathcal{N}(m_1,\mathcal{C}_1)\Vert \mathcal{N}(m_2,\mathcal{C}_2)).
	\end{align*}
\end{theorem}

The prior and posterior distributions in \eqref{eqn:observation_model} are equivalent, for every realisation $y$ in a set of probability 1, c.f.\ \cite[Theorem 6.31]{Stuart2010}. The Hessian $H$ defined in \eqref{eqn:hessian} has rank $n$, hence the posterior precision is a finite-rank update of the prior by \eqref{eqn:pos_precision}. Using the operators $R(\mathcal{C}_\pr\Vert \mathcal{C}_\pos)$ and $R(\mathcal{C}_\pos\Vert\mathcal{C}_\pr)$ and \Cref{thm:feldman--hajek}, we can obtain the following relations between the prior-preconditioned Hessian $\mathcal{C}_\pr^{1/2}H\mathcal{C}_\pr^{1/2}$ in \eqref{eqn:prior_preconditioned_Hessian}, the posterior-preconditioned Hessian in \eqref{eqn:posterior_preconditioned_Hessian}, and the `pencil' defined by the prior and the posterior covariance in \eqref{eqn:bayesian_cov_pencil}. The prior-preconditioned Hessian combines prior covariance information with information contained in the Hessian, i.e.\ information on the forward map, noise covariance, and data dimension. Recall the notation $v\otimes w$ for $u,w\in\H$ from \Cref{sec:notation}.

\begin{proposition}[{\cite[Proposition 3.7]{PartI}}]
	\label{prop:bayesian_feldman_hajek}
	There exists a nondecreasing sequence $(\lambda_i)_i\in\ell^2( (-1,0] )$ consisting of exactly $\rank{H}$ nonzero elements and ONBs $(w_i)_i$ and $(v_i)_i$ of $\mathcal{H}$ such that $w_i,v_i\in\ran{\mathcal{C}_\pr^{1/2}}$ and $v_i = \sqrt{1+\lambda_i}\mathcal{C}_\pos^{-1/2}\mathcal{C}_\pr^{1/2}w_i$ for every $i\in \N$, and
	\begin{subequations}
		\begin{align}
			R(\mathcal{C}_\pos\Vert \mathcal{C}_\pr) &= \sum_{i}^{}\lambda_i w_i\otimes w_i, \nonumber \\
			\label{eqn:prior_preconditioned_Hessian}
			\mathcal{C}_{\pr}^{1/2}H\mathcal{C}_{\pr}^{1/2}
			&= (\mathcal{C}_\pos^{-1/2}\mathcal{C}_\pr^{1/2})^* (\mathcal{C}_\pos^{-1/2}\mathcal{C}_\pr^{1/2}) - I
			=\sum_{i}\frac{-\lambda_i}{1+\lambda_i}w_i\otimes w_i, \\
			\label{eqn:posterior_preconditioned_Hessian}
			\mathcal{C}_{\pos}^{1/2}H\mathcal{C}_{\pos}^{1/2}	
			&= I-(\mathcal{C}_\pr^{-1/2}\mathcal{C}_\pos^{1/2})^* (\mathcal{C}_\pr^{-1/2}\mathcal{C}_\pos^{1/2})
			=\sum_{i}(-\lambda_i) v_i\otimes v_i,\\
			\label{eqn:bayesian_cov_pencil}
			\mathcal{C}_\pos^{1/2}\mathcal{C}_\pr^{-1/2}w_i 
			&= (1+\lambda_i)\mathcal{C}_\pos^{-1/2}\mathcal{C}_\pr^{1/2}w_i,\quad \forall i\in\N.
		\end{align}
	\end{subequations}
\end{proposition}

\begin{remark}
	\label{rmk:spectral_comparison}
	We note that the eigenvalues $(\frac{-\lambda_i}{1+\lambda_i})_i$ of \eqref{eqn:prior_preconditioned_Hessian} relate to the eigenvalues $(\delta_i^2)_i$ of \cite[eq.\ (2.8)]{Spantini2015} via the transformation $\lambda_i=\eta(\delta_i^2)$, $\delta_i^2=\eta(\lambda_i)$ with $\eta(x)=\frac{-x}{1+x}$ for $x\in(-1,\infty)$.
\end{remark}

From \Cref{prop:bayesian_feldman_hajek}, the following interpretation of the eigenpairs $(\lambda_i,w_i)_i$ of \Cref{prop:bayesian_feldman_hajek} follows. The proof can be found in \Cref{subsec:proofs_for_formulation}.

\begin{restatable}{proposition}{varianceReduction}
	\label{prop:variance_reduction}
	Let $(\lambda_i,w_i)_i$ be as in \Cref{prop:bayesian_feldman_hajek}. It holds that
	\begin{align}
		\label{eqn:relative_variance_reduction}
		\frac{\var_{X\sim\mu_\pos}(\langle X,\mathcal{C}_\pr^{-1/2}w_i\rangle)}{\var_{X\sim\mu_\pr}(\langle X,\mathcal{C}_\pr^{-1/2}w_i\rangle)}
		=1+\lambda_i=\frac{1}{1+\tfrac{-\lambda_i}{1+\lambda_i}},\quad \forall i\in\N,
	\end{align}
	and for any subspace $V_{r}\subset\ran{\mathcal{C}_\pr^{1/2}}$ of dimension $r\in\N$,
	\begin{align}
		\label{eqn:maximal_variance_reduction}
		\min_{z\in (\mathcal{C}_\pr^{-1/2}V_{r})^\perp\setminus\{0\}} \frac{\var_{X\sim\mu_\pos}(\langle X,z\rangle)}{\var_{X\sim\mu_\pr}(\langle X,z\rangle)}
		=
		\inf_{z\in (V_{r}^\perp\cap{\ran{\mathcal{C}_\pr^{1/2}}})\setminus\{0\}} \frac{\var_{X\sim\mu_\pos}(\langle X,\mathcal{C}_\pr^{-1/2}z\rangle)}{\var_{X\sim\mu_\pr}(\langle X,\mathcal{C}_\pr^{-1/2}z\rangle)}
		\leq 1+\lambda_{r+1},
	\end{align}
	with equality for $V_{r}=\Span{w_1,\ldots,w_{r}}$.
\end{restatable}

Note that while the ratios in \eqref{eqn:relative_variance_reduction} and \eqref{eqn:maximal_variance_reduction} depend on the posterior distribution, they only do so via the posterior covariance. Thus they are independent of the realisation of the data $y$, and only depend on the inverse problem via the choice of prior and the model structure \eqref{eqn:observation_model}.

The significance of \eqref{eqn:relative_variance_reduction} is that the posterior variance along the span of $\mathcal{C}_\pr^{-1/2}w_i$ is smaller than the prior variance along the same subspace by a factor of $(1+\tfrac{-\lambda_i}{1+\lambda_i})^{-1}$, for $i\in\N$. This was observed in the finite-dimensional case in \cite[eq.\ (3.4)]{Spantini2015}.
Thus, \Cref{prop:bayesian_feldman_hajek} implies that finite-dimensional data can only inform finitely many directions in parameter space, in the sense that posterior variance is reduced relative to prior variance only over a finite-dimensional subspace. 
The directions $(\mathcal{C}_\pr^{-1/2}w_i)_{i\leq\rank{H}}$ are orthogonal with respect to the $\mathcal{C}_\pr$-weighted inner product $\langle h_1,h_2\rangle_{\mathcal{C}_\pr}\coloneqq \langle \mathcal{C}_\pr h_1,h_2\rangle$, and not the unweighted inner product of $\H$. 

The equation \eqref{eqn:maximal_variance_reduction} can be interpreted as follows. Given an $r$-dimensional subspace $V_r\subset\ran{\mathcal{C}_\pr^{1/2}}$, the minimum in \eqref{eqn:maximal_variance_reduction} describes the maximal relative variance reduction that occurs among the directions of $\H$ orthogonal to $\mathcal{C}_\pr^{-1/2}V_r$. The inequality in \eqref{eqn:maximal_variance_reduction} implies this maximal relative variance reduction is by at least a factor of $1+\lambda_{r+1}$. If $V_{r}=\Span{w_1,\ldots,w_{r}}$, then this maximal relative variance reduction is by exactly a factor of $1+\lambda_{r+1}$. This shows that the largest relative variance reduction, among all directions in $\H$ orthogonal to $(\mathcal{C}_\pr^{-1/2}V_r)^\perp$, is as small as possible for the choice $V_r=\Span{w_1,\ldots,w_{r}}$, and hence the linearly-independent directions in 
\begin{align}
	\label{eqn:subspace_of_maximal_variance_reduction}
	W_r\coloneqq \Span{\mathcal{C}_\pr^{-1/2}w_1,\ldots,\mathcal{C}_\pr^{-1/2}w_{r}}
\end{align}
are subject to the largest relative variance reduction possible. Since $\mathcal{C}_\pr^{1/2}$ is injective, we thus conclude the following: among all $r$-dimensional subspaces of $\H$, it is the $r$-dimensional subspace $W_r$ that contains those $r$ linearly-independent directions in which the relative variance reduction is largest. This generalises the conclusion of \cite[Section 3.1]{Spantini2015} to infinite dimensions.

Recall from \Cref{sec:notation} the definition of the weighted inner product $\norm{\cdot}_{\mathcal{C}_\pr}$. The sequence $(\mathcal{C}_\pr^{-1/2}w_i)_i$ forms an ONB of $(\H,\norm{\cdot}_{\mathcal{C}_\pr})$. Indeed, $\langle \mathcal{C}_\pr^{-1/2}w_i,\mathcal{C}_\pr^{-1/2}w_j\rangle_{\mathcal{C}_\pr}=\langle w_i,w_i\rangle=\delta_{ij}$ and if $\langle h,\mathcal{C}_\pr^{-1/2}w_i\rangle_{\mathcal{C}_\pr}=0$ for all $i$, then $\mathcal{C}_\pr^{1/2}h=0$ and hence $h=0$ by injectivity of $\mathcal{C}_\pr$. Let 
\begin{align}
	\label{eqn:complement_subspace_of_maximal_variance}
	W_{-r}\coloneqq\overline{\Span{\mathcal{C}_\pr^{-1/2}w_i,\ i>r}},
\end{align}
where the closure is taken with respect to the $\H$-norm. Since $\langle\mathcal{C}_\pr^{-1/2} w_i,\mathcal{C}_\pr^{-1/2} w_j\rangle_{\mathcal{C}_\pr}=0$ for all $i\leq r<j$, it holds by linearity that $\langle h,k\rangle_{\mathcal{C}_\pr}=0$ for all $h\in W_r$ and $k\in\Span{\mathcal{C}_\pr^{-1/2}w_j,\ j>r}$. If $h\in W_r$ and if $(k_n)_n\subset \Span{\mathcal{C}_{\pr}^{-1/2}w_j,\ j>r}$ is a sequence converging to some $k\in W_{-r}$, then $\langle h,k\rangle_{\mathcal{C}_\pr}=\langle \mathcal{C}_\pr h,k\rangle = \lim_n\langle \mathcal{C}_\pr h,k_n\rangle =\lim_n\langle h,k_n\rangle_{\mathcal{C}_\pr} =0$. Hence, in the $\norm{\cdot}_{\mathcal{C}_\pr}$-norm we have the orthogonal decomposition $\H=W_r\oplus W_{-r}$ into the subspace of maximal relative variance reduction $W_r$ in \eqref{eqn:subspace_of_maximal_variance_reduction} and $W_{-r}$. Thus, the direct sum $\H=W_r + W_{-r}$ holds, but this decomposition is not orthogonal in general in the $\H$-inner product.

If, for some $r<\rank{H}$, there exists an $r$-dimensional subspace $W_r$ given by \eqref{eqn:subspace_of_maximal_variance_reduction} such that the variance reduction on the complement of this subspace is sufficiently small, then the subspace $\Span{\mathcal{C}_\pr^{1/2}w_1,\ldots,\mathcal{C}_\pr^{1/2}w_r}=\mathcal{C}_\pr(W_r)$ is also called the `likelihood-informed subspace' in literature, see e.g.\ \cite{Cui2014,CuiTong2022,Cui2022}.

\section{Optimal approximation of the covariance}
\label{sec:optimal_approximation_covariance}

This section discusses low-rank posterior covariance approximation, using \cite[Theorem 4.21]{PartI}. This approximation serves as a basis for the joint mean and covariance approximation discussed in \Cref{sec:optimal_joint_approximation}.

We aim to approximate the posterior distribution by approximating the posterior covariance and keeping the posterior mean fixed. The reverse KL divergence between such approximate posterior distributions and the exact posterior is used as a loss function on the set of approximate covariances. This set of candidates for covariance approximation is chosen as
\begin{align}
	\label{eqn:class_approx_covariance}
	\mathscr{C}_r\coloneqq\left\{ \mathcal{C}_{\pr}-KK^*>0:\ K\in\B(\R^r,\H),\ \ran{K}\subset\ran{\mathcal{C}_\pr}\right\},\quad r\in\N.
\end{align}
Since $\mathcal{C}_\pr-KK^*\in\mathscr{C}_r$ is positive and self-adjoint, it is an injective covariance operator. Furthermore, it is stated in \cite[Corollary 4.9]{PartI} that for every $\mathcal{C}\in\mathscr{C}_r$ it holds that $\mathcal{N}(m_\pos(y),\mathcal{C})$ is equivalent to the exact posterior. Since $\mathcal{C}_\pos$ does not depend on $y$, this equivalence holds for all $y$ simultaneously. This equivalence holds because of the range condition $\ran{K}\subset\ran{\mathcal{C}_\pr}$. Furthermore, the assumption $K\in\B(\R^r,\H)$ implies the rank restriction $\rank{K}\leq r$. Thus, for $r$ small compared to $n$, $\mathcal{C}_\pr-KK^*$ can be interpreted as a low-rank update of $\mathcal{C}_\pr.$ Therefore, the class $\mathscr{C}_r$ provides an extension to infinite dimensions of the finite-dimensional updates considered in \cite{Spantini2015}.

The low-rank posterior covariance problem is thus as follows.

\begin{problem}[Rank-$r$ nonpositive covariance updates] 
	\label{prob:optimal_covariance}
	Find $\mathcal{C}^\opt_r\in\mathscr{C}_r$ such that for all data $y$ in a set of probability 1,
	\begin{align*}
		D_\kl\left(\mathcal{N}(m_\pos(y),\mathcal{C}^\opt_r)\Vert\mathcal{N}(m_\pos(y),\mathcal{C}_\pos)\right)=\min\{D_\kl\left(\mathcal{N}(m_\pos(y),\mathcal{C})\Vert\mathcal{N}(m_\pos(y),\mathcal{C}_\pos)\right):\ \mathcal{C}\in\mathscr{C}_r\}.
	\end{align*}
\end{problem}

The KL divergences in \Cref{prob:optimal_covariance} are finite, because for $\mathcal{C}\in\mathscr{C}_r$ the equivalence $\mathcal{N}(m_\pos(y),\mathcal{C})\sim\mu_\pos(y)$ holds for all $y$ in a set of probability 1 by construction of $\mathscr{C}_r$, as discussed after \eqref{eqn:class_approx_covariance}.

The following theorem provides the solution to \Cref{prob:optimal_covariance}, which follows directly from \cite[Lemma 4.2(iii)]{PartI} and from \cite[Theorem 4.21]{PartI} applied with $f(x)\leftarrow f_\kl(\tfrac{-x}{1+x})$ , where
\begin{align}
	\label{eqn:def_f_kl}
	f_\kl:(-1,\infty)\rightarrow \R_{\geq 0},\quad f_\kl(x) = \frac{1}{2}(x-\log(1+x)).
\end{align}

\begin{restatable}{theorem}{optCovarianceAndPrecision}
	Let $r\leq n$ and let $(\lambda_i)_i\in\ell^2((-1,0])$ and $(w_i)_i\subset \ran{\mathcal{C}_\pr^{1/2}}$ be as given in \Cref{prop:bayesian_feldman_hajek}. Define
	\begin{align}
		\label{eqn:optimal_covariance}
		\mathcal{C}^\opt_r &\coloneqq \mathcal{C}_{\pr} - \sum_{i=1}^{r}-\lambda_i(\mathcal{C}_{\pr}^{1/2}w_i)\otimes(\mathcal{C}_{\pr}^{1/2}w_i).
	\end{align}
	Then $\mathcal{C}^\opt_r$ solves \Cref{prob:optimal_covariance}, $\dom{(\mathcal{C}^\opt_r)^{-1}}=\ran{\mathcal{C}_\pr}$ and $(\mathcal{C}_r^{\opt})^{-1}=\mathcal{C}_\pr - \sum_{i=1}^{r}(\mathcal{C}_\pr^{-1/2} w_i)\otimes (\mathcal{C}_\pr^{-1/2}w_i)$. Furthermore, the associated minimal loss is $\sum_{i>r}^{}f_\kl(\lambda_i)$, where $f_\kl$ is defined in \eqref{eqn:def_f_kl}. The solution $\mathcal{C}^\opt_r$ is unique if and only if the following holds: $\lambda_{r+1}=0$ or $\lambda_r<\lambda_{r+1}$.
	\label{thm:opt_covariance_and_precision}
\end{restatable}

The formulation of \Cref{thm:opt_covariance_and_precision} is a special case of \cite[Theorem 4.21]{PartI}, and this special case will suffice for the subsequent developments in this work. However, we note that the results of \cite[Theorem 4.21 and Corollary 4.23]{PartI} are more general than presented in \Cref{thm:opt_covariance_and_precision}. They state that $\mathcal{C}^\opt_r$ is not only the optimal low-rank approximation of $\mathcal{C}_\pos$ for the reverse KL divergence, but simultaneously also for all divergences in a more general class of divergences, including the forward KL divergence, the Hellinger distance, the R\'enyi divergences and the Amari $\alpha$-divergences for $\alpha\in(0,1)$.

\begin{remark}(Interpretation of $\mathcal{C}^\opt_r$)
		\label{rmk:interpretation_covariance_approx}
		Because $\mathcal{C}_\pr G^*(\mathcal{C}_\obs+ G\mathcal{C}_\pr G^*)^{-1/2}\in\B_{00,n}(\R^n,\H)$ maps into $\ran{\mathcal{C}_\pr}$, it holds that $\mathcal{C}_\pos\in\mathscr{C}_n$ by \eqref{eqn:pos_covariance} and the definition of $\mathscr{C}_n$ in \eqref{eqn:class_approx_covariance}. Thus, $\mathcal{C}^\opt_n= \mathcal{C}_\pos$. Taking $r\leftarrow n$ in \Cref{thm:opt_covariance_and_precision}, we then see that $\mathcal{C}_\pos = \mathcal{C}_\pr - \sum_{i=1}^{n}(-\lambda_i)(\mathcal{C}_\pr^{1/2}w_i)\otimes (\mathcal{C}_\pr^{1/2}w_i)$. Let $r\leq n$ be fixed. For $j\leq r$, we have that $\mathcal{C}^\opt_r \mathcal{C}_\pr^{-1/2}w_j = \mathcal{C}_\pr^{1/2}w_j+\lambda_j\mathcal{C}_\pr^{1/2}w_j = \mathcal{C}_\pos \mathcal{C}_\pr^{-1/2}w_j$.  With $W_r$ as defined in \eqref{eqn:subspace_of_maximal_variance_reduction}, we thus see that $\mathcal{C}^\opt_r = \mathcal{C}_\pos$ on $W_r$. Furthermore, for $j>r$, we have $\mathcal{C}^\opt_r \mathcal{C}_\pr^{-1/2}w_j = \mathcal{C}_\pr \mathcal{C}_\pr^{-1/2}w_j$. It then holds that $\mathcal{C}^{\opt}_r=\mathcal{C}_\pr$ on the dense subspace $\Span{\mathcal{C}_\pr^{-1/2}w_j,\ j>r}$ of $W_{-r}$ defined in \eqref{eqn:complement_subspace_of_maximal_variance}. Since $\mathcal{C}^\opt_r$ and $\mathcal{C}_\pr$ are both continuous, it then holds that $\mathcal{C}^\opt_r=\mathcal{C}_\pr$ on $W_{-r}$. 
\end{remark}

\section{Optimal approximation of the mean}
\label{sec:optimal_approximation_mean}

In this section, we discuss an optimal low-rank approximation procedure for the posterior mean $m_\pos(y) = \mathcal{C}_\pos G^* \mathcal{C}_\obs^{-1} y$, see \eqref{eqn:pos_mean}. Given the data $y$, the approximations considered are of the form $Ay$, where $A\in\mathscr{M}^{(i)}$ for $i=1$ is a structure-preserving update and for $i=2$ is a structure-ignoring update; see \eqref{eqn:class_approx_mean_with_covariance_update} and \eqref{eqn:class_approx_mean_without_covariance_update} respectively.
Unless otherwise specified, the proofs of the results below are given in \Cref{subsec:proofs_for_optimal_approximation_means}.

We shall assess the approximation quality of an approximate posterior mean by averaging the mean-dependent term for the R\'enyi divergence and the forward and reverse KL divergence over all possible realisations $y$ of $Y$. 
By averaging over $Y$, the optimal operator $A$ will be data-independent, i.e.\ will not depend on a specific realisation $y$ of $Y$. While averaging over $Y$ implies that the resulting posterior mean approximations are not optimal in general for a specific realisation $y$ of $Y$, this approach has the benefit that $A$ can be constructed before observing the data. This leads to an offline-online approach to posterior mean approximation: the preliminary `offline' stage computes one operator, which can then can be applied in the subsequent `online' stage to any realisation of the data. This is in analogy to the finite-dimensional case studied in \cite[Section 4.1]{Spantini2015} and its generalisation to certain nonlinear forward models and to losses with respect to the average Amari $\alpha$-divergences as studied in \cite[Section 5]{Li2024a}. Furthermore, averaging over $Y$ enables us to exploit recent work on reduced-rank operator approximation \cite{CarereLie2024}.

Recall that we use the observation model $Y=GX+\zeta$ for $\zeta\sim\mathcal{N}(0,\mathcal{C}_\obs)$ for $G\in\B(\H,\R^n)$ and positive $\mathcal{C}_\obs\in\B(\R^n)_\R$, and that our prior model is $X\sim\mathcal{N}(0,\mathcal{C}_\pr)$, with $X$ and $\zeta$ independent. These assumptions imply that the marginal distribution of $Y$ is $Y\sim\mathcal{N}(0,\mathcal{C}_\y)$, where 
\begin{align}
	\label{eqn:data_covariance}
	\mathcal{C}_\y\coloneqq G\mathcal{C}_\pr G^*+\mathcal{C}_\obs\in\B(\R^n).
\end{align}

Since $R(\mathcal{C}\Vert\mathcal{C})=0$ for any positive $\mathcal{C}\in L_1(\mathcal{H})_\R$, by \Cref{thm:gaussian_divergences}, the R\'enyi divergences and forward and reverse KL divergence of approximating $\mathcal{N}(m_\pos,\mathcal{C})$ by $\mathcal{N}(m,\mathcal{C})$ for any $m\in\mathcal{H}$ satisfying $m-m_\pos\in \ran{\mathcal{C}^{1/2}}$ is given by, for any $\rho\in(0,1)$,

\begin{align}
	\label{eqn:loss_for_means}
	\begin{split}
		\frac{1}{2}\norm{m-m_\pos}^2_{\mathcal{C}^{-1}} &= D_{\kl}(\mathcal{N}(m_\pos,\mathcal{C})\Vert\mathcal{N}(m,\mathcal{C})) 
		= D_{\ren,\rho}(\mathcal{N}(m_\pos,\mathcal{C})\Vert\mathcal{N}(m,\mathcal{C}))\\
		&= D_{\kl}(\mathcal{N}(m,\mathcal{C})\Vert\mathcal{N}(m_\pos,\mathcal{C})).
	\end{split}
\end{align}

We choose $\mathcal{C}$ to be $\mathcal{C}_\pos$, so that the optimal low-rank posterior mean then is given by the solution to the following problem. Note that the term inside the expectation on the left hand side corresponds to the mean-dependent term in \eqref{eqn:kullback_leibler_divergence}, and has the interpretation that it penalises errors in the approximation of the posterior mean more in those directions in which the posterior covariance is small.

\begin{problem}
	\label{prob:optimal_mean}
	Let $r \leq n$ and $i\in\{1,2\}$. Find $A^{\opt,(i)}_r\in \mathscr{M}_r^{(i)}$ such that 
	\begin{align*}
	\mathbb{E}\left[\norm{A^{\opt,(i)}_r Y-m_\pos(Y)}_{\mathcal{C}_\pos^{-1}}^2\right]= \min\left\{\mathbb{E}\left[\norm{AY-m_\pos(Y)}_{\mathcal{C}_\pos^{-1}}^2\right]:\ A\in\mathscr{M}_r^{(i)}\right\}.
	\end{align*}
\end{problem}

We only consider the case $r\leq n$ since the same problem for $r>n$ has the trivial solution $A^{\opt,(i)}_r=\mathcal{C}_\pos G^*\mathcal{C}_\obs^{-1}$ for $i=1,2$.

\begin{remark}[Comparison with Bayes risk]
	The Bayes risk $\mathcal{R}(A)\coloneqq\mathbb{E}\left[ \Norm{AY - X}_{\mathcal{C}_\pos^{-1}}^2 \right]$ for $A\in\mathscr{M}_r^{(i)}$, $i=1,2$, considered in \cite[Section 4.1]{Spantini2015} is not well-defined, since the event $\{X\in\dom{\mathcal{C}_\pos^{-1/2}}\}$ occurs with probability 0. However, one can show that $\mathcal{R}(A) = \mathbb{E}\left[ \Norm{AY-m_\pos(Y)}_{\mathcal{C}_\pos^{-1}}^2 \right]+\dim{\mathcal{H}}$ if $\dim{\mathcal{H}}<\infty$. Thus, not only does the approximation error \eqref{eqn:loss_for_means} used in \Cref{prob:optimal_mean} have a natural interpretation as the mean-dependent term of the R\'enyi, Amari, forward and reverse KL divergences, it also captures the relevant contribution to the Bayes risk which involves the approximation.
\end{remark}

In our derivation of the optimal $A^{\opt,(i)}_r$, we shall make use of specific non-self adjoint square roots $S_\pos\in L_2(\H)$ and $S_\y\in \B(\R^n)$ of the covariances $\mathcal{C}_\pos$ and $\mathcal{C}_y$ respectively. Since $n<\infty$, $\mathcal{C}_\obs^{-1}$ is bounded and self-adjoint and we can decompose $\mathcal{C}_{\obs}^{-1}=\mathcal{C}_\obs^{-1/2}(\mathcal{C}_{\obs}^{-1/2})^*$ by \Cref{lemma:adjoint_of_densely_defined_operators}. Therefore, by \eqref{eqn:prior_preconditioned_Hessian} in \Cref{prop:bayesian_feldman_hajek},
\begin{align}
	\label{eqn:decomposition_prior_preconditioned_hessian}
	(\mathcal{C}_\pr^{1/2}G^*\mathcal{C}_\obs^{-1/2})(\mathcal{C}_\pr^{1/2}G^*\mathcal{C}_\obs^{-1/2})^* = \mathcal{C}_\pr^{1/2}H\mathcal{C}_\pr^{1/2} = \sum_{i=1}^{n}\frac{-\lambda_i}{1+\lambda_i}w_i\otimes w_i,
\end{align}
with $(w_i)_i$ and $(\lambda_i)_i$ as in \Cref{prop:bayesian_feldman_hajek}. By \Cref{lemma:operator_svd}, we may apply the SVD to $\mathcal{C}_\pr^{1/2}G^*\mathcal{C}_\obs^{-1/2}$, and the singular values are then determined by \eqref{eqn:decomposition_prior_preconditioned_hessian}. That is, there exists an orthonormal sequence $(\varphi_i)_i$ in $\R^n$ such that
\begin{align}
	\label{eqn:svd_preconditioned_hessian}
	\mathcal{C}_\pr^{1/2} G^*\mathcal{C}_\obs^{-1/2} = \sum_{i=1}^{n}\sqrt{\frac{-\lambda_i}{1+\lambda_i}}w_i\otimes \varphi_i.
\end{align}
Using that $\lambda_i=0$ for all $i>n$ by \Cref{prop:bayesian_feldman_hajek}, we now define,
\begin{align}
	\begin{split}
		S_\pos &= \mathcal{C}_\pr^{1/2}\left( I + \sum_{i\in\N}^{}\frac{-\lambda_i}{1+\lambda_i}w_i\otimes w_i \right)^{-1/2}=\mathcal{C}_\pr^{1/2}\left( I + \sum_{i=1}^{n}\frac{-\lambda_i}{1+\lambda_i}w_i\otimes w_i \right)^{-1/2},\\
		S_\y &= \mathcal{C}_\obs^{1/2}\left( I + \sum_{i=1}^{n}\frac{-\lambda_i}{1+\lambda_i} \varphi_i\otimes \varphi_i\right)^{1/2}.
		\label{eqn:square_roots}
	\end{split}
\end{align}
Note that $\sum_{i=1}^{m}(1+\frac{-\lambda_i}{1+\lambda_i})w_i\otimes w_i$ does not converge in $\B(\H)$ as $m\rightarrow \infty$, when $\H$ is infinite-dimensional. Indeed, if $\sum_{i=1}^{m}(1+\frac{-\lambda_i}{1+\lambda_i})w_i\otimes w_i$ converges, then $\sum_{i=1}^{m}(1+\frac{-\lambda_i}{1+\lambda_i})w_i\otimes w_i - \sum_{i=1}^{n}\frac{-\lambda_i}{1+\lambda_i}w_i\otimes w_i$ is a sequence of finite rank operators converging to the identity. Since the identity in $\B(\H)$ is not compact when $\H$ is infinite-dimensional, the series $\sum_{i=1}^{m}(1+\frac{-\lambda_i}{1+\lambda_i})w_i\otimes w_i$ does not converge as $m\rightarrow\infty$. However, there is pointwise convergence: for $h\in \H$, we may compute,
\begin{align*}
	\left( I+\sum_{i}^{}\frac{-\lambda_i}{1+\lambda_i}w_i\otimes w_i \right) h = \sum_{i}^{}\left(1+\frac{-\lambda_i}{1+\lambda_i}\right)\langle h,w_i\rangle w_i = \sum_{i}^{}\frac{1}{1+\lambda_i}\langle h,w_i\rangle w_i.
\end{align*}
Similarly, a direct computation shows that for $h\in\H$ and $x\in\R^n$,
\begin{subequations}
	\label{eqn:square_root_expression}
	\begin{align}
		\label{eqn:square_root_expression_S_pos}
		\left( I+\sum_{i}^{}\frac{-\lambda_i}{1+\lambda_i}w_i\otimes w_i \right)^{-1/2} h &= \sum_{i}^{}(1+\lambda_i)^{1/2}\langle h,w_i\rangle w_i,\\
		\label{eqn:square_root_expression_S_y}
		\left( I+\sum_{i}^{}\frac{-\lambda_i}{1+\lambda_i}\varphi_i\otimes \varphi_i \right)^{1/2} x &= \sum_{i}^{}(1+\lambda_i)^{-1/2}\langle x,\varphi_i\rangle \varphi_i.
	\end{align}
\end{subequations}

We first note that $S_\pos,S_\y$ are indeed square roots, and that they have well-defined inverses. 

\begin{restatable}{lemma}{lemmaSquareRoots}
	\label{lemma:square_roots}
	Let $S_\pos$ and $S_\y$ be as in \eqref{eqn:square_roots}. It holds that 
	\begin{enumerate}
		\item
			\label{item:square_roots}
			$\mathcal{C}_\pos = S_\pos S_\pos^*$ and $\mathcal{C}_\y = S_\y S_\y^*$ and $S_\pos^{-1}:\ran{\mathcal{C}}_\pr^{1/2}\rightarrow\H$ and $S_\y^{-1}\in\B(\R^n)$ exist,
		\item
			\label{item:computation_cm_norm}
			$\norm{h}_{\mathcal{C}_\pos^{-1}}^2 =\norm{S_\pos^{-1}h}^2$ for all $h\in\ran{\mathcal{C}}_\pr^{1/2}=\ran{\mathcal{C}_{\pos}^{1/2}}$,
		\item
			\label{item:root_pre_cameron_martin_space}
			$S_\pos (\ran{\mathcal{C}_\pr^{1/2}})=\ran{\mathcal{C}_\pr}=\ran{\mathcal{C}_\pos}$.
	\end{enumerate}
\end{restatable}

\Cref{item:computation_cm_norm} can be used to evaluate the norms in \Cref{prob:optimal_mean} by replacing $\mathcal{C}_\pos^{-1/2}$ by $S_\pos^{-1}$. 

Let us define,
\begin{align}
	\begin{split}
		\label{eqn:classes_mean_reparametrised}
		\widetilde{\mathscr{M}}^{(1)}_r &\coloneqq \{(S_\pos^{-1}\mathcal{C}_\pr-\widetilde{B})G^*\mathcal{C}_\obs^{-1}:\ \widetilde{B}\in\B_{00,r}(\H)\},\\
		\widetilde{\mathscr{M}}^{(2)}_r &\coloneqq \B_{00,r}(\R^n,\mathcal{H}).
	\end{split}
\end{align}
We now consider the following problem. 
\begin{problem}
	\label{prob:optimal_mean_reformulated}
	Let $r\leq n$ and $i\in\{1,2\}$. Find $\widetilde{A}^{\opt,(i)}_r\in \widetilde{\mathscr{M}}_r^{(i)}$ such that 
	\begin{align*}
		\mathbb{E}\left[\Norm{\widetilde{A}^{\opt,(i)}_r Y-S_\pos^{-1}m_\pos(Y)}^2\right] = \min\left\{\mathbb{E}\left[\Norm{\widetilde{A} Y- S_\pos^{-1} m_\pos(Y)}^2\right]:\ \widetilde{A}\in\widetilde{\mathscr{M}}_r^{(i)}\right\}.
	\end{align*}
\end{problem}

It is shown in \cref{item:mean_approximation_problem_equivalence_3,item:mean_approximation_problem_equivalence_4} of the following result that \Cref{prob:optimal_mean_reformulated} is a reformulation of \Cref{prob:optimal_mean}. Using \Cref{thm:feldman--hajek}, \cref{item:mean_approximation_problem_equivalence_1} of the following result also provides an explicit description of the approximation classes $\mathscr{M}^{(i)}_r$ of \eqref{eqn:class_approx_means} in terms of the ranges of the operators $A$ and $B$, while \cref{item:mean_approximation_problem_equivalence_2} relates these classes to the classes $\widetilde{\mathscr{M}}^{(i)}_r$ from \eqref{eqn:classes_mean_reparametrised}.

\begin{restatable}{proposition}{propMeanApproximationProblemEquivalence}
	\label{prop:mean_approximation_problem_equivalence}
	Let $r\leq n$ and $i=1,2$. Let $S_\pos$ be as defined in \eqref{eqn:square_roots}, let $\mathscr{M}^{(i)}_r$ be as in \eqref{eqn:class_approx_means} and let $\widetilde{\mathscr{M}}^{{(i)}}_r$ be as in \eqref{eqn:classes_mean_reparametrised}. Then,
	\begin{enumerate}
		\item 
			\label{item:mean_approximation_problem_equivalence_1}
			$\mathscr{M}^{(i)}_r$ can equivalently be described by
			\begin{subequations}
				\label{eqn:equivalent_class_approx_means}
				\begin{align}
					\label{eqn:equivalent_class_approx_mean_with_covariance_update}
				\mathscr{M}^{(1)}_r =& \{(\mathcal{C}_\pr-B)G^*\mathcal{C}_\obs^{-1}:\ B\in\B_{00,r}(\H),\ B(\ker{G}^\perp)\subset\ran{\mathcal{C}_\pr^{1/2}}\},
					\\
					\label{eqn:equivalent_class_approx_mean_without_covariance_update}
					\mathscr{M}^{(2)}_r =& \{A\in\B_{00,r}(\R^n,\H):\ \ran{A}\subset\ran{\mathcal{C}_\pr^{1/2}}\},
				\end{align}
			\end{subequations}
		\item
			\label{item:mean_approximation_problem_equivalence_2}
			$\widetilde{\mathscr{M}}^{(i)}_r=S_\pos^{-1}\mathscr{M}^{(i)}_r$,
		\item
			\label{item:mean_approximation_problem_equivalence_3}
			$S_\pos\widetilde{A}^{\opt,(i)}_r$ solves \Cref{prob:optimal_mean} if and only if $\widetilde{A}^{\opt,(i)}_r$ solves \Cref{prob:optimal_mean_reformulated}.
		\item
			\label{item:mean_approximation_problem_equivalence_4}
			$A^{\opt,(i)}_r$ solves \Cref{prob:optimal_mean} if and only if $S_\pos^{-1}A^{\opt,(i)}_r$ solves \Cref{prob:optimal_mean_reformulated}.
	\end{enumerate}
\end{restatable}

The following lemma shows that the mean square error terms in \Cref{prob:optimal_mean_reformulated} can be computed by evaluating a Hilbert--Schmidt norm of an operator involving the non-self adjoint square root \eqref{eqn:svd_preconditioned_hessian} of the prior-preconditioned Hessian \eqref{eqn:decomposition_prior_preconditioned_hessian}.

\begin{restatable}{lemma}{lemmaMeanEquivalentLoss}
	\label{lemma:mean_equivalent_loss}
	It holds that
	\begin{align}
			\mathbb{E}\left[\Norm{\widetilde{A}Y-S_\pos^{-1}m_\pos(Y)}^2\right] &= \Norm{\widetilde{A}S_\y - \mathcal{C}_\pr^{1/2}G^*\mathcal{C}_\obs^{-1/2}}_{L_2(\mathcal{H})}^2,\quad \widetilde{A}\in\mathcal{B}(\R^n,\mathcal{H}).
		\label{eqn:mean_equivalent_loss}
	\end{align}
\end{restatable}

In order to solve \Cref{prob:optimal_mean_reformulated}, we use a result on reduced-rank operator approximation in $L_2(\mathcal{H})$ norm, proven in \cite{CarereLie2024}. It is a generalised version of the Eckart--Young theorem. Recall that compact operators, in particular Hilbert--Schmidt operators and finite-rank operators, have an SVD, c.f.\ \Cref{lemma:operator_svd}. Also recall the definition of the Moore--Penrose inverse $C^\dagger$ of $C\in\B(\mathcal{H})$ from \Cref{sec:notation}. If $C$ has closed range, then $C^\dagger$ is bounded, c.f.\ \cite[Proposition 2.4]{engl_regularization_1996}. The following is an application of \cite[Theorem 3.2]{CarereLie2024} to the case where the operators $B$ and $C$ occurring in the theorem have closed range. Note that when $T=I$ and $S=I$, we recover the Eckart--Young theorem.

\begin{theorem}[{\cite[Theorem 3.2, Remark 3.5]{CarereLie2024}}]
	\label{thm:generalised_friedland_torokhti}
Let $\mathcal{H}_1,\mathcal{H}_2,\mathcal{H}_3,\mathcal{H}_4$ be Hilbert spaces and let $T\in\B(\mathcal{H}_3,\mathcal{H}_4)$, $S\in\B(\mathcal{H}_1,\mathcal{H}_2)$ both have closed range and let $M\in L_2(\mathcal{H}_1,\mathcal{H}_4)$. Suppose $P_{\ran{T}}MP_{\ker{S}^\perp}$ has nonincreasing singular value sequence $(\sigma_i)_i\in\ell^2([0,\infty))$. Then, for each rank-$r$ truncated SVD $(P_{\ran{T}} M P_{\ker{S}^\perp})_r$ of $P_{\ran{T}} M P_{\ker{S}^\perp}$,
	\begin{align}
		\label{eqn:hat_N}
		\widehat{N} \coloneqq T^\dagger (P_{\ran{T}}MP_{\ker{S}^\perp})_rS^\dagger,
	\end{align}
	is a solution to the problem,
	\begin{align}
		\label{eqn:hilbert_schmidt_approximation_problem}
		\min\{\norm{M-TNS}_{L_2(\mathcal{H}_1,\mathcal{H}_4)},\ N\in\B_{00,r}(\mathcal{H}_2,\mathcal{H}_3)\},
	\end{align}
	such that 
	\begin{align}
		\label{eqn:minimality_property}
		N = P_{\ker{T}^\perp}NP_{\ran{S}}. 
	\end{align}
	Furthermore, \eqref{eqn:hat_N} is the only solution of \eqref{eqn:hilbert_schmidt_approximation_problem} satisfying \eqref{eqn:minimality_property} if and only if the following holds: $\sigma_{r+1}=0$ or $\sigma_r > \sigma_{r+1}.$
\end{theorem}

\begin{remark}[Uniqueness and minimality]
	Even when the uniqueness condition of \Cref{thm:generalised_friedland_torokhti} holds, there are in general infinitely many solutions to \eqref{eqn:hilbert_schmidt_approximation_problem}. For example, if $\ran{S}^\perp\not=\{0\}$, then one can modify $N$ on $\ran{S}^\perp$ without changing the operator $TNS$. The condition \eqref{eqn:minimality_property} ensures that a unique solution of \eqref{eqn:hilbert_schmidt_approximation_problem} can be obtained. Furthermore, \eqref{eqn:minimality_property} also has a natural interpretation as giving minimal solutions of \eqref{eqn:hilbert_schmidt_approximation_problem}. Indeed, any $N\in L_2(\mathcal{H}_2,\mathcal{H}_3)$ satisfies
	\begin{align*}
		N &= P_{\ker{T}^\perp} N P_{\ran{S}} + P_{\ker{T}} N P_{\ran{S}} + P_{\ker{T}^\perp} N P_{\ran{S}^\perp} + P_{\ker{T}} N P_{\ran{S}^\perp}.
	\end{align*}
	By orthogonality of $\ker{T}$ and $\ker{T}^\perp$ and of $\ran{S}$ and $\ran{S}^\perp$, this implies that $N\in L_2(\mathcal{H}_2,\mathcal{H}_3)$ satisfies \eqref{eqn:minimality_property} if and only if the terms $P_{\ker{T}} N P_{\ran{S}}$, $P_{\ker{T}^\perp} N P_{\ran{S}^\perp}$, $P_{\ker{T}} N P_{\ran{S}^\perp}$ are all zero. Taking the $L_2(\mathcal{H}_2,\mathcal{H}_3)$ norm,
	\begin{align*}
		\norm{N}_{L_2(\mathcal{H}_2,\mathcal{H}_3)}^2 
		= &\norm{P_{\ker{T}^\perp}NP_{\ran{S}}}_{L_2(\mathcal{H}_2,\mathcal{H}_3)}^2
		+ \norm{P_{\ker{T}}NP_{\ran{S}}}_{L_2(\mathcal{H}_2,\mathcal{H}_3)}^2\\
		&+ \norm{P_{\ker{T}^\perp}NP_{\ran{S}^\perp}}_{L_2(\mathcal{H}_2,\mathcal{H}_3)}^2
		+ \norm{P_{\ker{T}}NP_{\ran{S}^\perp}}_{L_2(\mathcal{H}_2,\mathcal{H}_3)}^2,
	\end{align*}
	which shows that $\norm{N}_{L_2(\mathcal{H}_2,\mathcal{H}_3)}^2\geq \norm{P_{\ker{T}^\perp}NP_{\ran{S}}}_{L_2(\mathcal{H}_1,\mathcal{H}_4)}^2$, with equality if and only if \eqref{eqn:minimality_property} holds. Thus, \eqref{eqn:minimality_property} can be interpreted as a minimality condition on $N$.
	To see that the equality in the display above holds, note that $\langle P_{\ker{T}}Ch, P_{\ker{T^\perp}} Ch\rangle=0$ and $\langle P_{\ran{S}}C^*k,P_{\ran{S}^\perp}C^*k\rangle=0$ for any $h\in\mathcal{H}_2$, $k\in\mathcal{H}_3$ and $C\in\B(\mathcal{H}_2,\mathcal{H}_3)$. Thus, in $L_2(\mathcal{H}_2,\mathcal{H}_3)$, the operators $P_{\ker{T}}C$ and $P_{\ker{T}^\perp}C$ are orthogonal, and the operators $P_{\ran{S}}C^*$ and $P_{\ran{S}^\perp}C^*$ are orthogonal. By the fact that $\langle A,B\rangle_{L_2(\mathcal{H}_2,\mathcal{H}_3)}=\langle B^*,A^*\rangle_{L_2(\mathcal{H}_3,\mathcal{H}_2)}$ for any $A,B\in L_2(\mathcal{H}_2,\mathcal{H}_3)$, we see that $CP_{\ran{S}}$ and $CP_{\ran{S}^\perp}$ are orthogonal for any $C\in L_2(\mathcal{H}_2,\mathcal{H}_3)$. Therefore, the cross terms in the above expansion of $\norm{N}_{L_2(\mathcal{H}_2,\mathcal{H}_3)}^2$ all vanish.
\end{remark}

\begin{remark}[Equivalent uniqueness statement]
	\label{rmk:equivalent_uniqueness_statement}
	An equivalent formulation of the uniqueness statement of \Cref{thm:generalised_friedland_torokhti} is as follows: $TN_1S=TN_2S$ for any two solutions $N_1$ and $N_2$ of \eqref{eqn:hilbert_schmidt_approximation_problem} if and only if either $\sigma_{r+1}=0$ or $\sigma_r > \sigma_{r+1}$. To see this, we need to show that the solution of \eqref{eqn:hilbert_schmidt_approximation_problem} which also satisfies \eqref{eqn:minimality_property} is unique if and only if $TN_1S=TN_2S$ for any two solutions $N_1$ and $N_2$ of \eqref{eqn:hilbert_schmidt_approximation_problem}. For the forward implication, assume that there exists a unique solution of \eqref{eqn:hilbert_schmidt_approximation_problem} satisfying \eqref{eqn:minimality_property}. Suppose that $N_1$ and $N_2$ are solutions of \eqref{eqn:hilbert_schmidt_approximation_problem}. Since $TP_{\ker{T}^\perp}N_iP_{\ran{S}}S=TN_iS$ for $i=1,2$, also $P_{\ker{T}^\perp}N_iP_{\ran{S}}$ solves \eqref{eqn:hilbert_schmidt_approximation_problem}. Now, $P_{\ker{T}^\perp}N_iP_{\ran{S}}$ satisfies \eqref{eqn:minimality_property}. Therefore, $P_{\ker{T}^\perp}N_1P_{\ran{S}}=P_{\ker{T}^\perp}N_2P_{\ran{S}}$ by hypothesis, which implies $TN_1S=TN_2S$. Conversely, assume that $TN_1S=TN_2S$ for any two solutions $N_1$ and $N_2$ of \eqref{eqn:hilbert_schmidt_approximation_problem}. Suppose that $N_1$ and $N_2$ are solutions of \eqref{eqn:hilbert_schmidt_approximation_problem} satisfying \eqref{eqn:minimality_property}. Since $N_1$ and $N_2$ solve $\eqref{eqn:hilbert_schmidt_approximation_problem}$, we have by hypothesis $TN_1S=TN_2S$. Applying to both sides of the equation $T^\dagger$ from the left and $S^\dagger$ from the right, and using $T^\dagger T=P_{\ker{T}^\perp}$ and $SS^\dagger=P_{\ran{S}}$, c.f.\ \cite[eqs.\ (2.12)-(2.13)]{engl_regularization_1996}, we obtain $P_{\ker{T}^\perp}N_1P_{\ran{S}}=P_{\ker{T}^\perp}N_2P_{\ran{S}}$. Because $N_1$ and $N_2$ satisfy \eqref{eqn:minimality_property}, this implies $N_1=N_2$.
\end{remark}

With \Cref{thm:generalised_friedland_torokhti} and \Cref{lemma:square_roots}\ref{item:mean_approximation_problem_equivalence_3}, we can now identify solutions of \Cref{prob:optimal_mean}, by solving \Cref{prob:optimal_mean_reformulated} for $\widetilde{A}^{\opt,(i)}\in\widetilde{\mathscr{M}}^{\opt,(i)}$ and setting ${A}^{\opt,(i)}=S_\pos\widetilde{A}^{\opt,(i)}$. We first consider the low-rank posterior mean approximation problem for the structure-ignoring approximation class ${\mathscr{M}}_r^{(2)}$ given in \eqref{eqn:equivalent_class_approx_mean_without_covariance_update}, compute the corresponding minimal loss, and show that the solution $A^{\opt,(2)}$ not only satisfies $\ran{A^{\opt,(2)}}\subset\ran{\mathcal{C}}_{\pr}^{1/2}$, but also $\ran{A^{\opt,(2)}}\subset\ran{\mathcal{C}}_{\pr}=\ran{\mathcal{C}_\pos}$. The latter condition is also satisfied by the exact posterior mean, since $\ran{\mathcal{C}_\pos G^*\mathcal{C}_\obs^{-1}}\subset\ran{\mathcal{C}_\pos}$.

\begin{restatable}{theorem}{thmOptimalLowRankMeanApprox}
	\label{thm:optimal_low_rank_mean_approx}
	Fix $r\leq n$. Let $(\lambda_i,w_i)_i$ be as in \Cref{prop:bayesian_feldman_hajek} and $(\varphi_i)_{i=1}^n$ be as in \eqref{eqn:svd_preconditioned_hessian}.
	Then a solution of \Cref{prob:optimal_mean} for $i=2$ is given by $A_r^{\opt,(2)} = \mathcal{C}_\pr^{1/2}(\sum_{i=1}^{r}\sqrt{-\lambda_i(1+\lambda_i)}w_i\otimes \varphi_i)\mathcal{C}_\obs^{-1/2}\in\mathscr{M}_r^{(2)}.$ Furthermore, 
	$\ran{A_r^{\opt,(2)}} \subset \ran{\mathcal{C}_\pos}$, the corresponding loss is $\frac{1}{2}\sum_{i>r}^{}\frac{-\lambda_i}{1+\lambda_i}$, and the solution $A^{\opt,(2)}_r$ is unique if and only if the following holds: $\lambda_{r+1}=0$ or $\lambda_r<\lambda_{r+1}$.
\end{restatable}

Next, we solve \Cref{prob:optimal_mean} for the structure-preserving approximation class $\mathcal{\mathscr{M}}_r^{(1)}$, and show that the solutions in fact satisfy $\ran{A^{\opt,(1)}}\subset\ran{\mathcal{C}}_{\pr}=\ran{\mathcal{C}_\pos}$.

\begin{restatable}{theorem}{optimalLowRankUpdateMeanApprox}
	\label{thm:optimal_low_rank_update_mean_approx}
	Fix $r\leq n$. Let $(\lambda_i)_i$ be as in \Cref{prop:bayesian_feldman_hajek} and $\mathcal{C}^\opt_r$ be an optimal rank-$r$ approximation of $\mathcal{C}_\pos$ from \eqref{eqn:optimal_covariance} in \Cref{thm:opt_covariance_and_precision}.
	Then a solution of \Cref{prob:optimal_mean} for $i=1$ is given by $A_r^{\opt,(1)} = \mathcal{C}^\opt_r G^* \mathcal{C}_\obs^{-1}\in\mathscr{M}_r^{(1)}$. Furthermore, $\ran{A_r^{\opt,(1)}} \subset \ran{\mathcal{C}_\pos}$, the corresponding loss is $\frac{1}{2}\sum_{i>r}^{}\left( \frac{-\lambda_i}{1+\lambda_i} \right)^3$ and the solution $A_r^{\opt,(1)}$ is unique if and only if the following holds: $\lambda_{r+1}=0$ or $\lambda_r<\lambda_{r+1}$.
\end{restatable}

By \eqref{eqn:optimal_covariance}, $\mathcal{C}_r^{\opt} = \mathcal{C}_\pr - \sum_{i>r}^{}-\lambda_i(\mathcal{C}_\pr w_i)\otimes (\mathcal{C}_\pr w_i)$. We thus see that the optimal operator $A^{\opt,(1)}_r$ in \Cref{thm:optimal_low_rank_update_mean_approx} is of the form $(\mathcal{C}_\pr-B)G^*\mathcal{C}^{-1}_\obs$, where $B$ satisfies the conditions in \eqref{eqn:equivalent_class_approx_mean_with_covariance_update} and is also self-adjoint.

\Cref{thm:optimal_low_rank_mean_approx} and \Cref{thm:optimal_low_rank_update_mean_approx} generalise the results of \cite[Theorem 4.1 and Theorem 4.2]{Spantini2015} to an infinite-dimensional setting, and add a uniqueness statement. We note that in both considered approximation classes $\mathscr{M}_r^{(i)}$, $i\in\{1,2\}$, the optimal operator $A^{\opt,(i)}_r$ maps into $\ran{\mathcal{C}_\pos}$, just like the exact operator $\mathcal{C}_\pos G^*\mathcal{C}_\obs^{-1}$ in \eqref{eqn:pos_mean}. 

By \eqref{eqn:loss_for_means}, the optimal posterior mean approximations given in \Cref{thm:optimal_low_rank_mean_approx} and \Cref{thm:optimal_low_rank_update_mean_approx} correspond to optimal approximations of the posterior distribution with respect to the average forward and reverse KL divergence and average R\'enyi divergences, when the posterior covariance is kept fixed. Let us define the following functions on $[0,\infty)$, where $\alpha\in(0,1)$:
\begin{align}
	\label{eqn:tranform_amari_hellinger}
	g_{\am,\alpha}(x) \coloneqq -\alpha^{-1}(1-\alpha)^{-1}\left( \exp(-\alpha(1-\alpha)x)-1 \right),\quad
	g_{\hel}(x)  \coloneqq \left( 2(1-\exp(-x/4)) \right)^{1/2}.
\end{align}
Both functions have a negative second derivative and are thus concave.
By \Cref{rmk:hellinger_and_amari_divergence}, \Cref{thm:optimal_low_rank_mean_approx,thm:optimal_low_rank_update_mean_approx}, and Jensen's inequality, we then directly obtain upper bounds on the average Amari $\alpha$-divergences $D_{\am,\alpha}(\cdot\Vert\cdot)$ and the average Hellinger distance $D_\hel(\cdot,\cdot)$. 
We summarise this in \Cref{cor:optimal_mean_for_amari_and_hellinger}.

\begin{restatable}{corollary}{optimalMeanForAmariAndHellinger}
	\label{cor:optimal_mean_for_amari_and_hellinger}
	Let $r\leq n$, $i=1,2$ and define $\gamma(1)=3$ and $\gamma(2)=1$. Let $(\lambda_j)_j$ be as in \Cref{prop:bayesian_feldman_hajek} and let $A^{\opt,(i)}_r$ be given by \Cref{thm:optimal_low_rank_update_mean_approx} for $i=1$ and by \Cref{thm:optimal_low_rank_mean_approx} for $i=2$.  Then, for $\alpha\in(0,1)$,
	\begin{align*}
		\mathbb{E}\left[D_{\am,\alpha}(\mathcal{N}(A^{\opt,(i)}_rY,{\mathcal{C}}_\pos)\Vert\mu_\pos(Y))\right]
		&\leq \frac{-1}{\alpha(1-\alpha)}\left(\exp\left(-\frac{\alpha(1-\alpha)}{2}\sum_{j>r}^{}\left(\frac{-\lambda_j}{1+\lambda_j}\right)^{\gamma(i)}\right)-1\right),\\
		\mathbb{E}\left[D_{\am,\alpha}(\mu_\pos(Y)\Vert\mathcal{N}(A^{\opt,(i)}_rY,{\mathcal{C}}_\pos))\right]
		&\leq \frac{-1}{\alpha(1-\alpha)}\left(\exp\left(-\frac{\alpha(1-\alpha)}{2}\sum_{j>r}^{}\left(\frac{-\lambda_j}{1+\lambda_j}\right)^{\gamma(i)}\right)-1\right),
	\end{align*}
	and
	\begin{align*}
		\mathbb{E}\left[D_{\hel}(\mu_\pos(Y),\mathcal{N}(A^{\opt,(i)}_rY,\mathcal{C}_\pos))\right]
		&\leq \sqrt{2\left(1-\exp\left(-\frac{1}{8}\sum_{j>r}^{}\left(\frac{-\lambda_j}{1+\lambda_j}\right)^{\gamma(i)}\right)\right)}.
	\end{align*}
The operator $A_r^{\opt,(i)}$ is unique if and only if the following holds: $\lambda_{r+1}=0$ or $\lambda_r<\lambda_{r+1}$.
\end{restatable}

Similarly to \cite[Section 4.1]{Spantini2015}, a comparison between the minimal losses of \Cref{thm:optimal_low_rank_mean_approx} and \Cref{thm:optimal_low_rank_update_mean_approx} gives us insight as to which approximation procedure is preferable in a specific setting. As the theorems show, the decay of the eigenvalues $(\lambda_i)_i$ of $R(\mathcal{C}_\pos\Vert\mathcal{C}_\pr)$ governs this choice. The loss of the optimal approximation in \Cref{thm:optimal_low_rank_mean_approx} and in \Cref{thm:optimal_low_rank_update_mean_approx} is $\frac{1}{2}\sum_{i>r}(\tfrac{-\lambda_i}{1+\lambda_i})$ and $\frac{1}{2}\sum_{i>r}(\tfrac{-\lambda_i}{1+\lambda_i})^3$ respectively. If $\tfrac{-\lambda_i}{1+\lambda_i}\leq 1$ or equivalently $-\lambda_i\leq \tfrac{1}{2}$ for every $i>r$, then we have $\sum_{i>r}(\tfrac{-\lambda_i}{1+\lambda_i})\geq \sum_{i>r}(\tfrac{-\lambda_i}{1+\lambda_i})^3$. 
Since the sequence $(\lambda_i)_{i}\subset (-1,0]$ increases to zero by \Cref{prop:bayesian_feldman_hajek}, and since $(\lambda_i)_{i}$ have the interpretation of variance reduction by the discussion after \Cref{prop:variance_reduction}, it follows that if there exists some $r<n$ such that the relative variance reduction along $\mathcal{C}_\pr^{-1/2}w_i$ is smaller than $\tfrac{1}{2}$ for $i>r$, then the loss $\frac{1}{2}\sum_{i>r}(\tfrac{-\lambda_i}{1+\lambda_i})^3$ that arises from exploiting the structure \eqref{eqn:pos_mean} of the posterior mean is smaller than the loss that ignores this structure. In other words, one can achieve on average a smaller loss in the posterior mean approximation that exploits the structure \eqref{eqn:pos_mean} of the posterior mean, if the ratio of the posterior variance to the prior variance along $\mathcal{C}_\pr^{-1/2}w_i$ decays below the threshold of $\tfrac{1}{2}$ for sufficiently large $i$. If for example $\lambda_i>-\tfrac{1}{2}$ for every $i\in\N$, then this decay does not occur, and one can obtain a smaller loss by ignoring the structure.

 In the following, we interpret the optimal low-rank posterior mean approximations in terms of projections of the prior and the posterior means.

\begin{restatable}{lemma}{interpretationOptimalMeans}
	\label{lemma:interpretation_optimal_means}
	Let $r\leq n$ and $A^{\opt,(i)}_{r}$ for $i=1,2$ be defined in \Cref{thm:optimal_low_rank_update_mean_approx,thm:optimal_low_rank_mean_approx} and denote by $m_\pr=0$ the prior mean. Let $\H=W_r+W_{-r}$ be the direct sum of $W_r$ and $W_{-r}$ defined in \eqref{eqn:subspace_of_maximal_variance_reduction} and \eqref{eqn:complement_subspace_of_maximal_variance}. Let $P_{W_r}$ and $P_{W_{-r}}$ be the orthogonal projectors onto $W_r$ and $W_{-r}$ respectively.
	Then for every realisation $y$ of $Y$, we have
	\begin{equation*}
		\begin{split}
			P_{W_r}A^{\opt,(1)}_ry &= P_{W_r}m_\pos(y), \\
			P_{W_r}A^{\opt,(2)}_ry &= P_{W_r}m_\pos(y), 
		\end{split}
		\qquad
		\begin{split}
			P_{W_{-r}}A^{\opt,(1)}_ry &=  P_{W_{-r}}\mathcal{C}_\pr G^*\mathcal{C}_\obs^{-1}y,\\
			P_{W_{-r}}A^{\opt,(2)}_ry &=  P_{W_{-r}}m_\pr.
		\end{split}
	\end{equation*}
\end{restatable}

From \Cref{lemma:interpretation_optimal_means} we see that $P_{W_r}A^{\opt,(1)}_ry= P_{W_r}A^{\opt,(2)}_ry$, but $P_{W_{-r}}A^{\opt,(1)}_ry$ and $ P_{W_{-r}}A^{\opt,(2)}_ry$ differ in general.

\section{Optimal joint approximation of the mean and covariance}
\label{sec:optimal_joint_approximation}

In \Cref{sec:optimal_approximation_covariance}, we considered the optimal rank-$r$ approximation of the posterior covariance given the same mean, while in \Cref{sec:optimal_approximation_mean} we considered the optimal rank-$r$ approximation of the posterior mean given the same posterior covariance. In this section, we consider jointly approximating the posterior mean and covariance in the reverse KL divergence defined in \Cref{sec:equivalence_and_divergences_between_gaussian_measures}. Approximation in reverse KL divergence is important in the context of variational inference, c.f.\ \cite[Theorem 5]{Ray2022}. We leave the solution of the optimal joint approximation of the mean and covariance for the forward KL divergence for future work.


Let $y\in\R^n$ be an arbitrary data vector and $m_\pos(y)$ be as in \eqref{eqn:pos_mean}. Let $\widetilde{m}_\pos(y)$ be an approximation of $m_\pos(y)$ and $\widetilde{\mathcal{C}}_\pos$ be an approximation of $\mathcal{C}_\pos$ such that $\mathcal{N}(\widetilde{m}_\pos(y),\widetilde{\mathcal{C}}_\pos)\sim\mu_\pos$, and let $m\in\H$ be arbitrary.
Then, by \eqref{eqn:kullback_leibler_divergence},
\begin{align*}
	D_{\kl}(\mathcal{N}(\widetilde{m}_\pos(y),\widetilde{\mathcal{C}}_\pos)\Vert \mu_\pos)
	= &\frac{1}{2}\Norm{\mathcal{C}_\pos^{-1/2}(\widetilde{m}_\pos(y)-m_\pos(y))}^2-\frac{1}{2}\log\det_2\left(I+R(\widetilde{\mathcal{C}}_\pos\Vert\mathcal{C}_\pos)\right) 
	\\
	= &\frac{1}{2}\Norm{\mathcal{C}_\pos^{-1/2}(\widetilde{m}_\pos(y)-m_\pos(y))}^2+D_{\kl}(\mathcal{N}(m,\widetilde{\mathcal{C}}_\pos)\Vert \mathcal{N}(m,\mathcal{C}_\pos))
	\\
	= &D_{\kl}(\mathcal{N}(\widetilde{m}_\pos(y),\mathcal{C}_\pos)\Vert \mathcal{N}(m_\pos(y),\mathcal{C}_\pos))\\
	&+D_{\kl}(\mathcal{N}(m,\widetilde{\mathcal{C}}_\pos)\Vert \mathcal{N}(m,\mathcal{C}_\pos)),
\end{align*}
which constitutes a Pythagorean-like identity for the Kullback--Leibler divergence between two Gaussians. The identity above is reasonable, since the Kullback--Leibler divergence is a Bregman divergence, which are known to satisfy generalised Pythagorean theorems. See e.g.\ \cite[Section 1.6]{Amari2016} or \cite{Nielsen2022} for the information geometry perspective on Pythagorean identities and \cite[Theorem 2.1]{Li2024a} for a Pythagorean theorem in the context of dimension reduction for Bayesian inverse problems.

In our context, the Pythagorean identity above implies that, in order to solve the joint approximation problem, it suffices to solve the posterior mean approximation problem and the posterior covariance approximation problems separately. Let $r\in\N$. Suppose we search for $\widetilde{m}_\pos(y)$ of the form $Ay$ for $A$ in one of the approximation classes $\mathscr{M}^{(i)}_r$ defined in \eqref{eqn:class_approx_means}, and that we search for $\widetilde{\mathcal{C}}_\pos$ of the form $\mathcal{C}_\pr-KK^*$ from $\mathscr{C}_r$ defined in \eqref{eqn:class_approx_covariance}. Then for $i=1,2$ and any $m\in\mathcal{H}$,
\begin{align*}
	&\min\left\{\mathbb{E}\left[D_{\kl}(\mathcal{N}( AY,\mathcal{C}_\pr-KK^*)\Vert \mathcal{N}(m_\pos(Y),\mathcal{C}_\pos))\right]:\ A\in\mathscr{M}^{(i)}_r,\ \mathcal{C}_\pr-KK^*\in\mathscr{C}_r\right\}
	\\
	=& \min\left\{ \mathbb{E}\left[D_{\kl}(\mathcal{N}(AY,\mathcal{C}_\pos)\Vert \mathcal{N}(m_\pos(Y),\mathcal{C}_\pos))\right]:\ A\in\mathscr{M}^{(i)}_r\right\}
	\\
	&+ \min\left\{ D_{\kl}(\mathcal{N}(m,\mathcal{C}_\pr-KK^*) \Vert \mathcal{N}(m,{\mathcal{C}_\pos})):\ \mathcal{C}_\pr-KK^*\in\mathscr{C}_r\right\}.
\end{align*}
The two minimisation problems can then be solved using \Cref{thm:opt_covariance_and_precision} and either \Cref{thm:optimal_low_rank_mean_approx} or \Cref{thm:optimal_low_rank_update_mean_approx}: 
\begin{restatable}{proposition}{optimalJointApproximation}
	\label{prop:optimal_joint_approximation}
	Let $r\leq n$, $i=1,2$, and $(\lambda_j)_j$ be as in \Cref{prop:bayesian_feldman_hajek}. Let $\mathcal{C}^\opt_r$ be as in \Cref{thm:opt_covariance_and_precision} and $A^{\opt,(i)}_r$ be as in either \Cref{thm:optimal_low_rank_mean_approx,thm:optimal_low_rank_update_mean_approx}. Then,
	\begin{align*}
		&\min\left\{\mathbb{E}\left[D_{\kl}(\mathcal{N}( AY,\mathcal{C}_\pr-KK^*)\Vert \mathcal{N}(m_\pos(Y),\mathcal{C}_\pos))\right]:\ A\in\mathscr{M}^{(i)}_r,\ \mathcal{C}_\pr-KK^*\in\mathscr{C}_r\right\}\\
		&=\mathbb{E}\left[D_{\kl}(\mathcal{N}( A^{\opt,(i)}_rY,\mathcal{C}^\opt_r)\Vert \mathcal{N}(m_\pos(Y),\mathcal{C}_\pos))\right],\\
		&=\sum_{j>r}^{}f_\kl\left(\frac{-\lambda_j}{1+\lambda_j}\right) + \frac{1}{2}\left( \frac{-\lambda_j}{1+\lambda_j} \right)^{\gamma(i)},
	\end{align*}
	where $\gamma(1)=3$ by \Cref{thm:optimal_low_rank_update_mean_approx}, $\gamma(2)=1$ by \Cref{thm:optimal_low_rank_mean_approx}, and where $f_\kl$ is defined in \eqref{eqn:def_f_kl}. Furthermore, $(A^{\opt,(i)}_r,\mathcal{C}^\opt_r)$ is the unique minimiser if and only if the following holds: $\lambda_{r+1}=0$ or $\lambda_r<\lambda_{r+1}$.
\end{restatable}

	The choice of the user-specified truncation parameter $r$ in \Cref{prop:optimal_joint_approximation}, \Cref{thm:optimal_low_rank_update_mean_approx,thm:optimal_low_rank_mean_approx}, and \Cref{cor:optimal_mean_for_amari_and_hellinger}, may depend on the specific inverse problem that is considered. Usually, $r$ can be chosen small due to the rapid decay of the prior-preconditioned Hessian, c.f.\ \cite{Bui-Thanh2013}. Clearly, $r\leq \rank{H}\leq n$, since the choice $r\leftarrow \rank{H}$ recovers the exact posterior. We now discuss some guidelines for choosing $r$ in \Cref{prop:optimal_joint_approximation}. One may choose $r$ based on a spectral cutoff criterion, in which $r$ is taken as the smallest integer such that $\lambda_{r+1}<\varepsilon$ or $\lambda_{r+1}/\lambda_1<\varepsilon$ for some chosen threshold $\varepsilon>0$. Alternatively, one may exploit that only finitely many $\lambda_j$ are nonzero by \Cref{prop:bayesian_feldman_hajek}, and bound the optimal error in \Cref{prop:optimal_joint_approximation} according to 
\begin{align*}
	\sum_{j>r}^{}f_\kl\left(\frac{-\lambda_j}{1+\lambda_j}\right) + \frac{1}{2}\left( \frac{-\lambda_j}{1+\lambda_j} \right)^{\gamma(i)}
	\leq (n-r)\left[f_{\kl}(-\lambda_{r+1}/(1+\lambda_{r+1}))+\frac{1}{2}(-\lambda_{r+1}/(1+\lambda_{r+1}))^{\gamma(i)}\right],
\end{align*}
for $i=1,2$. The right-hand side decreases in $r$ and can be made smaller than a chosen tolerance by choosing $r$ large enough. Furthermore, by \eqref{eqn:prior_preconditioned_Hessian} and by the functional calculus, the optimal error for $r=0$ satisfies
\begin{align*}
	\sum_{j\geq 0}^{}f_\kl\left(\frac{-\lambda_j}{1+\lambda_j}\right) + \frac{1}{2}\left( \frac{-\lambda_j}{1+\lambda_j} \right)^{\gamma(i)}
	= \tr{\omega^{(i)}(\mathcal{C}_\pr^{1/2} H \mathcal{C}_\pr^{1/2})}.
\end{align*}
Here, the function $\omega^{(i)}(x)\coloneqq f_{\kl}(x)+\frac{1}{2}x^{\gamma(i)}$ is analytic on a compact interval of $(-1,0]$ containing $(\lambda_j)_j$. By the definitions \eqref{eqn:class_approx_means} and \eqref{eqn:class_approx_covariance}, the optimal error for $r=0$ corresponds to the average reverse KL divergence $\mathbb{E}[D_{\kl}(\mu_\pos(Y)\Vert\mu_\pr)]$ between the prior and posterior. In a discretised setting, so-called `stochastic Lanczos quadrature' can be used to approximate $\tr{\omega^{(i)}(\mathcal{C}_\pr^{1/2} H \mathcal{C}_\pr^{1/2})}$ efficiently, see \cite{Ubaru2017}. Then, $r$ can be chosen to approximately control the reduction in average reverse KL divergence relative to the prior, which is given by
\begin{align*}
	\frac{\sum_{j>r}^{}f_\kl\left(\frac{-\lambda_j}{1+\lambda_j}\right) + \frac{1}{2}\left( \frac{-\lambda_j}{1+\lambda_j} \right)^{\gamma(i)}}{\sum_{j\geq 0}^{}f_\kl\left(\frac{-\lambda_j}{1+\lambda_j}\right) + \frac{1}{2}\left( \frac{-\lambda_j}{1+\lambda_j} \right)^{\gamma(i)}}
	=
	\frac{ \tr{\omega^{(i)}(\mathcal{C}_\pr^{1/2} H\mathcal{C}_\pr^{1/2})} - \sum_{j\leq r}^{}f_\kl\left(\frac{-\lambda_j}{1+\lambda_j}\right) + \frac{1}{2}\left( \frac{-\lambda_j}{1+\lambda_j} \right)^{\gamma(i)}}{\tr{\omega^{(i)}(\mathcal{C}_\pr^{1/2} H\mathcal{C}_\pr^{1/2})}}.
\end{align*}
Similar arguments can be applied for the choice of $r$ for the optimal posterior mean approximations and the corresponding losses of \Cref{thm:optimal_low_rank_update_mean_approx,thm:optimal_low_rank_mean_approx} and \Cref{cor:optimal_mean_for_amari_and_hellinger}, and the optimal posterior covariance approximations of \Cref{thm:opt_covariance_and_precision}. Recall that the optimal error of \Cref{prop:optimal_joint_approximation} consists of the contributions $\sum_{j>r}^{}f_\kl(-\lambda_j/(1+\lambda_j))$ and $\sum_{j>r}^{}\frac{1}{2}( -\lambda_j(1+\lambda_j))^{\gamma(i)}$ of the posterior covariance and the posterior mean approximations, respectively. Thus, these relative contributions can be balanced by choosing separate truncation parameters for the mean and covariance. Finally, we mention that the approximation errors in the different losses can be balanced against computational costs and storage costs, depending on user-defined computational objectives.


%

\section{Characterisation through optimal projection}
\label{sec:optimal_projector}

Let $P_r\in\B(\mathcal{H})$ be a projector of rank at most $r$, i.e $(P_r)^2=P_r$ and $\rank{P_r}\leq r$. Then $GP_r\in\B_{00,r}(\mathcal{H})$ and we consider the Bayesian inverse problem
\begin{align}
	\label{eqn:projected_observation_model}
	Y = GP_rX + \zeta,\quad \zeta\sim\mathcal{N}(0,\mathcal{C}_\obs),
\end{align}
where again $X\sim\mu_{\pr}=\mathcal{N}(0,\mathcal{C}_\pr)$.
This problem only differs from \Cref{sec:formulation} in the replacement of the forward map $G$ by $GP_r$. As before, we denote by $y$ an arbitrary realisation of $Y$. Let $\mu_{P_r,\pos}(y)=\mathcal{N}(m_{P_r,\pos}(y),\mathcal{C}_{P_r,\pos})$ be the posterior distribution corresponding to \eqref{eqn:projected_observation_model} and $\mu_\pr=\mathcal{N}(0,\mathcal{C}_{\pr})$. Because $GP_r$ is continuous, it follows from \cite[Theorem 6.31]{Stuart2010} that $\mu_{P_r,\pos}(y)\sim\mu_\pr\sim\mu_\pos(y)$, where $\mu_\pos(y)$ is the posterior distribution of the full observation model \eqref{eqn:observation_model}. 
For the chosen value of $r$ and $i=1,2$, let $\mu^{\opt,(i)}_{\pos,r}(y)=\mathcal{N}(m^{\opt,(i)}_{\pos,r}(y),\mathcal{C}^\opt_r)$ denote the data-averaged optimal posterior approximation of $\mu_\pos(y)$ obtained in \Cref{sec:optimal_joint_approximation}. Thus, $\mathcal{C}^{\opt}_r$ is given by \Cref{thm:opt_covariance_and_precision} and $m^{\opt,(i)}_{\pos,r}(y)=A^{\opt,(i)}_ry$ is given by \Cref{thm:optimal_low_rank_update_mean_approx} for $i=1$ and \Cref{thm:optimal_low_rank_mean_approx} for $i=2$. \Cref{prop:optimal_joint_approximation}, \eqref{eqn:pos_mean} applied with $G$ replaced by $GP_r$, and the definition of $\mathscr{M}_r^{(2)}$ in \eqref{eqn:class_approx_mean_without_covariance_update}, imply for $i=2$ that $\mathbb{E}\left[D_\kl(\mu_{P_r,\pos}(Y)\Vert\mu_{\pos}(Y))\right] \geq \mathbb{E}\left[D_\kl(\mu^{\opt,(i)}_{\pos,r}(Y)\Vert \mu_\pos(Y))\right]$. For $i=2$, we show that this lower bound is attained, that is, there exists a suitable choice $P^\opt_r$ of $P_r$ such that for every realisation $y$ we have $\mu_{P^\opt_r,\pos}(y)=\mu^{\opt,(2)}_{\pos,r}(y)$. The proof is given in \Cref{subsec:proofs_for_optimal_projector}.

\begin{restatable}{proposition}{optimalProjector}
	\label{prop:optimal_projector}
	Let $r\leq n$ and $(\lambda_i,w_i)_i$ be as in \Cref{prop:bayesian_feldman_hajek}. With $P^\opt_r\in\B(\mathcal{H})$ defined by $P^\opt_r\coloneqq\sum_{i=1}^{r} (\mathcal{C}_\pr^{1/2}w_i)\otimes (\mathcal{C}_\pr^{-1/2}w_i)$, it holds that $P^\opt_r$ is a projector of rank at most $r$, and that the Bayesian inverse problem \eqref{eqn:projected_observation_model} for $P_r\leftarrow P^\opt_r$ and for an arbitrary realisation $y$ of $Y$ has posterior distribution $\mathcal{N}(A^{\opt,(2)}_ry,\mathcal{C}^\opt_r)$, where $\mathcal{C}^\opt_r$ is a solution of \Cref{prob:optimal_covariance} as given by \eqref{eqn:optimal_covariance}, and $A^{\opt,(2)}_r$ is a solution to \Cref{prob:optimal_mean} for $i=2$.
\end{restatable}

In the finite-dimensional setting, it is shown in \cite[Corollary 3.2]{Spantini2015} that the posterior covariance corresponding to the model \eqref{eqn:projected_observation_model} agrees with the solution of \Cref{prob:optimal_covariance} for the choice of $P^\opt_r$ given in \Cref{prop:optimal_projector}. \Cref{prop:optimal_projector} generalises this to infinite dimensions and adds an analogous statement for the posterior mean of model \eqref{eqn:projected_observation_model}: the exact posterior mean of the projected problem \eqref{eqn:projected_observation_model} with $P_r\leftarrow P^\opt_r$ as in \Cref{prop:optimal_projector} is equal to the optimal low-rank structure-ignoring posterior mean approximation given by \Cref{thm:optimal_low_rank_mean_approx}.

From the analogue of \eqref{eqn:pos_mean} with $G$ replaced by $GP^\opt_r$ we immediately see that the posterior mean is a linear transformation of the data $y$ by an operator of rank at most $r$. Since $A^{\opt,(1)}_r$ given in \Cref{thm:optimal_low_rank_update_mean_approx} does not in general have rank at most $r$, it follows that $A^{\opt,(1)}_r y$ cannot be obtained as the posterior mean of model \eqref{eqn:projected_observation_model} for any $P^\opt_r\in\B_{00,r}(\mathcal{H})$.

For $W_r$ defined in \eqref{eqn:subspace_of_maximal_variance_reduction}, the likelihood-informed subspace $\ran{P^\opt_r}=\mathcal{C}_\pr(W_r)$ defined at the end of \Cref{sec:equivalence_and_divergences_between_gaussian_measures} is a one-to-one transformation of $W_r$. Recall from \Cref{prop:variance_reduction} and the discussion following it that $W_r$ is the $r$-dimensional subspace which reduces the prior variance the most in relative terms, among all $r$-dimensional subspaces of $\H$. By \Cref{rmk:interpretation_covariance_approx} and \Cref{lemma:interpretation_optimal_means}, it holds that $\mathcal{C}^\opt_r=\mathcal{C}_\pos$ on $W_r$ and $P_{W_r}A^{\opt,(2)}_r y = P_{W_r}m_\pos(y)$ for every realisation $y$ of $Y$, where $P_{W_r}$ denotes the orthogonal projector onto $W_r$. Furthermore, $\mathcal{C}^\opt_r=\mathcal{C}_\pr$ on $W_{-r}$ and $P_{W_{-r}}A^{\opt,(2)}_r y = P_{W_{-r}}m_{\pr}$, where $P_{W_{-r}}$ denotes the orthogonal projector onto the subspace $W_{-r}$ defined in \eqref{eqn:complement_subspace_of_maximal_variance} and $m_{\pr}=0$ is the prior mean.
Thus, the optimal joint approximation with structure-ignoring approximate mean yields the exact posterior measure for the projected inverse problem in which the data is only used to inform $W_r$.

\section{Examples}
\label{sec:examples_short}

In this section we consider two typical ill-posed inverse problems to illustrate the proposed framework. We identify the prior-preconditioned Hessian \eqref{eqn:prior_preconditioned_Hessian} and its non-self adjoint square root \eqref{eqn:svd_preconditioned_hessian} in terms of the functions occurring in the forward problem and the prior. After discretising these expressions, matrix-free methods such as Krylov or Lanczos algorithms and randomized parallel schemes can be used to efficiently approximate the corresponding truncated rank-r SVD; see e.g.\ \cite{Flath2011,Bui-Thanh2012b,Saad2003}. With the $r$ leading eigendirections, the optimal projector $P^\opt_r$ in \Cref{prop:optimal_projector} can then be constructed, yielding the projected Bayesian inverse problem \eqref{eqn:projected_observation_model} which contains the essential posterior information. Further details and explanations are provided in \Cref{sec:examples}.

\begin{example}[Deconvolution]
	\label{ex:deconvolution}
	Let $\H=L^2([0,1])$ and let $\kappa:[0,1]^2\rightarrow\R$ be square integrable. We consider a convolution operator $T_k$ on $\H$ with kernel $\kappa$. That is,
	$(T_\kappa h )(t) = \int_0^1 \kappa (t,s) h(s)\d s$, for $h\in\H$ and for almost every $t\in[0,1]$.
	The unknown $x^\dagger\in L^2([0,1])$ is convolved by $T_\kappa$, and needs to be recovered using the $n$ data points $y_i = \int_{t_i}^{t_{i+1}} (T_\kappa x^\dagger)(s) \gamma (s)\d s+ \zeta_i$, where $\gamma\in \mathcal{H}$ is known, $\zeta_i$ is i.i.d standard Gaussian, and $0\leq t_1<\cdots <t_{n+1}\leq 1$.
	Under suitable assumptions on $\kappa$, we have $\kappa(s,t) = \sum_{i=1}^{\infty}b_i f_i(s)f_i(t)$, where $(b_i)_i$ is a nonnegative zero-sequence and $(f_i)_i$ is an ONB of $\H$ consisting of bounded functions. 

	In the Bayesian perspective, we endow $x^\dagger$ with a prior distribution, which is taken to be $\mathcal{N}(0,\mathcal{C}_\pr)$ with $\mathcal{C}_\pr=\sum_{i}^{}c_i f_i\otimes f_i$ for some $c\in\ell^2( (0,\infty))$. 
	Then, the problem can be cast in the form \eqref{eqn:observation_model} and the operators \eqref{eqn:svd_preconditioned_hessian} and \eqref{eqn:prior_preconditioned_Hessian} take the form, for $z\in\R^n$,
	\begin{align*}
		\mathcal{C}_\pr^{1/2}G^*\mathcal{C}_\obs^{-1/2}z&= \sum_{i=1}^{n}\sum_{j}^{}z_ib_jc_ja_{j,i}f_j,\quad
		\mathcal{C}_\pr^{1/2} H\mathcal{C}_\pr^{1/2} =  \sum_{j,k}^{}d_{k,j}f_k\otimes f_j. 
	\end{align*}
	The coefficients $d_{k,j}=b_jc_jb_kc_k\sum_{i=1}^{n}\langle f_j,1_{[t_i,t_{i+1}]}\gamma\rangle\langle f_k,1_{[t_i,t_{i+1}]}\gamma\rangle$ and the orthonormal sequence $(f_j)_j$ are explicitly known and depend on the choice of prior via $(c_i)_i$ and on the forward model via $(f_k)_k$, $(b_i)_i$ and $\gamma$.
\end{example}

\begin{example}[Inferring the initial condition of the heat equation]
	\label{ex:heat_equation}
	Suppose the temperature field $(x,t)\mapsto u(x,t)$ on $(0,1)\times[0,T]$ solves the heat equation
	\begin{align*}
		\partial_t u - \partial_{xx} u &=  0, &&\text{in } (0,1)\times(0,T),\\
		u(\cdot,0) &= x^\dagger, \quad &&\text{on } (0,1),\\
		u(0,\cdot) =u(1,\cdot) &= 0, && \text{on }(0,T].
	\end{align*}
	The true initial state $x^\dagger$ is unknown and needs to be estimated from noisy observations of $u$ at $(x_i,t_i)_{i=1}^n\subset(0,1)\times(0,T]$. We assume i.i.d.\ standard Gaussian noise. This problem is similar to \cite[Example 3.5]{Stuart2010} and \cite[Section 4.2]{Flath2011}. However, in this example we do not observe the entire spatial temperature profile, but observe at finitely many fixed spatial locations, and we consider periodic instead of Dirichlet boundary conditions. 

	The Laplacian can be expressed as $\Delta h=-\sum_{i}^{}a_i \langle h,e_i\rangle e_i$ for any $h\in\dom{\Delta}=\{h\in L_2((0,1)):\ \sum_{i}^{}a_i^2\langle h,e_i\rangle^2<\infty\}$, where $a_i=i^2\pi^2$ and $e_i(x)=\sqrt{2}\sin(i\pi x)$.	
	We take the Bayesian perspective by considering $x^\dagger$ as an $\H$-valued random variable $X\sim\mathcal{N}(0,\mathcal{C}_\pr)$ with $\mathcal{C}_\pr=(-\Delta)^{-s}$ for some $s>\frac{1}{2}$ as in \cite{Stuart2010}.
	We can then formulate the problem in the form \eqref{eqn:observation_model}, and the operators \eqref{eqn:svd_preconditioned_hessian} and \eqref{eqn:prior_preconditioned_Hessian} can be expressed as, for $z\in\R^n$,
	\begin{align*}
		\mathcal{C}_\pr^{1/2}G^*\mathcal{C}_\obs^{-1/2}z &= \sum_{i=1}^{n}\sum_{k}^{}z_ia_k^{-s/2}\exp(-t_ia_k)e_k(x_i)e_k,\quad
		\mathcal{C}_\pr^{1/2}H\mathcal{C}_\pr^{1/2} =  \sum_{j,k}^{}d_{j,k}e_k\otimes e_j,
	\end{align*}
	where $d_{j,k}=\sum_{i=1}^{n}a_j^{-s/2}\exp(-t_ia_j)a_k^{-s/2}\exp(-t_ia_k)e_j(x_i)e_k(x_i)$ are explicitly available. 
\end{example}

\section{Numerical example}
\label{sec:numerical_example}
To verify several aspects of the theory developed in this work, we consider a numerical implementation of a linear Gaussian inverse problem governed by the parabolic heat equation. 
We introduce the inverse problem and its discretisation in \Cref{sec:heat_equation_example_formulation_and_discretisation}, and study the low-rank approximations as a function of the rank $r$ and of the discretisation dimension in \Cref{sec:numerical_results}.

\subsection{Formulation and discretisation}
\label{sec:heat_equation_example_formulation_and_discretisation}
We consider the inverse problem studied in \cite{Sanz-Alonso2024}, in which the initial condition of the heat equation is inferred based on noisy and partial observations of the final state. This inverse problem is similar to the one described in \Cref{ex:heat_equation} of \Cref{sec:examples_short}. The main differences are the choice of prior covariance operator, the choice of observation operator, and the dimension of the physical domain. 

The parameter space is given by $\H=L^2(\mathcal{D})$ with a two-dimensional smooth spatial domain $\mathcal{D}$. As in \Cref{ex:heat_equation}, the goal is to infer the initial condition $X$ of the heat equation
\begin{align}
	\label{eqn:heat_equation}
	\begin{aligned}
		\partial_t u - \Delta u &=  0, &&\text{in } \mathcal{D}\times(0,T),\\
		u &= X, \quad &&\text{on } \mathcal{D}\times\{t=0\},\\
		u &= 0, && \text{on }\partial\mathcal{D}\times(0,T],
	\end{aligned}
\end{align}
where we have imposed homogeneous Dirichlet boundary conditions on the boundary $\partial\mathcal{D}$. The observation $y$ arises by integrating the solution field $u(\cdot,T)$ at the final time $T$ against $n$ indicator functions $\psi_i\in \H, i=1,\ldots,n$. We take $\psi_i=\abs{B_\delta(s^i)}^{-1}1_{B_\delta(s^i)}$, i.e.\ $\psi_i$ is given by the indicator of a ball of radius $\delta$ centered at $s^i=(s^i_1,s^i_2)$, and is scaled to have unit $\H$-norm. The forward model $G\in\B(\H,\R^n)$ is thus given by $X\mapsto (\langle u(\cdot,T),\psi_i\rangle)_{i=1}^n$, where $u$ solves \eqref{eqn:heat_equation}. Let us denote by $\mathcal{F}\in\B(\H)$ the solution operator of the heat equation that sends the initial condition to the solution at the final time. Furthermore, let $\mathcal{O}\in\B(\H,\R^n)$ denote the observation map $h\mapsto (\langle h,\psi_i\rangle)_{i=1}^n$. Then we have $G=\mathcal{O}\circ \mathcal{F}$, which corresponds to the forward model considered in \cite[Example  2.2]{Sanz-Alonso2024}, noting that we have interchanged the notation of $G$ and $\mathcal{F}$.

The prior covariance is chosen as in \cite[Example 2.1]{Sanz-Alonso2024} to be  $\mathcal{C}_\pr = \mathcal{A}^{-\alpha}$, where $\alpha\in2\N$ and $\mathcal{A}:\dom(\mathcal{A})\subset\H\rightarrow \H$ is given by $Au = \nabla{\cdot(\Theta\nabla u)} + bu$. That is, $\mathcal{C}_\pr = (\nabla{\cdot(\Theta\nabla(\cdot))} + bI)^{-\alpha}$. The domain of $\mathcal{A}$ is given by $\dom{\mathcal{A}}=H^2(\mathcal{D})\cap H^1_0(\mathcal{D})$, where for $k\in\N$ the linear space $H^k(\mathcal{D})$ consists of the functions in $L^2(\mathcal{D})$ that have $k$ square-integrable weak derivatives, and $H^1_0(\mathcal{D})$ consists of the functions in $H^1(\mathcal{D})$ that vanish at $\partial\mathcal{D}$ in the sense of traces, see \cite[Section 1.3]{Quarteroni1994}.
The functions $\Theta,b:\mathcal{D}\rightarrow(0,\infty)$ are smooth enough and positive to ensure ellipticity of the operator $\mathcal{A}$ and the trace-class property of $\mathcal{C}_\pr$. The choice of $\alpha$ regulates the smoothness of draws from the Gaussian prior. We refer to \cite{Sanz-Alonso2024,Stuart2010} for further details.

The parameter space $\mathcal{H}$ and the prior distribution $\mu_\pr$ are approximated using a sequence of approximation spaces $\mathcal{V}_d\subset\mathcal{H}$ with $\dim{\mathcal{V}_d}=d<\infty$. The application of the prior covariance corresponds to solving a PDE, and thus the prior can be discretised by Galerkin projection onto $\mathcal{V}_d$. Indeed, the application of $\mathcal{C}_\pr$ to a function $h\in \H$ amounts to solving the following $\alpha$ elliptic PDEs, stated in weak formulation:
\begin{align*}
	\text{for each }1\leq j\leq\alpha, \text{ find }u_j\in H^1_0(\mathcal{D})\text{ s.t. }\int_\mathcal{D} (\Theta\nabla{u}_j\cdot\nabla{p} +b up)\d{x} = \int_{\mathcal{D}}^{} h_j p\d{x},\quad\text{for all }p\in H^1_0(\mathcal{D}).
\end{align*}
Here $h_\alpha=h$ and $h_j = u_{j+1}$ for $1\leq j<\alpha$. 
The forward model $G$ is discretised by discretising the heat equation \eqref{eqn:heat_equation} via a Galerkin projection onto $\mathcal{V}_d$ and a Crank--Nicolson discretisation in time with step size $\Delta{t}$. The space $\mathcal{V}_d$ is chosen as a subspace of $H^1_0(\mathcal{D})$ based on piecewise linear Lagrangian finite elements. We refer the reader to \cite[Section 2.3.1]{Sanz-Alonso2024} for more details on the discretisation. 

We denote by $\mathcal{F}_{(d,\Delta{t})}\in\B(\mathcal{V}_d)$, $\mathcal{O}_d\in\B(\mathcal{V}_d,\R^n)$, $G_{(d,\Delta{t})}\coloneqq\mathcal{O}_d\circ\mathcal{F}_{(d,\Delta{t})}$, $\mathcal{C}_{\pr,d}\in\B(\mathcal{V}_d)$ the discretised counterparts to $\mathcal{F}$, $\mathcal{O}$, $G$, and $\mathcal{C}_\pr$, respectively. The posterior distribution corresponding to the discretised inverse problem on $\mathcal{V}_d$ with forward model $G_{(d,\Delta{t})}$ and with prior $\mathcal{N}(0,\mathcal{C}_{\pr,d})$ is denoted by $\mu_{\pos,(d,\Delta{t})}(y) = \mathcal{N}(m_{\pos,(d,\Delta{t})}(y),\mathcal{C}_{\pos,(d,\Delta{t})})$. Let us also denote by $Q_d:\H\rightarrow\mathcal{V}_d$ the orthogonal projector onto $\mathcal{V}_d$ with codomain restricted to $\mathcal{V}_d$. Then the discretised posterior mean $m_{\pos,(d,\Delta{t})}(y)$ and posterior covariance $\mathcal{C}_{\pos,(d,\Delta{t})}$ provide approximations $Q_d^*m_{\pos,(d,\Delta{t})}(y)$ and $Q_d^*\mathcal{C}_{\pos,(d,\Delta{t})}Q_d$ of the exact posterior mean $m_\pos(y)$ and posterior covariance $\mathcal{C}_\pos$. In \cite[Sections 3.1 and 3.2]{Sanz-Alonso2024}, it is proven under suitable conditions that the discretisation of the inverse problem is consistent, in the sense that
\begin{align*}
	\norm{m_{\pos}(y) - Q_d^* m_{\pos,(d,\Delta{t})}(y)}\rightarrow 0,\quad\norm{\mathcal{C}_\pos - Q_d^*\mathcal{C}_{\pos,(d,\Delta{t})}Q_d}\rightarrow 0, \quad \text{ as } d\rightarrow\infty,\ \Delta{t}\rightarrow 0.
\end{align*}
Analogously, the infinite-dimensional formulation of the optimal low-rank posterior approximations developed in \Cref{sec:optimal_approximation_covariance,sec:optimal_approximation_mean,sec:optimal_joint_approximation,sec:optimal_projector} enables one to study the consistency of discretisations of these optimal approximations. This endeavor goes beyond the scope of the current work, however, and in the next section we shall instead consider a numerical implementation of the above inverse problem with the described discretisation.

\subsection{Numerical results}
\label{sec:numerical_results}
\graphicspath{ {./numerical_example} }

In this section, we describe some numerical simulations\footnote{Simulations are performed in Python 3.10 using Dolfinx v.0.10.0.post2 \cite{Baratta2023,Alnaes2014}, petsc4py v.3.15.1 \cite{Balay2025}, and slepc4py v.3.15.1 \cite{Hernandez2005,Dalcin2011}. Images are made using Paraview v.6.0.1 \cite{Paraview}.} for the inverse problem described in \Cref{sec:heat_equation_example_formulation_and_discretisation} and analyse the numerical results. We choose specific values for the constants in \Cref{sec:heat_equation_example_formulation_and_discretisation}, and the experiments we run are described in \Cref{sec:experiment_description}. The results of these experiments are presented in \Cref{sec:posterior_information,sec:spectral_decay,sec:optimal_approximations_for_varying_rank,sec:perturbed_optimal_approximations}.

\subsubsection{Experiment description}
\label{sec:experiment_description}
The physical domain is chosen to be the unit square $\mathcal{D}=(0,1)^2$ with boundary $\partial\mathcal{D}=[0,1]^2\setminus{(0,1)^2}$. For the prior, we take $\alpha=2$, $\Theta=1$ and $b=1$. That is, the application of the square root of the prior covariance $\mathcal{C}_\pr^{1/2}$ to a vector $h\in\H$ is given by the solution $v\in H^1_0(\mathcal{D})$ of the elliptic PDE $(-\Delta+I)v = h$ with homogeneous Dirichlet boundary conditions. For the forward problem, we take $T=1.5\cdot 10^{-3}$ as the final time at which the observations are made. We choose $n=400$ observables $\psi_i=\abs{B_{\delta}(s^i)}^{-1}1_{B_\delta(s^i)}, 1\leq i\leq n$, with centers $(s^i)_i$ that are uniformly spaced inside of $\mathcal{D}$, as shown in \Cref{fig:observables}. The radius $\delta=0.02$ is small enough such that the supports of the $(\psi_i)_i$ do not overlap. We use a true parameter value $x^\dagger$ given by 
\begin{align*}
	x^\dagger(s_1,s_2) \coloneqq 1_{\{s_1 + s_2 > 1.3\}} + 0.2  \sin(3\pi s_1)\sin(2\pi s_2),
\end{align*}
shown in \Cref{fig:ground_truth}. Thus, $x^\dagger$ is the sum of a discontinuous function with nonzero boundary conditions and a smooth function vanishing at the boundary. The noiseless data $Gx^\dagger$ is discretised using a finer discretisation $(d,\Delta{t})=(10^6,10^{-6})$ than used for our experiments to address the issue of inverse crimes, c.f.\ \cite[Section 1.2]{Kaipio2005}. A data vector $y=Gx^\dagger+\zeta^\dagger$ is generated using a random draw $\zeta^\dagger$ from the noise distribution $\mathcal{N}(0,\mathcal{C}_\obs)$. Here, the covariance $\mathcal{C}_\obs$ is a randomly chosen self-adjoint and positive matrix. Furthermore, $\mathcal{C}_\obs$ is scaled in such a way that draws from the noise distribution are of slightly smaller order than the noiseless data $G x^\dagger$ corresponding to the ground truth $x^\dagger$, i.e.\ $\tr{\mathcal{C}_\obs}\approx\norm{Gx^\dagger}/10$.

\begin{figure}
	\centering
	\begin{subfigure}[t]{0.32\textwidth}
		\centering
		\includegraphics[width=\linewidth]{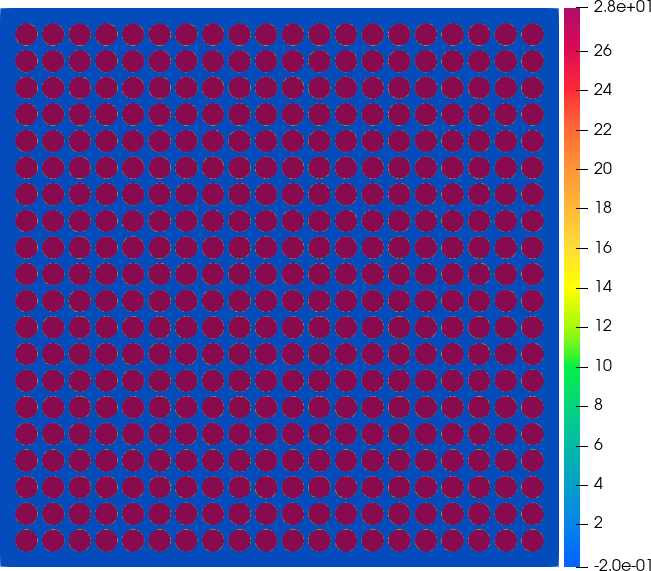}
		\caption{Observables ($n=400$)}
		\label{fig:observables}
	\end{subfigure}
	\begin{subfigure}[t]{0.32\textwidth}
		\centering
		\includegraphics[width=\linewidth]{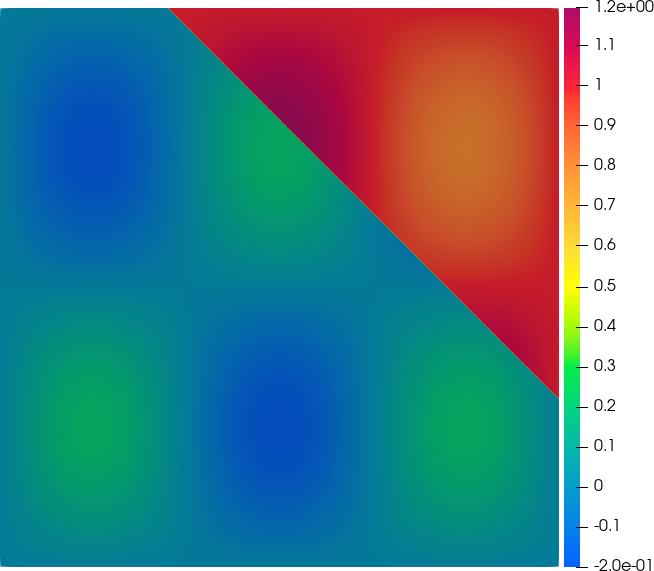}
		\caption{Ground truth}
		\label{fig:ground_truth}
	\end{subfigure}
	\begin{subfigure}[t]{0.32\textwidth}
		\centering
		\includegraphics[width=\linewidth]{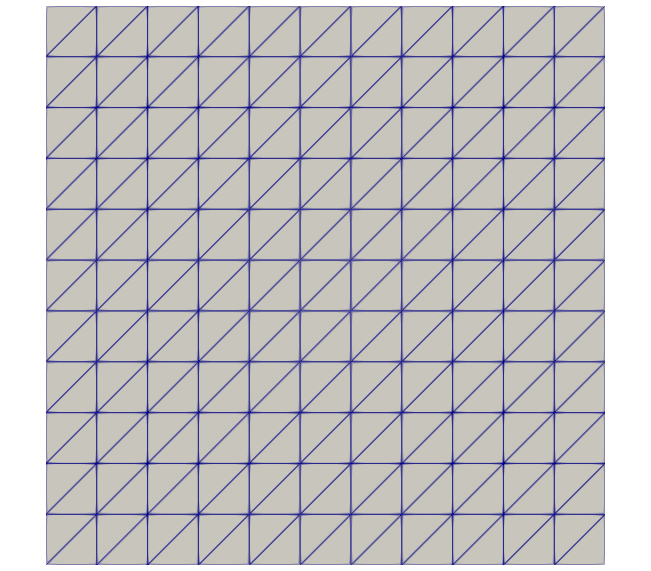}
		\caption{Example mesh for $d=10^2$}
		\label{fig:mesh}
	\end{subfigure}
	\caption{Experiment setup. (a) The observables $\psi_i$ for $i=1,\ldots,n=400$. (b) The chosen ground truth $x^\dagger$ for the initial condition of the heat equation. (c) A triangulated mesh corresponding to $d=10^2$. }
	\label{fig:experiment_setup}
\end{figure}

In our experiments, the dimension $d$ of the finite element space $\mathcal{V}_d\subset H^1_0(\mathcal{D})$ and the time step $\Delta t$ of the Crank--Nicolson discretisation are varied. We choose a triangular mesh, see \Cref{fig:mesh} for an example for $d=10^2$. Thus, the mesh size $h$ of the spatial discretisation is related to $d$ via $h = \sqrt{2}/(\sqrt{d}+1)$.
We shall approximate the limit $(d,\Delta t)\rightarrow(\infty,0)$ by increasing spatial and temporal refinement levels, and then compare the discretised posteriors corresponding to different values of $(d,\Delta{t})$. Furthermore, we shall also fix the refinement level and instead vary the truncation parameter $r$. 


In the discretisation of the heat equation $\mathcal{F}$, we relate the choice of $\Delta{t}$ to the choice of $d$. By \cite[Theorem 7.7]{Thomee2006} applied to a temporal discretisation with the Crank--Nicolson scheme and to a spatial discretisation with piecewise linear Lagrangian finite elements, which are are second-order accurate in time and space respectively, we have the error bound $\norm{u(T)-u_h(T)} \leq c_1 h^2 + c_2\Delta{t}^2.$ The constants $c_1$ and $c_2$ depend on $T^{-1}$ and on the $\H$-norm of the initial condition. Balancing the spatial and time discretisation errors by choosing $\Delta{t}^2/h^2$ constant will therefore control the error in $L^2(\mathcal{D})$ norm. However, even with this choice of $\Delta{t}$, the Crank--Nicolson discretisation in time is known to result in oscillatory behaviour of the numerical solution, for nonsmooth initial conditions and at small times, see \cite{Osterby2003} and the third bullet point in \cite[Section 9.9]{LeVeque2007}. To mitigate this oscillatory behaviour of the discretised solution for small $t$, we choose a smaller time step, as suggested by \cite{Osterby2003}. We choose $d\Delta{t}=O(1)$. The choice $\Delta{t} = T / \max(1,\lfloor Td\rfloor)$ ensures $\Delta{t}\leq T$, $T/\Delta{t}\in\N$, and $d\Delta{t}=O(1)$. Since $\Delta{t}$ is now chosen as a function of $d$, we simply write $\mu_{\pos,d}$ instead of $\mu_{\pos,(d,\Delta{t})}$, and similarly for any other discretised quantities, c.f.\ \Cref{sec:heat_equation_example_formulation_and_discretisation}.

The reason that choosing a smaller time step size according to $\Delta{t}=O(d^{-1})=O(h^2)$ eliminates oscillatory behaviour near $t=0$ can be seen as follows. The solution of \eqref{eqn:heat_equation} with initial condition $x^\dagger$ is given by $\exp(t\Delta)x^\dagger$, see the details on \Cref{ex:heat_equation} in \Cref{sec:examples}. Denoting an ONB of eigenfunctions of $-\Delta$ by $(e_i)_i$ with corresponding eigenvalues $(a_i)$, it holds that the finite element discretisation in space allows only for those $e_i$ to be resolved that have sufficiently large corresponding eigenvalue $a_i$ compared to the mesh size $h$. Such eigenfunctions with large eigenvalue exhibit high frequency oscillations, and because $x^\dagger$ is not smooth, some of these will contribute significantly when decomposing $x^\dagger$ in the ONB $(e_i)_i$. A Crank--Nicolson discretisation in time evolves such eigenfunctions $e_i$ as $\mathscr{R}(a_i \Delta{t})^ke_i$ at time $k\Delta{t}$, $k\in\N$. Here, $\mathscr{R}(z) \coloneqq (1-z/2)/(1+z/2)$, $z\in\R$, and it holds that $\mathscr{R}(a_i\Delta{t})\approx -1$ if $a_i\Delta{t}$ is large, while $\mathscr{R}(a_i\Delta{t})^k\approx \exp(-a_ik\Delta{t})$ for $a_i\Delta{t}$ small. Therefore, denoting by $a_{d,\max}$ the largest resolved eigenvalue of a finite element approximation of $-\Delta$ using the approximation space $\mathcal{V}_d$, we shall require $a_{d,\max}\Delta{t}$ to be $O(1)$. The inverse estimate $\norm{\nabla{\xi}}^2\leq Ch^{-2}\norm{\xi}^2$, $\xi\in\mathcal{V}_d$, valid for some $C>0$, see \cite[eq.\ (1.12)]{Thomee2006}, yields together with the Courant-Fisher min-max principle \cite[eq.\ (4.13)]{Hsing2015} and an integration by parts, 
\begin{align*}
	a_{d,\max} = \max_{\xi\in\mathcal{V}_d}\frac{\langle -\Delta \xi,\xi\rangle}{\norm{\xi}^2} = \max_{\xi\in\mathcal{V}_d}\frac{\langle \nabla \xi,\nabla\xi\rangle}{\norm{\xi}^2} \leq Ch^{-2}.
\end{align*}
It follows that $a_{d,\max} = O(h^{-2})$, which leads to the choice $\Delta{t} = O(a_{d,\max}^{-1}) = O(h^2)$.

For the discretisation of the observation operator $\mathcal{O}$, integrals against $\psi_i$, $1\leq i\leq n$, are computed via quadrature as suggested in \cite[Example 2.4]{Sanz-Alonso2024}. Such quadrature is accurate with small quadrature degrees (e.g.\ 4) if $\delta$ is not too small compared to $h$. Since we are interested in discretisations for large $d$, this is the case for most of our purposes. However, when we do require $\delta/h\leq 1$, then we increase the quadrature degree, so that we can also approximate the observation operator $\mathcal{O}_d$ with reasonable accuracy for coarse meshes, i.e.\ for relatively small $d$.

The experiments we perform are the following:
\begin{enumerate}
	\item (Posterior information) For a fixed discretisation level $(d,\Delta{t})$, we examine the exact and approximate posterior distributions, by drawing from these distributions.
	\item (Spectral decay) For increasingly fine discretisation levels, we investigate the spectral decay of the operator $R(\mathcal{C}_{\pos,d}\Vert\mathcal{C}_{\pr,d})$ as defined in \eqref{eqn:feldman_hajek_operator}. 	
	\item (Optimal approximations for varying rank) For a fixed discretisation level and increasing values of $r$, we compare reverse KL divergences of Gaussians with identical covariance and with either
		\begin{enumerate}
			\item the full posterior mean,
			\item the optimal structure-ignoring low-rank posterior mean approximation of \Cref{thm:optimal_low_rank_mean_approx},
			\item the optimal structure-preserving low-rank posterior mean approximation of \Cref{thm:optimal_low_rank_update_mean_approx},
			\item the posterior mean of the projected inverse problem of \Cref{prop:optimal_projector}.
		\end{enumerate}
	\item (Perturbed optimal approximations) For increasingly fine discretisation levels, we compare the approximation quality of Gaussians with the posterior covariance and with either
		\begin{enumerate}
			\item the optimal structure-ignoring low-rank posterior mean approximation of \Cref{thm:optimal_low_rank_mean_approx},
			\item a perturbed low-rank posterior mean approximation that lies in the Cameron--Martin space,
			\item a perturbed low-rank posterior mean approximation that lies outside of the Cameron--Martin space.
		\end{enumerate}
\end{enumerate}

\subsubsection{Posterior information}
\label{sec:posterior_information}

For $(d,\Delta{t})=(10^4,10^{-4})$, we first consider draws from the exact and approximate posteriors.
\Cref{thm:opt_covariance_and_precision} suggests a computationally efficient method to approximate draws from the posterior. To motivate this, notice that by \Cref{thm:opt_covariance_and_precision},
\begin{align*}
	\mathcal{C}^\opt_r = \mathcal{C}_\pr - \sum_{i=1}^{r}(-\lambda_i)(\mathcal{C}_\pr^{1/2} w_i)\otimes (\mathcal{C}_\pr^{1/2} w_i) = LL^*,\quad
	L \coloneqq \mathcal{C}_\pr^{1/2}\left( I + \sum_{i=1}^{r}\left(\sqrt{\lambda_i+1}-1\right)w_i\otimes w_i \right).
\end{align*}
This also holds in the discretised setting. Thus, if $v\sim\mathcal{N}(0,I)$ in $\mathcal{V}_d$, then by \Cref{lemma:linear_push_forward_gaussian} $\widehat{v}\coloneqq\mathcal{C}_{\pr,d}^{1/2}v\sim \mathcal{N}(0,\mathcal{C}_{\pr,d})$ and $\widehat{v} + \sum_{i=1}^{r}(\sqrt{\lambda_i+1}-1)\langle \mathcal{C}_{\pr,d}^{-1/2}w_i,\widehat{v}\rangle\mathcal{C}_{\pr,d}^{1/2}w_i = Lv\sim\mathcal{N}(0,\mathcal{C}^\opt_{r,d})$. Thus, we can draw from $\mathcal{N}(0,\mathcal{C}^\opt_{r,d})$ by drawing $\widehat{v}\sim \mathcal{N}(0,\mathcal{C}_{\pr,d})$ in $\mathcal{V}_d$ and then updating $\widehat{v}$ in the $r$ directions $\mathcal{C}_{\pr,d}^{1/2}w_i$, which can be precomputed and stored. Drawing from the discretised prior can be done, for example, by a possibly truncated Karhunen-Lo\`{e}ve expansion, see \cite[Theorem 6.19]{Stuart2010}. Given the discretised optimal rank-$r$ posterior mean approximation $m_{\pos,r,d}^{\opt,(i)}(y)$ for $i=1,2$, the sum $m_{\pos,r,d}^{\opt,(i)}(y)+Lv$ then yields a draw from $\mu_{\pos,r,d}^{\opt,(i)}(y)$. Setting $r\leftarrow n$ yields draws from the full discretised posterior. 

Using this method, we draw the posterior samples shown in \Cref{fig:posterior}. The required draws from the prior are made using a truncated Karhunen--Lo\`eve expansion, where we truncate after 1000 terms. \Cref{fig:pos_draw} shows a draw from the full posterior $\mu_{\pos,d}(y)$ and \Cref{fig:pos_draw_low_rank_sp} and \Cref{fig:pos_draw_low_rank_si} show draws from the optimal rank-$r$ posterior approximations $\mu_{\pos,r,d}^{\opt,(1)}(y)$ and $\mu_{\pos,r,d}^{\opt,(2)}(y)$ respectively, for $r=20$. With $r=20$, only a 20-dimensional update of the prior mean and covariance is performed, in a $10^5$-dimensional approximate parameter space. The structure-ignoring posterior mean approximation appears to yield a better approximation than the structure-preserving posterior mean approximation, and in fact it appears to represent the exact posterior draw relatively well. 

\begin{figure}
	\centering
	\begin{subfigure}[t]{0.32\textwidth}
		\centering
		\includegraphics[width=\linewidth]{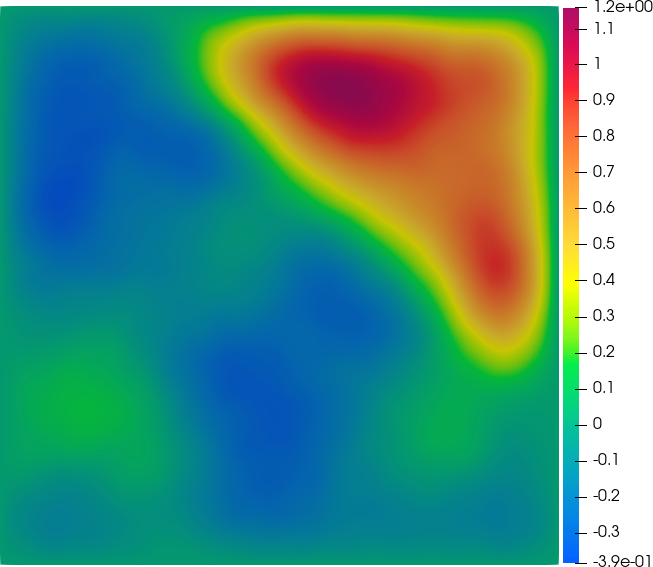}
		\caption{Draw from $\mu_{\pos,d}(y)$}
		\label{fig:pos_draw}
	\end{subfigure}
	\begin{subfigure}[t]{0.32\textwidth}
		\centering
		\includegraphics[width=\linewidth]{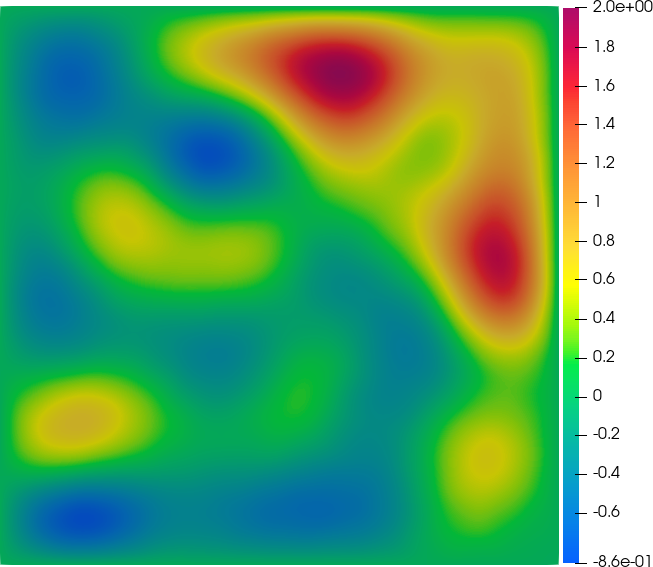}
		\caption{Draw from $\mu_{\pos,r,d}^{\opt,(1)}(y)$}
		\label{fig:pos_draw_low_rank_sp}
	\end{subfigure}
	\begin{subfigure}[t]{0.32\textwidth}
		\centering
		\includegraphics[width=\linewidth]{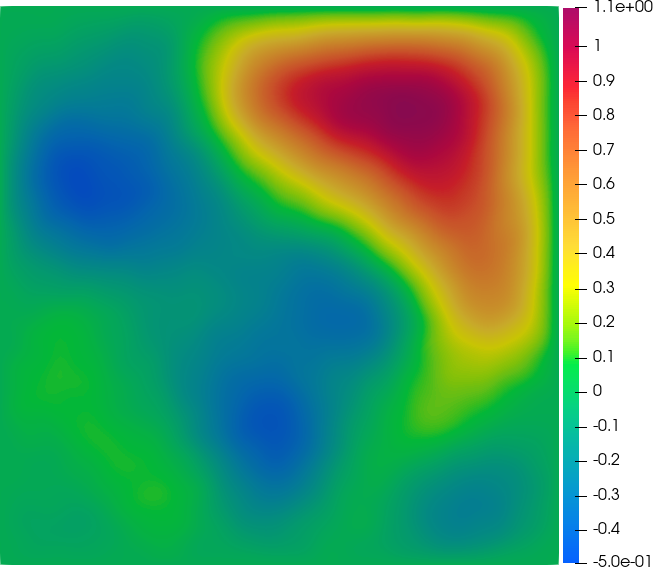}
		\caption{Draw from $\mu_{\pos,r,d}^{\opt,(2)}(y)$}
		\label{fig:pos_draw_low_rank_si}
	\end{subfigure}
	\caption{Posterior draws corresponding to $(d,\Delta{t})=(10^4,10^{-4})$. 
		(a) A draw from the full posterior distribution. 
		(b) A draw from the optimal rank-$r$ posterior distribution with structure-preserving mean with $r=20$.
		(c) A draw from the optimal rank-$r$ posterior distribution with structure-ignoring mean with $r=20$.}
	\label{fig:posterior}
\end{figure}

\subsubsection{Spectral decay}
\label{sec:spectral_decay}

Next, we turn to the low-rank behaviour of $R(\mathcal{C}_\pos\Vert\mathcal{C}_\pr)$ with $R(\cdot\Vert\cdot)$ defined in \eqref{eqn:feldman_hajek_operator}, the operator occurring in the Feldman--Hajek theorem with spectrum in $(-1,0]$, c.f.\ \Cref{prop:bayesian_feldman_hajek}. 
\Cref{fig:spectra_n_400_log} and \Cref{fig:spectra_n_400_first} show the leading part of the spectrum of four discretised versions $-R(\mathcal{C}_{\pos,d}\Vert\mathcal{C}_{\pr,d})$ of $-R(\mathcal{C}_{\pos}\Vert\mathcal{C}_{\pr})$, each corresponding to a different discretisation level $(d,\Delta{t})$. \Cref{fig:spectra} suggests that the spectra of $R(\mathcal{C}_{\pos,d}\Vert\mathcal{C}_{\pr,d})$ become independent of the discretisation for sufficiently fine discretisation, and thus approach the spectrum of the infinite-dimensional formulation of the inverse problem, which is necessary for the numerical consistency of the finite-dimensional posterior distributions $\mu^{\opt,(i)}_{\pos,r,d}\rightarrow \mu^{\opt,(i)}_{\pos,r}$, $i=1,2$. The coarsest discretisation $(d,\Delta{t}) = (10^2, 1.5\cdot 10^{-3})$ seems only to capture the first 50 eigenvalues. This can be both due to too coarse a discretisation or a poor performance of the quadrature used in the computation of the observation operator $\mathcal{O}_d$. We also see that this spectrum is near zero for indices larger than $r=70$, thereby confirming numerically that low-rank behaviour occurs in the infinite-dimensional formulation of the inverse problem. This low-rank behaviour then allows one to construct qualitatively good low-rank approximations via \Cref{thm:opt_covariance_and_precision,thm:optimal_low_rank_mean_approx,thm:optimal_low_rank_update_mean_approx} and \Cref{prop:optimal_joint_approximation,prop:optimal_projector}.

\begin{figure}
	\centering
	\begin{subfigure}[t]{0.48\textwidth}
		\centering
		\includegraphics[width=\linewidth]{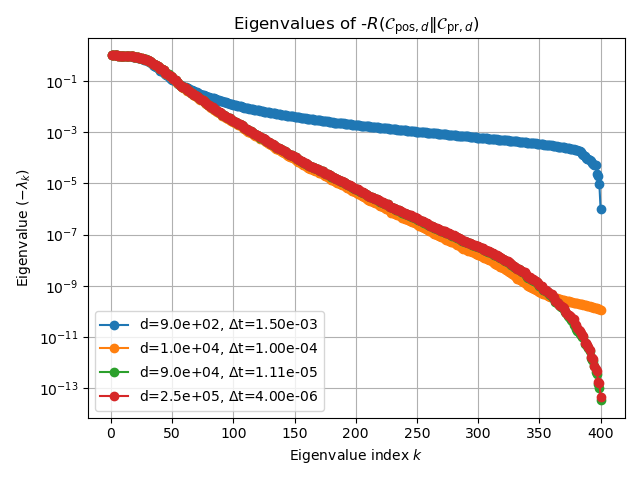}
		\caption{All 400 nonzero eigenvalues of $-R(\mathcal{C}_{\pos,d} || \mathcal{C}_{\pr,d})$}
		\label{fig:spectra_n_400_log}
	\end{subfigure}\hfill
	\begin{subfigure}[t]{0.48\textwidth}
		\centering
		\includegraphics[width=\linewidth]{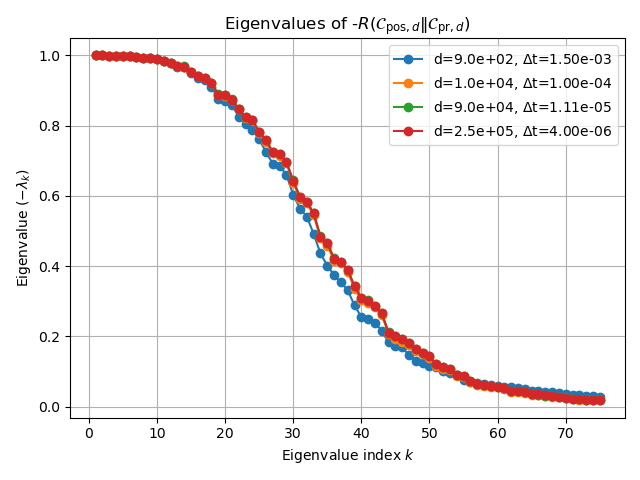}
		\caption{First 75 nonzero eigenvalues of $-R(\mathcal{C}_{\pos,d} || \mathcal{C}_{\pr,d})$}
		\label{fig:spectra_n_400_first}
	\end{subfigure}
	\caption{Spectral decay of different discretisations of the negative Feldman--Hajek operator $-R(\mathcal{C}_\pos\Vert\mathcal{C}_\pr)$ for data dimension $n=400$. (a) Log-linear plot of all nonzero eigenvalues. (b) Linear-linear plot of first 75 eigenvalues.}
	\label{fig:spectra}
\end{figure}

\subsubsection{Optimal approximations for varying rank}
\label{sec:optimal_approximations_for_varying_rank}

\Cref{prop:optimal_projector} states that the optimal low-rank posterior approximation $\mu_{\pos,r}^{\opt,(2)}$ with structure-ignoring posterior mean approximation corresponds to the exact posterior $\mu_{P^\opt_r,\pos}$ of a projected inverse problem. This must hold in particular for any discretisation of the inverse problem. To verify numerically that indeed $\mu_{\pos,r,d}^{\opt,(2)}=\mu_{P^\opt_r,\pos,d}$ holds, we fix a discretisation level $(d,\Delta{t})=(9\cdot 10^2,1.5\cdot 10^{-3})$ and use Monte Carlo sampling with 100 samples from the distribution of the data $Y$ to approximate certain data-averaged KL divergences. We recall that $Y$ has distribution $\mathcal{N}(0,\mathcal{C}_{\y,d})$, where the covariance $\mathcal{C}_{\y,d}$ is defined in \eqref{eqn:data_covariance} with $G$ replaced by $G_d$ and can be represented as a matrix in $\R^{n\times n}$. Note that the choice $\Delta{t}=1.5\cdot 10^{-3}=T$ implies that only one time step is used in the discretisation scheme.

Monte Carlo approximations of the data-averaged KL divergences $\mathbb{E}[D_{\kl}(\mu^{\opt,(2)}_{\pos,r,d}(Y)\Vert \mu_{\pos,d}(Y))]$, $\mathbb{E}[D_{\kl}(\mu_{P^\opt_r,\pos,d}(Y)\Vert \mu_{\pos,d}(Y))]$, and $\mathbb{E}[D_{\kl}(\mu_{P^\opt_r,\pos,d}(Y)\Vert \mu^{\opt,(2)}_{\pos,r,d}(Y))]$, are shown in \Cref{fig:kl_vs_r_full_vs_approx} as a function of $r$. The curves of $\mathbb{E}[D_{\kl}(\mu^{\opt,(2)}_{\pos,r,d}(Y)\Vert \mu_{\pos,d}(Y))]$ and $\mathbb{E}[D_{\kl}(\mu_{P^\opt_r,\pos,d}(Y)\Vert \mu_{\pos,d}(Y))]$ overlap, and $\mathbb{E}[D_{\kl}(\mu_{P^\opt_r,\pos,d}(Y)\Vert \mu^{\opt,(2)}_{\pos,r,d}(Y))]$ is of the order of numerical error for all $r$. Since the KL divergence is nonnegative and vanishes only between identical measures, this is consistent with the assertion that $\mu^{\opt,(2)}_{\pos,r,d}(y)=\mu_{P^\opt_r,\pos,d}(y)$ holds for all realisations $y$ of $Y$ in a set of probability 1. This verifies the statement $\mu^{\opt,(2)}_{\pos,r,d}=\mu_{P^\opt_r,\pos,d}$ implied by \Cref{prop:optimal_projector}. 

\Cref{fig:kl_vs_r_full_vs_approx} also shows that the average reverse KL divergence is around five orders of magnitude smaller when using the posterior approximation $\mu^{\opt,(2)}_{\pos,r,d}$ with $r\approx 50$ compared to using $r=0$, i.e.\ compared to using the prior to approximate the posterior. In \Cref{fig:kl_vs_r_si_vs_sp}, we compare the performance of $\mu^{\opt,(i)}_{\pos,r,d}$ for $i=1$ (structure-preserving) and $i=2$ (structure-ignoring). We see that for $r<28$ the structure-ignoring approximation performs better, while for $r\geq 28$ the structure-preserving approximation performs slightly better. This is consistent with \Cref{fig:spectra_n_400_first} and the discussion after \Cref{cor:optimal_mean_for_amari_and_hellinger}, which predict that the optimal structure-preserving mean approximation is better for $r\geq 33$, since $-\lambda_i<\frac{1}{2}$ for all $i\geq 33$.

\begin{figure}
	\centering
	\begin{subfigure}[t]{0.48\textwidth}
		\centering
		\includegraphics[width=\linewidth]{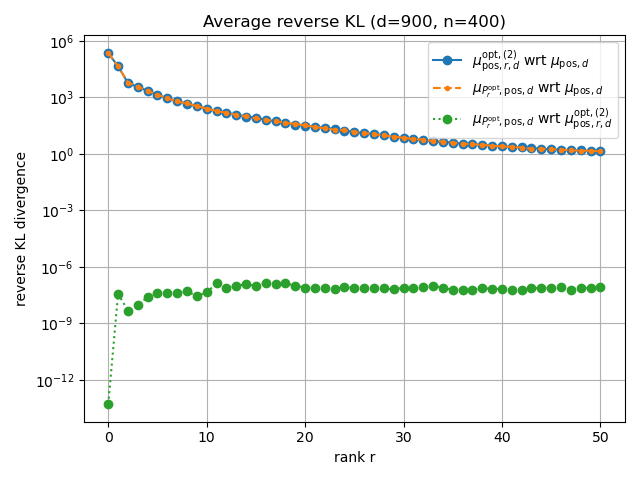}
		\caption{KL divergence-based comparison of structure-ignoring posterior and projection-based posterior}
		\label{fig:kl_vs_r_full_vs_approx}
	\end{subfigure}\hfill
	\begin{subfigure}[t]{0.48\textwidth}
		\centering
		\includegraphics[width=\linewidth]{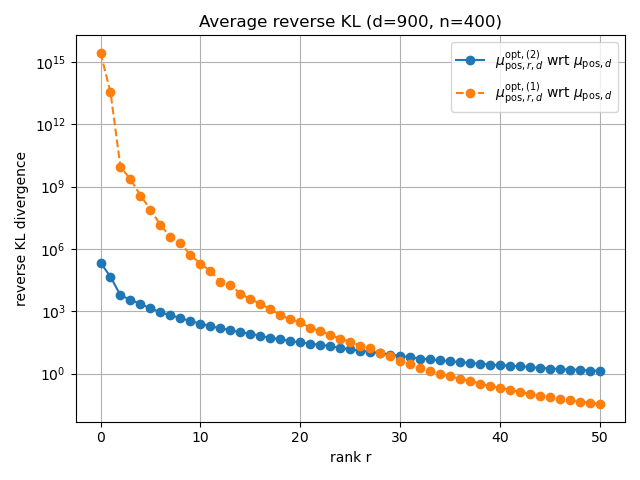}
		\caption{KL divergence-based comparison of structure-ignoring posterior and structure-preserving posterior}
		\label{fig:kl_vs_r_si_vs_sp}
	\end{subfigure}
	\caption{Monte Carlo averages of KL divergences computed using 100 samples of $Y$ versus the truncation parameter $r=0,\ldots,50$, at discretisation level $(d,\Delta{t})=(9\cdot 10^2, 1.5\cdot 10^{-3})$. 
		(a) The divergences shown are $\mathbb{E}[D_{\kl}(\mu_{\pos,r,d}^{\opt,(2)}(Y)\Vert \mu_\pos(Y))]$, $\mathbb{E}[D_{\kl}(\mu_{P^\opt_r,\pos,d}(Y)\Vert \mu_\pos(Y))]$ and $\mathbb{E}[D_{\kl}(\mu_{P^\opt_r,\pos,d}(Y)\Vert \mu_{\pos,r,d}^{\opt,(2)}(Y))]$. 
		(b) The divergences shown are $\mathbb{E}[D_{\kl}(\mu_{\pos,r,d}^{\opt,(i)}(Y)\Vert \mu_\pos(Y))]$ for $i=1$ (structure-preserving) and $i=2$ (structure-ignoring).}
	\label{fig:kl_vs_r}
\end{figure}

\subsubsection{Perturbed optimal approximations}
\label{sec:perturbed_optimal_approximations}

Finally, we compare the posterior mean $m_{\pos}(Y)$ with perturbations $m_{\pos,r}^{(2),\omega} (Y)$ of the optimal structure-ignoring rank-$r$ posterior mean approximation $m_{\pos,r}^{\opt,(2)}(Y)$, where $\omega\in\H$ denotes a vector that we shall use to generate perturbations. For this, we consider discretisations $m_{\pos,d}$ of $m_{\pos}$ and $m_{\pos,r,d}^{(2),\omega}$ of $m_{\pos,r}^{(2),\omega}$. Then, we can compute the average Cameron--Martin norm 
\begin{align}
	\label{eqn:discretised_mean_loss}
	\mathbb{E}\Norm{\mathcal{C}_{\pos,d}^{-1/2} \left(m_{\pos,r,d}^{(2),\omega}(Y) - m_{\pos,d}(Y)\right)}^2.
\end{align}
By \eqref{eqn:divergences}, this average Cameron--Martin norm is equal to the average reverse KL divergence between $\mathcal{N}(m_{\pos,d}(Y),\mathcal{C}_{\pos,d})$ and $\mathcal{N}(m_{\pos,r,d}^{(2),\omega}(Y),\mathcal{C}_{\pos,d})$, and also to the average forward KL divergences and average R\'enyi-$\rho$ divergences between said Gaussians, for $\rho\in(0,1)$.

We shall consider perturbations that with positive probability yield mutually singular approximations of both the exact posterior $\mu_\pos(Y)$ and the Gaussian with approximated posterior mean $\mathcal{N}(A_r^{\opt,(2)}Y,\mathcal{C}_\pos)$. This is possible, since $\H$ is not finite-dimensional. To do this, we consider a vector $\omega\in\H$ and perturb the optimal rank-$r$ posterior mean given in \Cref{thm:optimal_low_rank_mean_approx} by $\sqrt{-\lambda_i/(1+\lambda_i)}\langle\mathcal{C}_\obs^{-1/2}\varphi_1,Y\rangle \omega$. The resulting perturbed posterior mean approximation $m_{\pos}^{r,(2),\omega}$ is still a rank-$r$ linear transformation of the data, and is given by
\begin{align}
	\nonumber
	m_{\pos}^{r,(2),\omega}(y) &\coloneqq A^\omega_ry,\quad y\in\R^n,\\
	\label{eqn:perturbed_low_rank_mean_approximation}
	A^\omega_r &\coloneqq \sum_{i=1}^{r} \sqrt{\frac{-\lambda_i}{1+\lambda_i}}(\mathcal{C}_\pr^{1/2}w_i) \otimes (\mathcal{C}_\obs^{-1/2}\varphi_i) + \sqrt{\frac{-\lambda_1}{1+\lambda_1}}\omega\otimes (\mathcal{C}_\obs^{-1/2}\varphi_1).
\end{align}
As mentioned in the paragraph above, the posterior covariance $\mathcal{C}_{\pos,d}$ is not perturbed. We make four choices of $\omega$:
\begin{enumerate}
	\item
		$\omega=0$, so that $m_\pos^{r,(2),\omega} = m_{\pos}^{r,(2)}$,
	\item
		$\omega=1$, so that while $\omega$ is smooth, it does not satisfy the Dirichlet boundary conditions and hence $m_\pos^{r,(2),\omega}\not\in H^1_0(\mathcal{D})$,
	\item
		$\omega(s)=\textup{dist}(s,\partial\mathcal{D})^{\beta}$ for some $0<\beta<\frac{1}{2}$, so that while $\omega$ satisfies the Dirichlet boundary conditions, it is not sufficiently smooth and hence $m_\pos^{r,(2),\omega}\not\in H^1_0(\mathcal{D})$,
	\item
		$\omega(s) = \sin(\pi s_0)\sin(\pi s_1)$, so that $\omega\in H^1_0(\mathcal{D})$.
\end{enumerate}
Note that $s\mapsto \textup{dist}(s,\partial\mathcal{D})^\beta$ is equal to $s_1^\beta$ on $\mathcal{D}_0\coloneqq\{s\in\mathcal{D}:\ s_1<s_2,1-s_1>s_2\}$, and its derivative on $\mathcal{D}_0$ has norm $\beta s_1^{\beta-1}$, which is not square integrable on $\mathcal{D}_0$ for $0<\beta<\frac{1}{2}$. Thus, $\norm{\textup{dist}(\cdot,\partial\mathcal{D})^\beta}_{ H^1(\mathcal{D})} \geq \norm{\textup{dist}(\cdot,\partial\mathcal{D})^\beta}_{ H^1(\mathcal{D}_0)}  =\infty$, showing that $s\mapsto \textup{dist}(s,\partial\mathcal{D})^\beta$ does indeed not belong to $H^1_0(\mathcal{D})$.

If $\omega\not\in H^1_0(\mathcal{D})$, then $\omega$ does not lie in the Cameron--Martin space $\ran{\mathcal{C}_\pos^{1/2}}$, since $\ran{\mathcal{C}_\pos^{1/2}}=\ran{\mathcal{C}_\pr^{1/2}}=\dom{(-\Delta+I)}  = H^2(\mathcal{D})\cap H^1_0(\mathcal{D})\subset H^1_0(\mathcal{D})$ as sets, c.f.\ \Cref{sec:heat_equation_example_formulation_and_discretisation}. By \Cref{prop:mean_approximation_problem_equivalence}\ref{item:mean_approximation_problem_equivalence_1}, $A^\omega_r\not\in \mathscr{M}_r^{(2)}$ for such choices of $\omega$, for $\mathscr{M}_r^{(2)}$ defined in \eqref{eqn:class_approx_mean_without_covariance_update}. By definition of $\mathscr{M}_r^{(2)}$, $\mathcal{N}(A^\omega_r y,\mathcal{C}^\opt_r)$ then is mutually singular with respect to $\mu_\pos(y)$ for $y$ in a set of positive probability under the distribution $\mathcal{N}(0,\mathcal{C}_\y)$ of $Y$, with $\mathcal{C}_\y$ defined in \eqref{eqn:data_covariance}. Hence $\mathbb{E}[D_{\kl}( \mathcal{N}(A^\omega_r Y,\mathcal{C}_\pos)\Vert \mu_\pos(Y))]=\infty$. After discretisation, this average reverse KL divergence becomes
\begin{align}
	\label{eqn:perturbed_reverse_kl_mean}
	\mathbb{E}[D_{\kl}( \mathcal{N}(A^\omega_{r,d} Y,\mathcal{C}_{\pos,d})\Vert \mu_{\pos,d}(Y))],
\end{align}
with $A^\omega_{r,d}$ the discretised version of $A^\omega_r$ defined in \eqref{eqn:perturbed_low_rank_mean_approximation}, and we recall that \eqref{eqn:perturbed_reverse_kl_mean} is equal to \eqref{eqn:discretised_mean_loss}.
The average reverse KL divergence \eqref{eqn:perturbed_reverse_kl_mean} in the discretised setting is finite, but should grow to infinity as the discretisation is refined. We verify this in \Cref{fig:kl_vs_d}, where for $r=30$ a Monte Carlo approximation of the expected reverse KL divergence for each of the four perturbations is shown. For the perturbations obtained with $\omega=1$ or with the distance function with $\beta=0.3$ (labeled ``$\omega=\text{dist}^\beta$''), we see that once the discretisation is fine enough to resolve high-frequency components, the average KL divergence indeed blows up when the discretisation is refined further. Instead, for the smooth perturbation (labeled ``$\omega=\sin(\pi s_0)\sin(\pi s_1)$''), the average KL divergence remains bounded from above as the discretisation is refined, and is bounded from below by the zero perturbation, which corresponds to the optimal choice of low-rank structure-ignoring posterior mean approximation (labeled ``$\omega=0$''). We thus see that even in finite dimensions, approximations must be discretisations of smooth enough functions in $L^2(\mathcal{D})$ that also satisfy the boundary conditions, for the approximation quality not to deteriorate under the refinement of discretisation.  This numerically verifies the importance of the infinite-dimensional Cameron--Martin space for constructing low-rank approximations, as was also identified in \Cref{prop:mean_approximation_problem_equivalence}\ref{item:mean_approximation_problem_equivalence_1} by relating the sets of admissible posterior mean approximations $\mathscr{M}^{(i)}_r$ defined in \eqref{eqn:class_approx_means} to this Cameron--Martin space.
\begin{figure}
	\centering
		\hspace{-2cm}
		\includegraphics[width=0.6\textwidth]{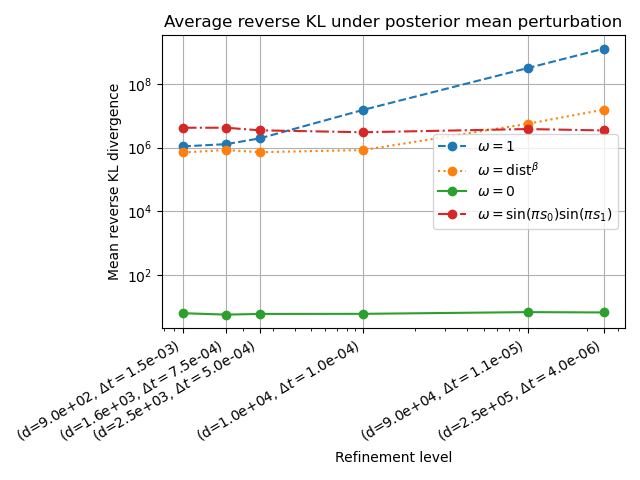}
		\caption{Monte Carlo averages of the KL divergence \eqref{eqn:perturbed_reverse_kl_mean} computed using 100 samples of $Y$ versus the discretisation level. The divergences shown are computed between the exact posterior mean and each of four choices of the perturbation $\omega$ of the optimal structure-ignoring rank-$r$ posterior mean, with $n=400$ and $r=30$.}
	\label{fig:kl_vs_d}
\end{figure}

\section{Conclusion}
\label{sec:conclusion}

This work considers low-rank approximations to linear Gaussian inverse problems on possibly infinite-dimensional separable Hilbert spaces. 
Numerical approximations for such problems transform them into finite-dimensional inverse problems, and optimal low-rank approximations in finite dimensions have been constructed in \cite{Spantini2015}. In order to show that numerical methods give optimal posterior approximations which are consistent with the infinite-dimensional formulation, one needs to formulate and find such optimal approximations on the infinite-dimensional space directly. To the best of our knowledge, the formulation and solution of these optimal approximation problems on infinite-dimensional spaces has not been addressed in the literature.

In this work, we have provided the formulation and solution of the low-rank posterior mean approximation problem directly on infinite-dimensional separable Hilbert spaces. We considered approximations that ignore and preserve the structure of the prior-to-posterior mean update in \Cref{thm:optimal_low_rank_mean_approx} and \Cref{thm:optimal_low_rank_update_mean_approx} respectively.
To quantify the posterior mean approximation quality, we have considered various loss classes. These loss classes consist of divergences between the exact Gaussian posterior and the approximate Gaussian posterior given by an approximate posterior mean and the exact posterior covariance, after averaging over the data distribution. The chosen divergences are the Hellinger distance and the R\'enyi, Amari, and forward and reverse KL divergences. 
These loss classes form a natural extension of the Bayes risk used in finite dimensions in \cite{Spantini2015}, and were used to assess optimality for low-rank approximations to the posterior covariance in \cite{PartI}.

The optimal low-rank posterior mean approximations satisfy the property that the resulting posterior distributions are equivalent to the exact posterior distribution, for any realisation of the data. The optimality of these low-rank posterior mean approximations holds for all of the structure-preserving and structure-ignoring posterior mean approximations which satisfy this equivalence property. Such approximations have been explicitly characterised in terms of range conditions on certain low-rank operators, as shown in \Cref{prop:mean_approximation_problem_equivalence}.

We have also provided a solution to the problem of finding optimal low-rank joint approximations of the posterior mean and covariance with respect to the average reverse KL divergence, using the results of \cite{PartI}, which considers separate posterior covariance approximation without posterior mean approximation. This joint problem is solved by combining the optimal mean approximation and the optimal covariance approximation, as shown in \Cref{prop:optimal_joint_approximation}. If the structure-ignoring posterior mean approximation is considered, we have shown in \Cref{prop:optimal_projector} that the solution to the joint approximation problem can equivalently be found by computing the exact posterior distribution of a linear Gaussian inverse problem with a projected forward model. This projected forward model involves a projection onto a low-dimensional subspace of the parameter space. This subspace is a one-to-one transformation of the subspace which contains the directions for which the ratio of posterior variance and prior variance is smallest, among all subspaces of the same dimension. The range of this projector was already studied in finite dimensions and is also known as the `likelihood-informed subspace'.

By solving the joint low-rank approximation problems and finding the corresponding optimal projection in parameter space, we have provided a perspective for the low-rank approximation problem that encompasses both mean and covariance simultaneously. Furthermore, since it is derived on the infinite-dimensional parameter space, we have shown that the optimal posterior approximation procedure is inherently discretisation independent and dimension independent.

\section*{Use of AI tools}

The large language models ChatGPT by OpenAI and Mistral Chat by Mistral AI were used only to assist in code development. No other AI tools were used in the creation of this manuscript.

\section*{Acknowledgements}
The research of the authors has been partially funded by the Deutsche Forschungsgemeinschaft (DFG) Project-ID 318763901 -- SFB1294. The authors thank Ricardo Baptista (California Institute of Technology) and Youssef Marzouk (Massachusetts Institute of Technology) for mentioning the joint approximation problem, Bernhard Stankewitz (University of Potsdam) for helpful discussions, Remo Kretschmann (University of Potsdam) for useful input on the PDE example, Thomas Mach (University of Potsdam) for constructive suggestions about the manuscript, and Francesco Romor (Weierstrass Institute) and Francesco Carere (Ghent University) for helpful suggestions on the numerical example.

\appendix

\section{Auxiliary results}
\label{sec:theoretical_facts}

In this section we collect some auxiliary results on Hilbert spaces and bounded operators, unbounded operators and Gaussian measures.

\subsection{Hilbert spaces and bounded operators}

\begin{lemma}[{\cite[Lemma A.1]{PartI}}]
	\label{lemma:extension_of_finite_basis_in_dense_subspace}
	Let $\mathcal{H}$ be a separable Hilbert space and $\mathcal{D}\subset\mathcal{H}$ be a dense subspace and $(e_i)_{i=1}^m$ be an orthonormal sequence in $\mathcal{D}$ for $m\in\N$. Then there exists a countable sequence $(d_i)_i\subset \mathcal{D}$ such that $(d_i)_i$ is an ONB of $\mathcal{H}$ and $d_i=e_i$ for $i\leq m$.
\end{lemma}




\begin{lemma}[{\cite[Lemma A.4]{PartI}}]
	\label{lemma:positive_is_injective_nonnegative}
	Let $\H$ be a Hilbert space and $A\in\B(\H)$. Then $A>0$ if and only if $A\geq 0$ and $A$ is injective.
\end{lemma}


\begin{lemma}[{\cite[Theorem 4.3.1]{Hsing2015}}]
	\label{lemma:operator_svd}
	Let $\H,\mathcal{K}$ be Hilbert spaces, and $A\in\B(\H,\mathcal{K})$ be compact. Then $A$ is diagonalisable, that is, there exists an ONB $(h_i)_i$ of $\H$ and an orthonormal sequence $(k_i)_i$ of $\mathcal{K}$ and a nonnegative and nonincreasing sequence $(\sigma_i)_i$ such that $A=\sum_{i}^{}\sigma_i k_i\otimes h_i$.
\end{lemma}

\begin{lemma}[{\cite[Proposition VI.1.8]{conway_course_2007}}]
	\label{lemma:kernel_range}
	Let $\mathcal{H}$, $\mathcal{K}$ be Hilbert spaces and $A\in\mathcal{B}(\mathcal{H},\mathcal{K})$. Then $\ker{A}=\ran{A^*}^\perp$ and $\ker{A}^\perp=\overline{\ran{A^*}}$.
\end{lemma}

\begin{lemma}[{\cite[Lemma A.7]{PartI}}]
	\label{lemma:kernel_of_square}
	Let $\mathcal{H}$ and $\mathcal{K}$ be Hilbert spaces and $A\in\mathcal{B}(\mathcal{H},\mathcal{K})$. Then $\ker{AA^*}=\ker{A^*}$.
\end{lemma}

%
\begin{lemma}[{\cite[Lemma A.8]{PartI}}]
	\label{lemma:range_of_square_of_finite_rank}
	Let $\mathcal{H}$, $\mathcal{K}$ be Hilbert spaces and $A\in\mathcal{B}_{00}(\mathcal{H},\mathcal{K})$. Then $\ran{AA^*}=\ran{A}$.
\end{lemma}


\begin{lemma}[{{\cite[Lemma A.9]{PartI}}}]
	\label{lemma:inverse_of_self_adj_hilbert_schmidt_perturbation}
	Let $\H$ be a Hilbert space, $(e_i)_i$ an orthonormal sequence, $(\delta_i)_i\in\ell^2(\R)$ and $T\coloneqq I+\sum_{i}^{}\delta_ie_i\otimes e_i$. The following holds. 
	\begin{enumerate}
		\item
			\label{item:inverse_of_self_adj_hilbert_schmidt_perturbation_1}
			$T$ is invertible in $\B(\H)$ if and only if $\delta_i\not=-1$ for all $i$.
		\item
			\label{item:inverse_of_self_adj_hilbert_schmidt_perturbation_2}
			$T\geq 0$ if and only if $\delta_i\geq -1$ for all $i$.
		\item
			\label{item:inverse_of_self_adj_hilbert_schmidt_perturbation_3}
			$T>0$ if and only if $\delta_i>-1$ for all $i$.
	\end{enumerate}
	In cases \ref{item:inverse_of_self_adj_hilbert_schmidt_perturbation_1} and \ref{item:inverse_of_self_adj_hilbert_schmidt_perturbation_3} above, the inverse of $T$ is $I-\sum_{i}^{}\frac{\delta_i}{1+\delta_i}e_i\otimes e_i$. 
\end{lemma}

\begin{lemma}
	\label{lemma:eigenpairs_of_squares_relation}
	Let $\H,\mathcal{K}$ be separable Hilbert spaces and $A\in\B(\H,\mathcal{K})$. Suppose $AA^*=\sum_{i}^{}\delta_i e_i\otimes e_i$ for $(e_i)_i$ an ONB of $\mathcal{K}$ and $(\delta_i)_i\subset[0,\infty)$ a nonincreasing sequence converging to 0.
	Then $(\delta_i,A^*e_i)$ is an eigenpair of $A^*A$.
\end{lemma}

\begin{proof}
	This follows from $A^*AA^*e_i=\delta_i A^*e_i$.
\end{proof}

\begin{lemma}
	\label{lemma:weakly_continuous_inner_products}
	Let $\mathcal{H}$ be a Hilbert space and $A\in\B_0(\mathcal{H})$. Then $h\mapsto\langle Ah,h\rangle$ is weakly continuous on $\mathcal{H}$.
\end{lemma}

\begin{proof}
	Suppose that $(h_n)_n\subset \mathcal{H}$ is weakly convergent with limit $h\in\mathcal{H}$, i.e.\ $\langle h_n,k\rangle\rightarrow\langle h,k\rangle $ for all $k\in\mathcal{H}$ as $n\rightarrow\infty$. In particular, $\langle Ah,h-h_n\rangle\rightarrow 0$. Since the sequence $(\langle h_n,k\rangle)_n$ is bounded for each $k\in\mathcal{H}$, the principle of uniform boundedness, c.f.\ \cite[Theorem III.14.3]{conway_course_2007}, implies that $(h_n)_n$ is a bounded sequence. By \cite[Theorem VI.11]{Reed1980}, $(Ah_n)_n$ converges in norm to $Ah$ since $A$ is compact. Thus, $\abs{\langle A(h-h_n),h_n\rangle}\leq \norm{A(h-h_n)}\sup_n\norm{h_n}\rightarrow 0$. We conclude that $\abs{\langle Ah,h\rangle - \langle Ah_n,h_n\rangle}\leq \abs{\langle A(h-h_n),h_n\rangle} + \abs{\langle Ah,h-h_n\rangle}\rightarrow 0$.
\end{proof}

\subsection{Unbounded operators}

\begin{definition}[{\cite[Definition X.1.5]{conway_course_2007}}]
	\label{def:unbounded_ajoint}
	Let $\H,\mathcal{K}$ be separable Hilbert spaces and $A:\H\rightarrow \mathcal{K}$ be a densely defined linear operator on $\H$. Then we define 
	\begin{align*}
		\dom{A^*}\coloneqq\{k\in \mathcal{K}:\ h\mapsto \langle Ah,k\rangle \text{ is a bounded linear functional on }\dom{A}\}.
	\end{align*}
	As $\dom{A}\subset \H$ is dense, if $k\in \mathcal{K}$, there exists by the Riesz representation theorem some $f\in \H$ such that $\langle Ah,k\rangle = \langle h,f\rangle $ for all $h\in \H$. We define $A^*:\dom{A^*}\rightarrow \H$ by setting $A^*k=f$.
\end{definition}

\begin{lemma}[{{\cite[Lemma A.19]{PartI}}}]
	\label{lemma:adjoint_of_densely_defined_operators}
	Let $\H$ be a separable Hilbert space. If $A,B,AB:\H\rightarrow \H$ are densely defined, then
	\begin{enumerate}
		\item 
			\label{item:adjoint_of_densely_defined_operators_1}
			$(AB)^*\supset B^*A^*$,
		\item 
			\label{item:adjoint_of_densely_defined_operators_2}
			If $B^*A^*$ is bounded, then $(AB)^*=B^*A^*$.
	\end{enumerate}
\end{lemma}

\begin{lemma}[{{\cite[Lemma A.23]{PartI}}}]
	\label{lemma:covariance_properties}
	Let $\mathcal{H}$ be a separable Hilbert space and $\mathcal{C}_1, \mathcal{C}_2\in L_1(\mathcal{H})_\R$ be nonnegative.
	If $\ran{\mathcal{C}_1^{1/2}}\subset\H$ densely, then the following hold.
	\begin{enumerate}
		\item $\mathcal{C}_1>0$ and $\mathcal{C}_1^{1/2}>0$.
			\label{item:covariance_properties_1}
		\item $\mathcal{C}_1^{-1/2}:\ran{\mathcal{C}_1^{1/2}}\rightarrow \H$ and $\mathcal{C}_1^{-1}:\ran{\mathcal{C}_1}\rightarrow \H$ are bijective and self-adjoint operators that are unbounded if $\dim{\H}$ is unbounded.
			\label{item:covariance_properties_2}
	\end{enumerate}
\end{lemma}


%
\begin{lemma}[{\cite[Lemma A.24]{PartI}}]
	\label{lemma:equivalent_cm_condition}
	Let $\H$ be a Hilbert space and $\mathcal{C}_1,\mathcal{C}_2\in\B(\H)$ be injective. Then $\ran{\mathcal{C}_1^{1/2}}=\ran{\mathcal{C}_2^{1/2}}$ if and only if $\mathcal{C}_2^{-1/2}\mathcal{C}_1^{1/2}$ is a well-defined invertible operator in $\B(\H)$.
\end{lemma}

\subsection{Gaussian measures on Hilbert spaces}

\begin{lemma}
	\label{lemma:gaussian_expected_norm_squared}
	Let $\H$ be a separable Hilbert space and $\mu=\mathcal{N}(0,\mathcal{C})$ be a Gaussian measure on $\H$. If $X\sim \mu$ and $\mathcal{C}=SS^*$ for $S\in L_2(\mathcal{H})$, then $\mathbb{E}\norm{X}^2 = \norm{S^*}^2_{L_2(\mathcal{H})} = \norm{S}_{L_2(\mathcal{H})}^2.$
\end{lemma}

\begin{proof}
	Let $(e_i)_i$ be an ONB of $\H$ and $X=\sum_{i}^{}\langle X,e_i\rangle e_i$. Then by Tonelli's theorem, the definition of the covariance operator, the hypothesis that $\mathcal{C}=SS^*$, and the invariance of the Hilbert--Schmidt norm under adjoints,
	\begin{align*}
		\mathbb{E}\norm{X}^2 &= \mathbb{E} \sum_{i}^{}\abs{\langle X,e_i\rangle}^2 = \sum_{i}^{}\mathbb{E}\abs{\langle X,e_i\rangle}^2 =\sum_{i}^{} \langle \mathcal{C}e_i,e_i\rangle = \sum_{i}^{}\norm{S^*e_i}^2 = \norm{S^*}_{L_2(\mathcal{H})}^2 = \norm{S}_{L_2(\mathcal{H})}^2.
	\end{align*}
\end{proof}

\begin{lemma}
	\label{lemma:linear_push_forward_gaussian}
	If $\mathcal{H}_1,\mathcal{H}_2$ are separable Hilbert spaces, $X\sim\mu=\mathcal{N}(m,\mathcal{C})$ is a Gaussian distribution on $\mathcal{H}_1$ and $A\in\B(\mathcal{H}_1,\mathcal{H}_2)$, then the distribution of $AX$ is $\mathcal{N}(Am,ACA^*)$.
\end{lemma}

\section{Proofs of results}
\label{sec:proofs_of_results}

\subsection{Proofs of Section \ref{sec:equivalence_and_divergences_between_gaussian_measures}}
\label{subsec:proofs_for_formulation}

\varianceReduction*
\begin{proof}[Proof of \Cref{prop:variance_reduction}]
	Applying $\mathcal{C}_\pos^{1/2}$ to both sides of the equation \eqref{eqn:bayesian_cov_pencil} implies $\mathcal{C}_\pos\mathcal{C}_\pr^{-1/2}w_i=(1+\lambda_i)\mathcal{C}_\pr^{1/2}w_i=(1+\lambda_i)\mathcal{C}_\pr(\mathcal{C}_\pr^{-1/2}w_i)$.
	Taking the inner product of both sides of the last equation with $\mathcal{C}_\pr^{-1/2}w_i$, we obtain the equality $\langle \mathcal{C}_\pos \mathcal{C}_\pr^{-1/2}w_i,\mathcal{C}_\pr^{-1/2}w_i\rangle =(1+\lambda_i) \langle \mathcal{C}_\pr\mathcal{C}_\pr^{-1/2}w_i,\mathcal{C}_\pr^{-1/2}w_i\rangle$.
	By \Cref{lemma:linear_push_forward_gaussian}, $\var_{X\sim\mu_\pos}(\langle X,z\rangle)=\langle \mathcal{C}_\pos z,z\rangle$ and $\var_{X\sim\mu_\pr}(\langle X,z\rangle)=\langle \mathcal{C}_\pr z,z\rangle$ for any $z\in\H$. Thus we obtain \eqref{eqn:relative_variance_reduction}.
	We now prove the final statement. It holds that $\ran{\mathcal{C}_\pr^{1/2}}=\ran{\mathcal{C}_\pos^{1/2}}$ by \Cref{thm:feldman--hajek}\ref{item:fh_ranges}. Then, by definition of the domain of compositions of unbounded operators, $\dom{\mathcal{C}_\pos^{1/2}\mathcal{C}_\pr^{-1/2}}=\dom{\mathcal{C}_\pr^{-1/2}}=\ran{\mathcal{C}_\pr^{1/2}}$. Furthermore, $\mathcal{C}_\pr^{-1/2}\mathcal{C}_\pos^{1/2}$ is a well-defined bounded operator on $\H$ by \Cref{lemma:equivalent_cm_condition}, and hence so is $(\mathcal{C}_\pr^{-1/2}\mathcal{C}_\pos^{1/2})^*$. We now apply $\mathcal{C}_\pr^{-1/2}\mathcal{C}_\pos^{1/2}$ to both sides of \eqref{eqn:bayesian_cov_pencil} and obtain
	\begin{align*}
		\mathcal{C}_{\pr}^{-1/2}\mathcal{C}_\pos^{1/2}\mathcal{C}_\pos^{1/2}\mathcal{C}_\pr^{-1/2}w_i=(1+\lambda_i)w_i.
	\end{align*}
	By \Cref{lemma:adjoint_of_densely_defined_operators}\ref{item:adjoint_of_densely_defined_operators_1}, $(\mathcal{C}_\pr^{-1/2}\mathcal{C}_\pos^{1/2})^*\in\B(\H)$ satisfies $(\mathcal{C}_\pr^{-1/2}\mathcal{C}_\pos^{1/2})^*w_i = \mathcal{C}_\pos^{1/2}\mathcal{C}_\pr^{-1/2}w_i$. The above display thus shows $I-\mathcal{C}_\pr^{-1/2}\mathcal{C}_\pos^{1/2}(\mathcal{C}_\pr^{-1/2}\mathcal{C}_\pos^{1/2})^* = \sum_{i}^{}-\lambda_i w_i\otimes w_i$. This is a nonnegative and compact operator, since $(-\lambda_i)_i\in\ell^2([0,1))$. Applying \cite[eq.\ (4.13)]{Hsing2015} to this operator, we get for any subspace $V_r\subset\ran{\mathcal{C}_\pr^{1/2}}$ of dimension $r$,
	\begin{align*}
		1+\max_{z\in V_{r}^\perp\setminus\{0\}}\frac{\langle-\mathcal{C}_\pr^{-1/2}\mathcal{C}_\pos^{1/2}(\mathcal{C}_\pr^{-1/2}\mathcal{C}_\pos^{1/2})^*z,z\rangle}{\norm{z}^2}
		=\max_{z\in V_{r}^\perp\setminus\{0\}}\frac{\langle I-\mathcal{C}_\pr^{-1/2}\mathcal{C}_\pos^{1/2}(\mathcal{C}_\pr^{-1/2}\mathcal{C}_\pos^{1/2})^*z,z\rangle}{\norm{z}^2}\geq -\lambda_{r+1},
	\end{align*}
	with equality for $V_{r}=\Span{w_1,\ldots,w_{r}}$. Using $\max_x -f(x)=-\min_x f(x)$ for any real-valued $f$, 
	\begin{align*}
		\min_{z\in V_{r}^\perp\setminus\{0\}}\frac{\langle \mathcal{C}_\pr^{-1/2}\mathcal{C}_\pos^{1/2}(\mathcal{C}_\pr^{-1/2}\mathcal{C}_\pos^{1/2})^*z,z\rangle}{\norm{z}^2}\leq 1+\lambda_{r+1},
	\end{align*}
	with equality for $V_{r}=\Span{w_1,\ldots,w_{r}}$. 
	Next, we show that, for any subspace $V_{r}\subset\ran{\mathcal{C}_\pr^{1/2}}$ of dimension $r$,
	\begin{align}
		\label{eqn:variance_reduction_over_dense_set}
		\min_{z\in V_{r}^\perp\setminus\{0\}} \frac{\langle \mathcal{C}_\pr^{-1/2}\mathcal{C}_\pos^{1/2}(\mathcal{C}_\pr^{-1/2}\mathcal{C}_\pos^{1/2})^*z,z\rangle}{\norm{z}^2}
		= \inf_{z\in (V_{r}^\perp\cap{\ran{\mathcal{C}_\pr^{1/2}}})\setminus\{0\}} \frac{\var_{X\sim\mu_\pos}(\langle X,\mathcal{C}_\pr^{-1/2}z\rangle)}{\var_{X\sim\mu_\pr}(\langle X,\mathcal{C}_\pr^{-1/2}z\rangle)}.
	\end{align}
	Let $v_1,\ldots,v_{r}$ be any basis of $V_{r}$. By \Cref{lemma:extension_of_finite_basis_in_dense_subspace}, we may extend this to a sequence $(v_i)_i\subset\ran{\mathcal{C}_\pr^{1/2}}$ which forms an ONB of $\H$. Thus, $(v_i)_{i>r}\subset V_{r}^\perp\cap{\ran{\mathcal{C}_\pr^{1/2}}}$ is an ONB of $V_{r}^\perp$. This shows that $V_{r}^\perp\cap{\ran{\mathcal{C}_\pr^{1/2}}}$ is dense in $V_{r}^\perp$. Since $\mathcal{C}_\pr^{-1/2}\mathcal{C}_\pos^{1/2}(\mathcal{C}_\pr^{-1/2}\mathcal{C}_\pos^{1/2})^*$ is continuous, it follows that the map $z\mapsto \norm{z}^{-2}\langle \mathcal{C}_\pr^{-1/2}\mathcal{C}_\pos^{1/2}(\mathcal{C}_\pr^{-1/2}\mathcal{C}_\pos^{1/2})^*z,z\rangle$ is continuous. Thus,
	\begin{align*}
		\min_{z\in V_{r}^\perp\setminus\{0\}}\frac{\langle \mathcal{C}_\pr^{-1/2}\mathcal{C}_\pos^{1/2}(\mathcal{C}_\pr^{-1/2}\mathcal{C}_\pos^{1/2})^*z,z\rangle}{\norm{z}^2}
		= \inf_{z\in (V_{r}^\perp\cap{\ran{\mathcal{C}_\pr^{1/2}}})\setminus\{0\}}\frac{\langle \mathcal{C}_\pr^{-1/2}\mathcal{C}_\pos^{1/2}(\mathcal{C}_\pr^{-1/2}\mathcal{C}_\pos^{1/2})^*z,z\rangle}{\norm{z}^2}.
	\end{align*}
	Now, $(\mathcal{C}_\pr^{-1/2}\mathcal{C}_\pos^{1/2})^*z = \mathcal{C}_{\pos}^{1/2}\mathcal{C}_\pr^{-1/2}z$ for $z\in\ran{\mathcal{C}_\pr^{1/2}}$ by \Cref{lemma:adjoint_of_densely_defined_operators}\ref{item:adjoint_of_densely_defined_operators_1}. Hence, for $z\in\ran{\mathcal{C}_\pr^{1/2}}$ we have $\langle \mathcal{C}_\pr^{-1/2}\mathcal{C}_\pos^{1/2}(\mathcal{C}_\pr^{-1/2}\mathcal{C}_\pos^{1/2})^*z,z\rangle = \langle \mathcal{C}_\pos \mathcal{C}_\pr^{-1/2}z,\mathcal{C}_\pr^{-1/2}z\rangle = \var_{X\sim\mu_\pos}(\langle X,\mathcal{C}_\pr^{-1/2}z\rangle)$ using \Cref{lemma:linear_push_forward_gaussian}. The equation \eqref{eqn:variance_reduction_over_dense_set} now follows, because $\norm{z}^2 = \var_{X\sim\mu_\pr}(\langle X,\mathcal{C}_\pr^{-1/2}z\rangle)$ for $z\in\ran{\mathcal{C}_\pr^{1/2}}$ by \Cref{lemma:linear_push_forward_gaussian}. 
	We note that the infimum in \eqref{eqn:variance_reduction_over_dense_set} is equal to
\begin{align*}
		\inf_{z\in\H:\ \mathcal{C}_\pr^{1/2}z\in V_{r}^\perp\setminus\{0\}} \frac{\var_{X\sim\mu_\pos}(\langle X,z\rangle)}{\var_{X\sim\mu_\pr}(\langle X,z\rangle)}
		&=  
		\inf_{z\in(\mathcal{C}_\pr^{-1/2}V_{r})^\perp\setminus\{0\}} \frac{\var_{X\sim\mu_\pos}(\langle X,z\rangle)}{\var_{X\sim\mu_\pr}(\langle X,z\rangle)}\\
		&= 
		\inf_{z\in(\mathcal{C}_\pr^{-1/2}V_{r})^\perp,\norm{z} = 1} \frac{\var_{X\sim\mu_\pos}(\langle X,z\rangle)}{\var_{X\sim\mu_\pr}(\langle X,z\rangle)},
\end{align*} 
where in the final step we use that the ratio $\frac{\var_{X\sim\mu_\pos}(\langle X,z\rangle)}{\var_{X\sim\mu_\pr}(\langle X,z\rangle)}$ is invariant under scaling of $X$. It remains to show that the final infimum above is attained. Since $\{z\in\H:\ \norm{z}\leq 1\}$ is weakly compact by \cite[Theorem V.4.2]{conway_course_2007}, the closed subspace $(\mathcal{C}_\pr^{-1/2}V_r)^\perp\cap\{z\in\H:\ \norm{z}=1\}$ of $\{z\in\H:\ \norm{z}\leq1\}$ is also weakly compact. Furthermore, $\var_{X\sim \mu_\pos}(\langle X,z\rangle)=\langle \mathcal{C}_\pos z,z\rangle $ by \Cref{lemma:linear_push_forward_gaussian}, which is weakly continuous in $z$ by \Cref{lemma:weakly_continuous_inner_products}. Similarly, $\var_{X\sim \mu_\pr}(\langle X,z\rangle)$ is weakly continuous. Thus the ratio $\frac{\var_{X\sim\mu_\pos}(\langle X,z\rangle)}{\var_{X\sim\mu_\pr}(\langle X,z\rangle)}$ is weakly continuous on the weakly compact set $(\mathcal{C}_\pr^{-1/2}V_r)^\top\cap\{z\in\H:\ \norm{z}=1\}$. It follows that the infima above are attained, proving \eqref{eqn:maximal_variance_reduction}.
\end{proof}

\subsection{Proofs of Section \ref{sec:optimal_approximation_mean}}
\label{subsec:proofs_for_optimal_approximation_means}

\lemmaSquareRoots*
\begin{proof}[Proof of \Cref{lemma:square_roots}]
	We recall that $\lambda_i=0$ for $i>n$ by \Cref{prop:bayesian_feldman_hajek}.
	Since $\mathcal{C}_\obs^{1/2}$ has a bounded inverse, \Cref{lemma:adjoint_of_densely_defined_operators} and \eqref{eqn:svd_preconditioned_hessian} imply
	\begin{align*}
		\mathcal{C}_\obs^{-1/2}G^*\mathcal{C}_\pr G\mathcal{C}_\obs^{-1/2} = (\mathcal{C}_\pr^{1/2}G^*\mathcal{C}_\obs^{-1/2})^*(\mathcal{C}_\pr^{1/2}G^*\mathcal{C}_\obs^{-1/2}) = \sum_{i=1}^{n}\frac{-\lambda_i}{1+\lambda_i}\varphi_i\otimes\varphi_i.
	\end{align*}
	Therefore, using the definitions of $S_\y$ and $\mathcal{C}_\y$ in \eqref{eqn:square_roots} and \eqref{eqn:data_covariance}, we have 
	\begin{align*}
		\mathcal{C}_\y = \mathcal{C}_\obs + G^*\mathcal{C}_\pr G = \mathcal{C}_\obs^{1/2}(I+\mathcal{C}_\obs^{-1/2}G^*\mathcal{C}_\pr G\mathcal{C}_\obs^{-1/2})\mathcal{C}_\obs^{1/2} = S_\y S_\y^*. 
	\end{align*}
	Because $S_\y$ is a rank-$n$ operator on $\R^n$, it has a bounded inverse. 
	Next, $I+\sum_{i} \frac{-\lambda_i}{1+\lambda_i}w_i\otimes w_i$ is boundedly invertible by \Cref{lemma:inverse_of_self_adj_hilbert_schmidt_perturbation}\ref{item:inverse_of_self_adj_hilbert_schmidt_perturbation_1} since $\frac{-\lambda_i}{1+\lambda_i}\not=-1$ for all $i$, 
	hence $\ran{S_\pos}=\ran{\mathcal{C}_\pr^{1/2}}$. Because $S_\pos$ is an injective operator, this shows that the inverse of $S_\pos:\mathcal{H}\rightarrow\ran{\mathcal{C}_\pr^{1/2}}$ exists. 
	Furthermore, $I+\sum_{i} \frac{-\lambda_i}{1+\lambda_i}w_i\otimes w_i$ maps $\ran{\mathcal{C}_\pr^{1/2}}$ onto itself, since $(w_i)_i\subset\ran{\mathcal{C}_\pr^{1/2}}$ by \Cref{prop:bayesian_feldman_hajek}. Hence also $(I+\sum_{i} \frac{-\lambda_i}{1+\lambda_i}w_i\otimes w_i)^{-1}$ maps $\ran{\mathcal{C}_\pr^{1/2}}$ onto itself. Recalling that $\ran{\mathcal{C}_\pr}=\ran{\mathcal{C}_\pos}$ by the discussion after \eqref{eqn:pos_precision}, it follows that $\ran{S}_\pos S_\pos^* = \ran{\mathcal{C}_\pr^{1/2}(I+\sum_{i} \frac{-\lambda_i}{1+\lambda_i}w_i\otimes w_i)^{-1}\mathcal{C}_\pr^{1/2}}=\ran{\mathcal{C}_\pr}=\ran{\mathcal{C}}_\pos$.
	By \eqref{eqn:pos_precision} and \eqref{eqn:decomposition_prior_preconditioned_hessian}, it holds on $\ran{\mathcal{C}_\pos}$,
	\begin{align*}
		\mathcal{C}_\pos^{-1} = \mathcal{C}_\pr^{-1} + H = \mathcal{C}_\pr^{-1/2}(I+\mathcal{C}_\pr^{1/2}H\mathcal{C}_\pr^{1/2})\mathcal{C}_\pr^{-1/2} = \mathcal{C}_\pr^{-1/2}\left(I+\sum_{i=1}^{n}\frac{-\lambda_i}{1+\lambda_i}w_i\otimes w_i\right)\mathcal{C}_\pr^{-1/2}=(S_\pos S_\pos^*)^{-1}.
	\end{align*}
	This shows that $\mathcal{C}_\pos=S_\pos S_\pos^*$, which proves \cref{item:square_roots}.
	\Cref{item:computation_cm_norm} now immediately follows from \cite[Corollary B.3]{da_prato_stochastic_2014} and the equality $\ran{S_\pos}=\ran{\mathcal{C}_\pr^{1/2}}=\ran{\mathcal{C}_\pos^{1/2}}.$ 
	For \cref{item:root_pre_cameron_martin_space}, we note that by \eqref{eqn:square_root_expression_S_pos} we have for $h\in\ran{\mathcal{C}_\pr^{1/2}}$, $\left( I+\sum_{i=1}^{n}\frac{-\lambda_i}{1+\lambda_i}w_i\otimes w_i\right)^{-1/2}h = \sum_{i}^{}(1+\lambda_i)^{1/2}\langle h,w_i\rangle w_i = h - \sum_{i=1}^{n}\langle h,w_i\rangle w_i + \sum_{i=1}^{n}(1+\lambda_i)^{1/2}\langle h,w_i\rangle w_i\in\ran{\mathcal{C}_\pr^{1/2}}$ as a sum of elements of $\ran{\mathcal{C}_\pr^{1/2}}$, because $(w_i)_i\subset\ran{\mathcal{C}_\pr^{1/2}}$ by \Cref{prop:bayesian_feldman_hajek}. Furthermore, if $k\in\ran{\mathcal{C}_\pr^{1/2}}$, then $h\coloneqq \sum_{i}^{}(1+\lambda_i)^{-1/2}\langle k,w_i\rangle w_i$ satisfies $h=k - \sum_{i=1}^{n}\langle k,w_i\rangle w_i + \sum_{i=1}^{n}(1+\lambda_i)^{-1/2}\langle k,w_i\rangle w_i\in\ran{\mathcal{C}_\pr^{1/2}}$. By \eqref{eqn:square_root_expression_S_pos}, we have $\left( I+\sum_{i=1}^{n}\frac{-\lambda_i}{1+\lambda_i}w_i\otimes w_i\right)^{-1/2}h=\sum_{i}^{}(1+\lambda_i)^{1/2}\langle h,w_i\rangle w_i=\sum_{i}^{}\langle k,w_i\rangle w_i =k$. We conclude that $\left( I+\sum_{i=1}^{n}\frac{-\lambda_i}{1+\lambda_i}w_i\otimes w_i\right)^{-1/2}$ maps $\ran{\mathcal{C}_\pr^{1/2}}$ onto $\ran{\mathcal{C}_\pr^{1/2}}$, so that 
	\begin{align*}
		S_\pos(\ran{\mathcal{C}_\pr^{1/2}})=\mathcal{C}_\pr^{1/2}\left( I+\sum_{i=1}^{n}\frac{-\lambda_i}{1+\lambda_i}w_i\otimes w_i\right)^{-1/2}(\ran{\mathcal{C}_\pr^{1/2}})=\mathcal{C}_\pr^{1/2}(\ran{\mathcal{C}_\pr^{1/2}})=\ran{\mathcal{C}_\pr}=\ran{\mathcal{C}_\pos}.
	\end{align*}
\end{proof}

\propMeanApproximationProblemEquivalence*

\begin{proof}[Proof of \Cref{prop:mean_approximation_problem_equivalence}]
	\ref{item:mean_approximation_problem_equivalence_1} Note that by \eqref{eqn:pos_mean}, $m_\pos(Y)\in\ran{\mathcal{C}_\pos}\subset\ran{\mathcal{C}_\pos^{1/2}}$ with probability 1.
	We first show the reverse inclusions. Suppose that $A\in\B(\R^n,\mathcal{H})$ satisfies $\ran{A}\subset\ran{\mathcal{C}_\pr^{1/2}}=\ran{\mathcal{C}_\pos^{1/2}}$. Because $m_\pos(Y)\in\ran{\mathcal{C}_\pos^{1/2}}$ with probability 1, it follows that $AY-m_\pos(Y)\in\ran{\mathcal{C}_\pos^{1/2}}$ with probability 1. Hence, by \Cref{thm:feldman--hajek}, it holds that $\mathcal{N}(m_\pos(Y),\mathcal{C}_\pos)\sim\mathcal{N}(AY,\mathcal{C}_\pos)$ with probability 1. This implies the reverse inclusion for $i=2$. To see that it also implies the reverse inclusion for $i=1$, we show that $\ran{A}\subset\ran{\mathcal{C}_\pr^{1/2}}$ holds true if $A\in\mathscr{M}^{(1)}_r$, that is, if $A=(\mathcal{C}_\pr-B)G^*\mathcal{C}_\obs^{-1}$ for some $B\in\B_{00,r}(\mathcal{H})$ with $B(\ker{G}^\perp)\subset\ran{\mathcal{C}_\pr^{1/2}}$. Since $G$ has finite rank, its range is closed. Thus, $\ran{BG^*}=B(\ran{G^*})=B(\overline{\ran{G^*}})=B(\ker{G}^\perp)$ by \Cref{lemma:kernel_range}. Therefore, $\ran{BG^*\mathcal{C}_\obs^{-1}}\subset\ran{BG^*}\subset\ran{\mathcal{C}_\pr^{1/2}}$. With $\ran{\mathcal{C}_{\pr}}\subset\ran{\mathcal{C}_\pr^{1/2}}$ it follows that $\ran{A}=\ran{(\mathcal{C}_{\pr}-B)G^*\mathcal{C}_\obs^{-1}}\subset\ran{\mathcal{C}_\pr^{1/2}}$.

	We show the forward inclusions next. 
	Suppose that $A\in\mathscr{M}^{(i)}_r$ for $i=1$ or $i=2$. By \Cref{thm:feldman--hajek}\ref{item:fh_means}, $AY-m_\pos(Y)\in\ran{\mathcal{C}^{1/2}_\pos}$ with probability 1. Since $m_\pos(Y)\in\ran{\mathcal{C}^{1/2}_\pos}$ with probability 1 by \eqref{eqn:pos_mean}, this implies $AY\in\ran{\mathcal{C}_\pos^{1/2}}$ with probability 1. 
	Now fix $i=2$. By \Cref{lemma:linear_push_forward_gaussian}, $AY$ is a Gaussian measure with covariance $A\mathcal{C}_\y A^*$, where $\mathcal{C}_\y$ is the covariance of $Y$. By \cite[Theorem 2.4.7]{bogachev_gaussian_1998} or \cite[Proposition 4.45]{Hairer2023}, the Cameron--Martin space of a Gaussian measure is contained in every measurable linear subspace of full measure. Thus, since $AY \in\ran{\mathcal{C}_\pos^{1/2}}$ with probability 1, the Cameron--Martin space of $AY$, which is $\ran{(A\mathcal{C}_\y A^*)^{1/2}}$, is contained in $\ran{\mathcal{C}_\pos^{1/2}}=\ran{\mathcal{C}_\pr^{1/2}}$. Because $A$ has finite rank, $A\mathcal{C}_\y A^*$ has finite rank and therefore $\ran{A\mathcal{C}_\y A^*} = \ran{(A\mathcal{C}_\y A^*)^{1/2}}$, by \Cref{lemma:range_of_square_of_finite_rank} applied to $A\leftarrow (A\mathcal{C}_\y A^*)^{1/2}$. Furthermore, by \Cref{lemma:range_of_square_of_finite_rank} applied to $A\leftarrow A\mathcal{C}_\y^{1/2}$ and invertibility of $\mathcal{C}_\y$, we have $\ran{A}=\ran{A\mathcal{C}_\y}^{1/2}=\ran{A\mathcal{C}_\y^{1/2}(A\mathcal{C}_\y^{1/2})^*}=\ran{A\mathcal{C}_\y A^*}$. As a consequence, $\ran{A}=\ran{(A\mathcal{C}_\y A^*)^{1/2}}\subset\ran{\mathcal{C}_\pr^{1/2}}$. This shows the forward inclusion for $i=2$.
	Finally, let $i=1$. Thus, $A=(\mathcal{C}_\pr-B)G^*\mathcal{C}_\obs^{-1}$ for some $B\in \B_{00,r}(\mathcal{H})$. Since we just showed that $AY\in\ran{\mathcal{C}_\pos^{1/2}}$ with probability 1, and since $\ran{\mathcal{C}_\pr}\subset\ran{\mathcal{C}_\pr^{1/2}}=\ran{\mathcal{C}_\pos^{1/2}}$, it follows that $BG^*\mathcal{C}_\obs^{-1}Y\in\ran{\mathcal{C}_\pos^{1/2}}$ with probability 1. By replacing $A$ with $BG^*\mathcal{C}_\obs^{-1}$ in the argument for the case where $i=2$, we obtain
	$\ran{BG^*}\mathcal{C}_\obs^{-1}\subset\ran{\mathcal{C}_\pr^{1/2}}$. Since $\ran{G^*}$ is finite-dimensional, it is closed. Using that $\mathcal{C}_\obs$ is invertible, this implies $B(\ker{G}^\perp)=B(\ran{G^*})\subset\ran{\mathcal{C}_\pr^{1/2}}$ by \Cref{lemma:kernel_range}. This shows the forward inclusion for $i=1$.

	\ref{item:mean_approximation_problem_equivalence_2} By \Cref{lemma:square_roots}\ref{item:square_roots}, $S_\pos$ is injective and $\ran{S_\pos}=\ran{\mathcal{C}_\pr^{1/2}}=\ran{\mathcal{C}_\pos^{1/2}}$. Thus, $\rank{S_\pos \widetilde{A}}=\rank{\widetilde{A}}$ and $\ran{S_\pos \widetilde{A}\subset\ran{\mathcal{C}_\pos^{1/2}}}$ for every $\widetilde{A}\in\B_{00,r}(\R^n,\mathcal{H})$. By \eqref{eqn:equivalent_class_approx_mean_without_covariance_update}, $S_\pos\widetilde{\mathscr{M}}^{(2)}_r=\{A\in\B_{00,r}(\R^n,\mathcal{H}):\ \ran{A}\subset\ran{\mathcal{C}_\pr^{1/2}}\}=\mathscr{M}^{(2)}_r$. This shows the result for $i=2$. For $i=1$, first let $A\in\mathscr{M}^{(1)}_r$. By \eqref{eqn:equivalent_class_approx_mean_with_covariance_update}, this implies $A=(\mathcal{C}_\pr-B)G^*\mathcal{C}_\obs^{-1}$ for some $B\in\B_{00,r}(\mathcal{H})$ with $B(\ker{G}^\perp)\subset\ran{\mathcal{C}_\pr^{1/2}}$. Let $\widetilde{B}\coloneqq S_\pos^{-1}BP_{\ker{G}^\perp}$, where $P_{\ker{G}^\perp}$ denotes the orthogonal projector onto $\ker{G}^\perp$. Then $\widetilde{B}$ is well-defined, because $\ran{BP_{\ker{G}^\perp}}=B(\ker{G}^\perp)\subset\ran{\mathcal{C}_\pr^{1/2}}=\dom{S_\pos^{-1}}$ by \Cref{lemma:square_roots}\ref{item:square_roots}. Furthermore, $\rank{\widetilde{B}}\leq \rank{B}\leq r$ and $S_\pos\widetilde{B} = BP_{\ker{G}^\perp}$. Hence $S_\pos\widetilde{B}G^*= BG^*$ by \Cref{lemma:kernel_range}, showing $A=S_\pos(S_\pos^{-1}\mathcal{C}_\pr-\widetilde{B})G^*\mathcal{C}_\obs^{-1}$. Thus, $A\in S_\pos \widetilde{\mathscr{M}}^{(1)}_r$. For the reverse inclusion, let $A\in S_\pos\widetilde{\mathscr{M}}^{(1)}_r$. That is, let $A=S_\pos(S_\pos^{-1}\mathcal{C}_\pr-\widetilde{B})G^*\mathcal{C}_\obs^{-1}$ for some $\widetilde{B}\in\B_{00,r}(\mathcal{H})$. Then $A=(\mathcal{C}_\pr-B)G^*\mathcal{C}_\obs^{-1}$, where $B\coloneqq S_\pos \widetilde{B}$ satisfies $\rank{B}=\rank{\widetilde{B}}\leq r$ and $B(\ker{G}^\perp)\subset\ran{B}\subset\ran{S_\pos}=\ran{\mathcal{C}_\pos^{1/2}}=\ran{\mathcal{C}_\pr^{1/2}}$. By \eqref{eqn:equivalent_class_approx_mean_with_covariance_update}, this shows that $A\in\mathscr{M}^{(1)}_r$.

	\ref{item:mean_approximation_problem_equivalence_3} For $i=1,2$, we note that $\widetilde{A}_1,\widetilde{A}_2\in\widetilde{\mathscr{M}}^{(i)}_r$ satisfy
	\begin{align*}
		\mathbb{E} \Norm{\widetilde{A}_1Y-S_\pos^{-1}m_{\pos}(Y)}^2 \leq \mathbb{E} \Norm{\widetilde{A}_2Y-S_\pos^{-1}m_{\pos}(Y)}^2,
	\end{align*}
	if and only if 
	\begin{align*}
		\mathbb{E} \Norm{S_\pos^{-1}(S_\pos \widetilde{A}_1Y-m_{\pos}(Y))}^2 \leq \mathbb{E} \Norm{S_\pos^{-1}(S_\pos\widetilde{A}_2Y-m_{\pos}(Y))}^2.
	\end{align*}
	By \Cref{lemma:square_roots}\ref{item:computation_cm_norm} and \cref{item:mean_approximation_problem_equivalence_2} above, this shows that $\widetilde{A}_1$ solves \Cref{prob:optimal_mean_reformulated} if and only if $S_\pos \widetilde{A}_1$ solves \Cref{prob:optimal_mean}.

	\ref{item:mean_approximation_problem_equivalence_4} This follows immediately from \cref{item:mean_approximation_problem_equivalence_2,item:mean_approximation_problem_equivalence_3}.
\end{proof}

\lemmaMeanEquivalentLoss*
\begin{proof}[Proof of \Cref{lemma:mean_equivalent_loss}]
	Let $\widetilde{A}\in\B(\R^n,\mathcal{H})$. Recall from \Cref{lemma:square_roots} that $\mathcal{C}_\pos=S_\pos S_\pos^*$ and from \eqref{eqn:pos_mean} that $m_\pos(y)=\mathcal{C}_\pos G^*\mathcal{C}_\obs^{-1}y$ for $y\in\R^n$. Thus if we let $Z\coloneqq\widetilde{A}-S_\pos^* G^*\mathcal{C}_\obs^{-1}$, then $\widetilde{A} Y- S_\pos^{-1} m_\pos(Y)=ZY$. By \Cref{lemma:linear_push_forward_gaussian}, the covariance of $ZY$ is $Z\mathcal{C}_\y Z^*$. Then, by applying \Cref{lemma:gaussian_expected_norm_squared} with $X\leftarrow ZY$, $\mathcal{C}\leftarrow Z\mathcal{C}_\y Z^*$, and $S\leftarrow ZS_\y$,
	\begin{align*}
		\mathbb{E}\Norm{\widetilde{A}Y-S_\pos^{-1}m_\pos(Y)}^2 = \norm{ZS_\y}_{L_2(\mathcal{H})}^2 = \norm{\widetilde{A}S_\y - S_\pos^*G^*\mathcal{C}_\obs^{-1}S_\y}_{L_2(\mathcal{H})}^2.
	\end{align*}
	Thus, to show \eqref{eqn:mean_equivalent_loss} it remains to show $\mathcal{C}_\pr^{1/2}G^*\mathcal{C}_\obs^{-1/2} = S_\pos^*G^*\mathcal{C}_\obs^{-1}S_\y$.
	By \eqref{eqn:square_roots},
	\begin{align*}
		S_\pos^*G^*\mathcal{C}_\obs^{-1}S_\y &= \left( I+\sum_{i=1}^{n}\frac{-\lambda_i}{1+\lambda_i}w_i\otimes w_i \right)^{-1/2}\mathcal{C}_\pr^{1/2}G^*\mathcal{C}_\obs^{-1/2} 
		\left( I+\sum_{i=1}^{n}\frac{-\lambda_i}{1+\lambda_i}\varphi_i\otimes \varphi_i \right)^{1/2}.
	\end{align*}
	Fix an arbitrary $x\in\R^n$. Then, 
	\begin{align*}
		S_\pos^*G^*\mathcal{C}_\obs^{-1}S_\y x &= \left( I+\sum_{i}^{}\frac{-\lambda_i}{1+\lambda_i}w_i\otimes w_i \right)^{-1/2} \left( \sum_{i=1}^{n}\sqrt{\frac{-\lambda_i}{1+\lambda_i}}w_i\otimes \varphi_i \right)
		 \sum_{i}^{}(1+\lambda_i)^{-1/2}\langle x,\varphi_i\rangle \varphi_i \\
		&= \left( I+\sum_{i}^{}\frac{-\lambda_i}{1+\lambda_i}w_i\otimes w_i \right)^{-1/2} \sum_{i=1}^{n}\sqrt{\frac{-\lambda_i}{(1+\lambda_i)^2}}\langle x,\varphi_i\rangle w_i \\
		&= \sum_{i=1}^{n}\sqrt{\frac{-\lambda_i}{1+\lambda_i}}\langle x,\varphi_i\rangle w_i\\
		&= \left(\sum_{i=1}^{n} \sqrt{\frac{-\lambda_i}{1+\lambda_i}}w_i\otimes \varphi_i\right) x = \mathcal{C}_\pr^{1/2}G^*\mathcal{C}_\obs^{-1/2}x,
	\end{align*}
	where we use \eqref{eqn:svd_preconditioned_hessian} and \eqref{eqn:square_root_expression_S_y} in the first equation, \eqref{eqn:square_root_expression_S_pos} in the third equation and \eqref{eqn:svd_preconditioned_hessian} in the last equation.
\end{proof}

\thmOptimalLowRankMeanApprox*
\begin{proof}[Proof of \Cref{thm:optimal_low_rank_mean_approx}]
	In order to solve \Cref{prob:optimal_mean}, it suffices by \Cref{lemma:mean_equivalent_loss} and \eqref{eqn:classes_mean_reparametrised} to first find $\widetilde{A}_{r}^{\opt,(2)}$ that solves the rank-constrained operator approximation problem
	\begin{align}
		\label{eqn:min_hs_problem_mean}
		\min\left\{\Norm{\widetilde{A} S_\y - \mathcal{C}_\pr^{1/2} G^*\mathcal{C}_\pos^{-1/2}}^2_{L_2(\H)}:\ \widetilde{A}\in \widetilde{\mathscr{M}}^{(2)}_r=\B_{00,r}(\R^r,\mathcal{H})\right\},
	\end{align}
	and then set $A_r^{\opt,(2)} \coloneqq S_\pos \widetilde{A}_r^{\opt,(2)}$ using \Cref{prop:mean_approximation_problem_equivalence}\ref{item:mean_approximation_problem_equivalence_3}. Note that $I^{\dagger}=I$, that $S_y^{\dagger}=S_\y^{-1}$ by \Cref{lemma:square_roots}\ref{item:square_roots}, and that $(\mathcal{C}_\pr^{1/2} G^*\mathcal{C}_\pos^{-1/2})_r\coloneqq \sum_{i=1}^{r}\sqrt{\frac{-\lambda_i}{1+\lambda_i}}w_i\otimes \varphi_i$ is a rank-$r$ truncated SVD of $\mathcal{C}_\pr^{1/2} G^*\mathcal{C}_\pos^{-1/2}$ by \eqref{eqn:svd_preconditioned_hessian}. Since $I\in\mathcal{B}(\mathcal{H})$ and $S_\y\in\B(\R^n)$ have closed range, and since $\mathcal{C}_\pr^{1/2} G^*\mathcal{C}_\pos^{-1/2}$ has finite rank and is thus Hilbert--Schmidt, we may apply \Cref{thm:generalised_friedland_torokhti} with $\mathcal{H}_i\leftarrow \R^n$ for $i\in\{1,2\}$, $\mathcal{H}_i\leftarrow\mathcal{H}$ for $i\in\{3,4\}$, $T\leftarrow I$, $S\leftarrow S_\y$, $M\leftarrow \mathcal{C}_\pr^{1/2} G^*\mathcal{C}_\pos^{-1/2}$ to find
	\begin{align*}
		\widetilde{A}_r^{\opt,(2)} &= \left(\sum_{i=1}^{r}\sqrt{\frac{-\lambda_i}{1+\lambda_i}}w_i\otimes \varphi_i\right) S_\y^{-1}.
	\end{align*}
	Since $(w_i)_i\subset\ran{\mathcal{C}_\pr^{1/2}}$ by \Cref{prop:bayesian_feldman_hajek}, it follows by \Cref{lemma:square_roots}\ref{item:root_pre_cameron_martin_space} that $\ran{{A}_r^{\opt,(2)}}=\ran{S_\pos\widetilde{A}_r^{\opt,(2)}}\subset\Span{S_\pos w_i,\ i\leq r}\subset \ran{\mathcal{C}}_\pr=\ran{\mathcal{C}_\pos}$.
	Thus,
	\begin{align*}
		A_r^{\opt,(2)} &= S_\pos \widetilde{A}_r^{\opt,(2)} = S_\pos \left(\sum_{i=1}^{r}\sqrt{\frac{-\lambda_i}{1+\lambda_i}}w_i\otimes \varphi_i\right) S_\y^{-1}\\
		&= \mathcal{C}_\pr^{1/2}\left( I+\sum_{i=1}^{n}\frac{-\lambda_i}{1+\lambda_i} w_i\otimes w_i\right)^{-1/2}\sum_{i=1}^{r}\sqrt{\frac{-\lambda_i}{1+\lambda_i}}w_i\otimes \varphi_i
		\left( I+\sum_{i=1}^{n}\frac{-\lambda_i}{I+\lambda_i}\varphi_i\otimes \varphi_i \right)^{-1/2}\mathcal{C}_\obs^{-1/2}\\
		&= \mathcal{C}_\pr^{1/2} \left(\sum_{i=1}^{r}\sqrt{-\lambda_i(1+\lambda_i)} w_i\otimes \varphi_i\right) \mathcal{C}_\obs^{-1/2},
	\end{align*}
	where we used \eqref{eqn:square_roots} in the third equation and \eqref{eqn:square_root_expression} in the last equation.
	Using \eqref{eqn:svd_preconditioned_hessian}, the definition of the Hilbert--Schmidt norm and the definition of $\widetilde{A}_r^{\opt,(2)}$, we can compute the corresponding minimal loss:
	\begin{align*}
		\Norm{\widetilde{A}_r^{\opt,(2)} S_\y - \mathcal{C}_\pr^{1/2} G^*\mathcal{C}_\pos^{-1/2} }_{L_2(\mathcal{H})}^2 &= \Norm{\sum_{i=1}^{r}\sqrt{\frac{-\lambda_i}{1+\lambda_i}}w_i\otimes \varphi_i - \sum_{i=1}^{n}\sqrt{\frac{-\lambda_i}{1+\lambda_i}}w_i\otimes \varphi_i}_{L_2(\mathcal{H})}^2 
		= \sum_{i>r}^{}\frac{-\lambda_i}{1+\lambda_i}.
	\end{align*}
	Finally, by \Cref{prop:mean_approximation_problem_equivalence}\ref{item:mean_approximation_problem_equivalence_3}-\ref{item:mean_approximation_problem_equivalence_4} and \Cref{lemma:mean_equivalent_loss} it holds that \Cref{prob:optimal_mean} has a unique solution if and only if \eqref{eqn:min_hs_problem_mean} has a unique solution. With the above choices of $M$, $T$ and $S$ it holds that $P_{\ker{T}^\perp}=I$ and $P_{\ran{S}}=I$, and \Cref{thm:generalised_friedland_torokhti} and \eqref{eqn:svd_preconditioned_hessian} imply that \eqref{eqn:min_hs_problem_mean} has a unique solution if and only if ${-\lambda_{r+1}}(1+\lambda_{r+1})^{-1}=0$ or ${-\lambda_r}(1+\lambda_r)^{-1}>{-\lambda_{r+1}}(1+\lambda_{r+1})^{-1}$. Since $(\lambda_i)_i\subset(-1,0]$ is a nonincreasing sequence by \Cref{prop:bayesian_feldman_hajek} and $x\mapsto {-x}(1+x)^{-1}$ is decreasing on $(-1,\infty)$, the latter condition holds if and only if $\lambda_{r+1}=0$ or $\lambda_r<\lambda_{r+1}$. This concludes the proof of uniqueness.
\end{proof}

\optimalLowRankUpdateMeanApprox*
\begin{proof}[Proof of \Cref{thm:optimal_low_rank_update_mean_approx}]
	In order to solve \Cref{prob:optimal_mean}, it suffices by \Cref{lemma:mean_equivalent_loss} and \eqref{eqn:classes_mean_reparametrised} to first find $\widetilde{A}^{\opt,(1)}_r$ that solves the rank-constrained operator approximation problem
	\begin{align}
		\label{eqn:min_hs_problem_mean_update}
		\min\left\{\Norm{\widetilde{A} S_\y- \mathcal{C}_\pr^{1/2} G^*\mathcal{C}_\pos^{-1/2}}^2_{L_2(\H)}:\ \widetilde{A}\in \widetilde{\mathscr{M}}^{(1)}_r\right\},
	\end{align}
	and then set $A_r^{\opt,(1)} \coloneqq S_\pos \widetilde{A}^{\opt,(1)}_r$ using \Cref{prop:mean_approximation_problem_equivalence}\ref{item:mean_approximation_problem_equivalence_3}.
	Recall that by definition \eqref{eqn:classes_mean_reparametrised}, $\widetilde{A}\in\widetilde{\mathscr{M}}_r^{(1)}$ if and only if $\widetilde{A}=(S_\pos^{-1}\mathcal{C}_\pr-\widetilde{B})G^*\mathcal{C}_\obs^{-1}$ for some $\widetilde{B}\in\B_{00,r}(\H)$. 
	Notice that for such $\widetilde{A}$,
	\begin{align*}
		\widetilde{A}S_\y-\mathcal{C}_\pr^{1/2}G^*\mathcal{C}_\obs^{-1/2} = S_\pos^{-1}\mathcal{C}_\pr G^*\mathcal{C}_\obs^{-1}S_\y - \mathcal{C}_\pr^{1/2}G^*\mathcal{C}_\obs^{-1/2} - \widetilde{B}G^*\mathcal{C}_\obs^{-1}S_\y.
	\end{align*}
	The above rank-$r$ operator approximation problem can therefore be solved by solving the following rank-$r$ operator approximation problem
	\begin{align}
		\label{eqn:min_hs_problem_mean_update_2}
		\min\left\{\Norm{S_\pos^{-1}\mathcal{C}_\pr G^*\mathcal{C}_\obs^{-1}S_\y - \mathcal{C}_\pr^{1/2}G^*\mathcal{C}_\obs^{-1/2} -\widetilde{B} G^*\mathcal{C}_\obs^{-1}S_\y}_{L_2(\H)}\ :\ \widetilde{B}\in \B_{00,r}(\H)\right\},
	\end{align}
	and $\widetilde{A}$ solves \eqref{eqn:min_hs_problem_mean_update} if and only if $\widetilde{A}=(S_\pos^{-1}\mathcal{C}_\pr-\widetilde{B})G^*\mathcal{C}_\obs^{-1}$ for some $\widetilde{B}$ solving \eqref{eqn:min_hs_problem_mean_update_2}.
	Since $I\in\mathcal{B}(\mathcal{H})$ and $G^*\mathcal{C}_\obs^{-1}S_\y$ have closed range and since $S_\pos^{-1}\mathcal{C}_\pr G^*\mathcal{C}_\obs^{-1}S_\y - \mathcal{C}_\pr^{1/2}G^*\mathcal{C}_\obs^{-1/2}$ has finite rank and therefore is Hilbert--Schmidt, we may apply \Cref{thm:generalised_friedland_torokhti} with $\mathcal{H}_1\leftarrow \R^n$ and $\mathcal{H}_j\leftarrow\mathcal{H}$ for $j\in\{2,3,4\}$, $T\leftarrow I$, $S\leftarrow G^*\mathcal{C}_\obs^{-1}S_\y$ and $M\leftarrow S_\pos^{-1}\mathcal{C}_\pr G^*\mathcal{C}_\obs^{-1}S_\y - \mathcal{C}_\pr^{1/2}G^*\mathcal{C}_\obs^{-1/2}$ to find a solution $\widetilde{B}^\opt$ to the approximation problem \eqref{eqn:min_hs_problem_mean_update_2}. For the given choices of $T$ and $S$, we have that $T^\dagger=I$, while for the finite-rank operator $S$ we have from \eqref{eqn:svd_preconditioned_hessian} and \eqref{eqn:square_roots} that
	\begin{align*}
		S &= \mathcal{C}_\pr^{-1/2}\left( \mathcal{C}_\pr^{1/2}G^*\mathcal{C}_\obs^{-1/2}\right)\mathcal{C}_\obs^{-1/2}S_\y= \mathcal{C}_\pr^{-1/2}\left( \sum_{i=1}^{n}\sqrt{\frac{-\lambda_i}{1+\lambda_i}} w_i\otimes \varphi_i\right) \left( I+\sum_{i=1}^{n}\frac{-\lambda_i}{1+\lambda_i}\varphi_i\otimes \varphi_i \right)^{1/2},
	\end{align*}
	where $w_i$ is the eigenvector corresponding to the eigenvalue $\lambda_i$ given by \Cref{prop:bayesian_feldman_hajek} and $\varphi_i$ is the right singular vector corresponding to $\lambda_i$ in \eqref{eqn:svd_preconditioned_hessian}.
	By \cite[Theorem 2.8]{engl_regularization_1996}, the Moore--Penrose inverse of $\sum_{i=1}^{n}\sqrt{\frac{-\lambda_i}{1+\lambda_i}}w_i\otimes \varphi_i$ is given by $\sum_{i=1}^{n}\sqrt{\frac{1+\lambda_i}{-\lambda_i}}\varphi_i\otimes w_i$. Furthermore, the Moore--Penrose inverse of a composition of bounded operators is the composition in reverse order of the Moore--Penrose inverses of these operators, see e.g.\ \cite[eq.\ (3.23)]{Hsing2015}. Since $\mathcal{C}_\pr^{-1/2}$ and $I+ \sum_{i=1}^{n}\frac{-\lambda_i}{1+\lambda_i} \varphi_i\otimes \varphi_i$ are boundedly invertible by \Cref{lemma:inverse_of_self_adj_hilbert_schmidt_perturbation}, it thus holds that the bounded operator 
	\begin{align*}
		\left( I+ \sum_{i=1}^{n}\frac{-\lambda_i}{1+\lambda_i} \varphi_i\otimes \varphi_i\right)^{-1/2}\left(\sum_{i=1}^{n}\sqrt{\frac{1+\lambda_i}{-\lambda_i}}\varphi_i\otimes w_i\right)\mathcal{C}_\pr^{1/2}
	\end{align*}
	has Moore--Penrose inverse equal to $S$. Because \cite[Theorem 9.2(f)]{ben-israel2003} implies that $(\mathfrak{S}^{\dagger})^\dagger=\mathfrak{S}$ for any bounded operator $\mathfrak{S}$, the operator in the display above is equal to $S^\dagger$.
	Furthermore, by \cite[eq.\ (2.12)]{engl_regularization_1996}, $P_{\ker{S}^\perp}= S^\dagger S$, showing that $P_{\ker{S}^\perp}= \sum_{i=1}^{n}\varphi_i\otimes \varphi_i.$
	Next, we compute for the given choice of $M$,
	\begin{align}
		\nonumber
		M = &S_\pos^{-1}\mathcal{C}_\pr^{1/2} \left(\mathcal{C}_\pr^{1/2}G^*\mathcal{C}_\obs^{-1/2}\right)\mathcal{C}_\obs^{-1/2}S_\y-\mathcal{C}_\pr^{1/2}G^*\mathcal{C}_\obs^{-1/2}\\
		\nonumber
		= &\left(I+ \sum_{i=1}^{n}\frac{-\lambda_i}{1+\lambda_i}w_i\otimes w_i\right)^{1/2}\left( \sum_{i=1}^{n}\sqrt{\frac{-\lambda_i}{1+\lambda_i}}w_i\otimes \varphi_i\right)
		\left( I+ \sum_{i=1}^{n}\frac{-\lambda_i}{1+\lambda_i}\varphi_i\otimes \varphi_i\right)^{1/2} \\
		\nonumber
		&- \sum_{i=1}^{n}\sqrt{\frac{-\lambda_i}{1+\lambda_i}}w_i\otimes \varphi_i\\
		\label{eqn:svd_M}
		= &\sum_{i=1}^{n}\left(\sqrt{\frac{-\lambda_i}{(1+\lambda_i)^3}}- \sqrt{\frac{-\lambda_i}{1+\lambda_i}} \right)w_i\otimes \varphi_i,
	\end{align}
	where in the second equation we use \eqref{eqn:svd_preconditioned_hessian} and \eqref{eqn:square_roots}, and in the last equation we use \eqref{eqn:square_root_expression}.
	Hence, $MP_{\ker{S}^\perp}=M$ and \Cref{thm:generalised_friedland_torokhti} yields, with $(M)_r$ a rank-$r$ truncated SVD of $M$,
	\begin{align*}
		\widetilde{B}^\opt &= T^\dagger (M)_r S^\dagger = \left(\sum_{i=1}^{r}\left(\sqrt{\frac{-\lambda_i}{(1+\lambda_i)^3}}- \sqrt{\frac{-\lambda_i}{1+\lambda_i}} \right)w_i\otimes \varphi_i\right) S^\dagger\\
		&= \left(\sum_{i=1}^{r}\left(\sqrt{\frac{-\lambda_i}{(1+\lambda_i)^2}}- \sqrt{-\lambda_i} \right)w_i\otimes \varphi_i\right) \left( \sum_{i=1}^{n}\sqrt{\frac{1+\lambda_i}{-\lambda_i}}\varphi_i\otimes w_i \right)\mathcal{C}_\pr^{1/2}\\
		&= \left(\sum_{i=1}^{r}\left(\sqrt{\frac{1}{1+\lambda_i}} - \sqrt{1+\lambda_i} \right)w_i\otimes w_i\right) \mathcal{C}_\pr^{1/2},
	\end{align*}
	where the third equation follows from the formula for $S^\dagger$ above, \eqref{eqn:square_root_expression_S_y}, and direct computation.
	It follows by \eqref{eqn:square_roots}, \eqref{eqn:square_root_expression_S_pos} and direct computation, that
	\begin{align*}
		S_\pos \widetilde{B}^\opt &=\mathcal{C}_\pr^{1/2}\left( I + \sum_{i\in\N}\frac{-\lambda_i}{1+\lambda_i}w_i\otimes w_i \right)^{-1/2} \left(\sum_{i=1}^{r}\left(\sqrt{\frac{1}{1+\lambda_i}} - \sqrt{1+\lambda_i} \right)w_i\otimes w_i\right) \mathcal{C}_\pr^{1/2}
		\\
		&=\mathcal{C}_\pr^{1/2}\left(\sum_{i=1}^{r}\left( 1-(1+\lambda_i) \right)w_i\otimes w_i\right) \mathcal{C}_\pr^{1/2}
		\\
		&=\sum_{i=1}^{r}(-\lambda_i)\mathcal{C}_\pr^{1/2}w_i\otimes \mathcal{C}_\pr^{1/2}w_i.
	\end{align*}
	Recall that $\widetilde{A}^\opt$ and $\widetilde{B}^\opt$ are related by $\widetilde{A}^\opt=(S_\pos^{-1}\mathcal{C}_\pr-\widetilde{B}^\opt)G^*\mathcal{C}_\obs^{-1}$. Note that the expression for $S_\pos \widetilde{B}^\opt$ above coincides with the second term on the right-hand side of \eqref{eqn:optimal_covariance} in \Cref{thm:opt_covariance_and_precision}. Thus, 
	\begin{align*}
		A_r^{\opt,(1)} = S_\pos \widetilde{A}_r^{\opt,(1)} = S_\pos(S_\pos^{-1}\mathcal{C}_\pr-\widetilde{B}^\opt)G^*\mathcal{C}_\obs^{-1} = (\mathcal{C}_\pr-S_\pos\widetilde{B}^\opt)G^*\mathcal{C}_\obs^{-1} = \mathcal{C}^\opt_r G^*\mathcal{C}_\obs^{-1}.
	\end{align*}
	Since $(w_i)_i\subset\ran{\mathcal{C}_\pr^{1/2}}$ by \Cref{prop:bayesian_feldman_hajek}, we note that $\ran{A_r^{\opt,(1)}} \subset \ran{\mathcal{C}_r^{\opt}}\subset \Span{\mathcal{C}_\pr^{1/2}w_i,\ i\leq n}\subset \ran{\mathcal{C}_\pr}=\ran{\mathcal{C}_\pos}$.
	Next, we compute the corresponding loss. By \eqref{eqn:optimal_covariance} and \eqref{eqn:svd_preconditioned_hessian},
	\begin{align*}
		\mathcal{C}_\pr^{-1/2}\mathcal{C}^\opt_r G^*\mathcal{C}_\obs^{-1/2} &= \left(I-\sum_{i=1}^{r}(-\lambda_i)w_i\otimes w_i\right)\mathcal{C}_\pr^{1/2}G^*\mathcal{C}_\obs^{-1/2}\\
		&= \left(I-\sum_{i=1}^{r}(-\lambda_i)w_i\otimes w_i\right) \sum_{i=1}^{n}\sqrt{\frac{-\lambda_i}{1+\lambda_i}}w_i\otimes \varphi_i.
	\end{align*}
	Together with \eqref{eqn:square_roots}, the preceding equation implies that
	\begin{align*}
		S_\pos^{-1}A_r^{\opt,(1)} S_\y
		&= \sum_{i=1}^{n}\sqrt{\frac{-\lambda_i}{(1+\lambda_i)^3}}w_i\otimes \varphi_i - \sum_{i=1}^{r}\sqrt{\frac{-\lambda_i}{1+\lambda_i}}^3 w_i\otimes \varphi_i.
\end{align*}
We prove the equation above as follows. Fix an arbitrary $x\in\R^n$.
Then 
\begin{align*}
		S_\pos^{-1}A_r^{\opt,(1)} S_\y x=&\left( I+ \sum_{i=1}^{n}\frac{-\lambda_i}{1+\lambda_i}w_i\otimes w_i\right)^{1/2}\mathcal{C}_\pr^{-1/2}\mathcal{C}^\opt_r G^*\mathcal{C}_\obs^{-1/2}\left( I+ \sum_{i=1}^{n}\frac{-\lambda_i}{1+\lambda_i}\varphi_i\otimes \varphi_i\right)^{1/2}x\\
		=&\left( I+ \sum_{i=1}^{n}\frac{-\lambda_i}{1+\lambda_i}w_i\otimes w_i\right)^{1/2}\mathcal{C}_\pr^{-1/2}\mathcal{C}^\opt_r G^*\mathcal{C}_\obs^{-1/2} \sum_{i=1}^{n}\frac{1}{\sqrt{1+\lambda_i}}\langle x,\varphi_i\rangle \varphi_i\\
		=& \left( I+ \sum_{i=1}^{n}\frac{-\lambda_i}{1+\lambda_i}w_i\otimes w_i\right)^{1/2}\left(I-\sum_{i=1}^{r}(-\lambda_i)w_i\otimes w_i\right) \sum_{i=1}^{n}\sqrt{\frac{-\lambda_i}{(1+\lambda_i)^2}}\langle x,\varphi_i\rangle w_i \\
		=&\left( I+ \sum_{i=1}^{n}\frac{-\lambda_i}{1+\lambda_i}w_i\otimes w_i\right)^{1/2}\left( \sum_{i=1}^{n}\sqrt{\frac{-\lambda_i}{(1+\lambda_i)^2}}\langle x,\varphi_i\rangle w_i-\sum_{i=1}^{r}\sqrt{\frac{(-\lambda_i)^3}{(1+\lambda_i)^2}}\langle x,\varphi_i\rangle w_i\right),
\end{align*}
where the first equation follows from \eqref{eqn:square_roots}, the second equation from \eqref{eqn:square_root_expression_S_y}, and the third and fourth equations follow from the equation for $\mathcal{C}_\pr^{-1/2}\mathcal{C}^\opt_r G^*\mathcal{C}_\obs^{-1/2}$ above and direct computations. Now the analogue of \eqref{eqn:square_root_expression_S_y} with $\varphi_i\leftarrow w_i$ and $x\leftarrow w$ for arbitrary $w\in\H$ yields the desired equation for $S_\pos^{-1}A_r^{\opt,(1)} S_\y$.
Since $\sqrt\frac{-\lambda_i}{(1+\lambda_i)^3}=\sqrt\frac{-\lambda_i}{1+\lambda_i}\left(1+\frac{-\lambda_i}{1+\lambda_i}\right)$,
\begin{align*}
	\widetilde{A}_r^{\opt,(1)} S_\y = S_\pos^{-1} A_r^{\opt,(1)} S_y
	&= \sum_{i>r}^{}\sqrt{\frac{-\lambda_i}{1+\lambda_i}}^3 w_i\otimes \varphi_i + \sum_{i=1}^{n} \sqrt{\frac{-\lambda_i}{1+\lambda_i}} w_i\otimes \varphi_i \\
	&= \sum_{i>r}^{}\sqrt{\frac{-\lambda_i}{1+\lambda_i}}^3 w_i\otimes \varphi_i + \mathcal{C}_\pr^{1/2}G^*\mathcal{C}_\obs^{-1/2},
\end{align*}
where the last equation follows from \eqref{eqn:svd_preconditioned_hessian}. We conclude, by definition of the Hilbert--Schmidt norm,
	\begin{align*}
		\Norm{ \widetilde{A}^{\opt,(1)}_r S_\y-\mathcal{C}_\pr^{1/2}G^*\mathcal{C}_\obs^{-1/2}}_{L_2(\mathcal{H})}^2 = \sum_{i>r}^{}\sqrt{\frac{-\lambda_i}{1+\lambda_i}}^6.
	\end{align*}

	Finally, by \Cref{prop:mean_approximation_problem_equivalence}\ref{item:mean_approximation_problem_equivalence_3}-\ref{item:mean_approximation_problem_equivalence_4} and \Cref{lemma:mean_equivalent_loss} it holds that \Cref{prob:optimal_mean} has a unique solution if and only if \eqref{eqn:min_hs_problem_mean_update} has a unique solution. As described above, $\widetilde{A}$ solves \eqref{eqn:min_hs_problem_mean_update} if and only if $\widetilde{A}=(S_\pos^{-1}\mathcal{C}_\pr-\widetilde{B})G^*\mathcal{C}_\obs^{-1}$ for some $\widetilde{B}$ solving \eqref{eqn:min_hs_problem_mean_update_2}. Thus, \eqref{eqn:min_hs_problem_mean_update} has a unique solution if and only if any two solutions $\widetilde{B}_1$ and $\widetilde{B}_2$ of \eqref{eqn:min_hs_problem_mean_update_2} satisfy $\widetilde{B}_1G^*\mathcal{C}_\obs^{-1}S_\y = \widetilde{B}_2G^*\mathcal{C}_\obs^{-1}S_\y$. By \Cref{rmk:equivalent_uniqueness_statement} with the above choices of $M$, $T$ and $S$, any two solutions $\widetilde{B}_1$ and $\widetilde{B}_2$ of \eqref{eqn:min_hs_problem_mean_update_2} satisfy $\widetilde{B}_1G^*\mathcal{C}_\obs^{-1}S_\y = \widetilde{B}_2G^*\mathcal{C}_\obs^{-1}S_\y$ if and only if $\sigma_{r+1}=0$ or $\sigma_r>\sigma_{r+1}$, where $\sigma_i\coloneqq \sqrt{-\lambda_i(1+\lambda_i)^{-3}} - \sqrt{-\lambda_i(1+\lambda_i)^{-1}}=\sqrt{-\lambda_i^3(1+\lambda_i)^{-3}}$ is the $i$-th singular value of $MP_{\ker{S}^\perp}=M$.  In turn, this holds if and only if $\lambda_{r+1}=0$ or $\lambda_r<\lambda_{r+1}$, because $(\lambda_i)_i\subset(-1,0]$ is a nonincreasing sequence by \Cref{prop:bayesian_feldman_hajek} and $x\mapsto \sqrt{-x(1+x)^{-1}}^{3}$ is decreasing on $(-1,\infty)$. This concludes the proof of uniqueness.
\end{proof}

\optimalMeanForAmariAndHellinger*
\begin{proof}
	Let $g_{\am,\alpha}$ and $g_{\hel}$ be as in \eqref{eqn:tranform_amari_hellinger}.
	By \Cref{rmk:hellinger_and_amari_divergence}, Jensen's inequality, and \Cref{thm:optimal_low_rank_mean_approx,thm:optimal_low_rank_update_mean_approx},
	\begin{align*}
		\mathbb{E}\left[D_{\am,\alpha}(\mathcal{N}(A^{\opt,(i)}_rY,{\mathcal{C}}_\pos)\Vert\mu_\pos)\right]
		&= \mathbb{E}\left[g_\alpha\left(D_{\ren,\alpha}(\mathcal{N}(A^{\opt,(i)}_rY,{\mathcal{C}}_\pos)\Vert\mu_\pos)\right)\right] \\
		&\leq g_\alpha\left(\mathbb{E}\left[D_{\ren,\alpha}(\mathcal{N}(A^{\opt,(i)}_rY,{\mathcal{C}}_\pos)\Vert\mu_\pos)\right]\right) \\
		&= \frac{-1}{\alpha(1-\alpha)}\left(\exp\left(-\frac{\alpha(1-\alpha)}{2}\sum_{j>r}^{}\left(\frac{-\lambda_j}{1+\lambda_j}\right)^{\gamma(i)}\right)-1\right).
	\end{align*}
	The case for the forward Amari-$\alpha$ divergence follows analogously. For the Hellinger distance, we invoke once more \Cref{rmk:hellinger_and_amari_divergence}, Jensen's inequality, and \Cref{thm:optimal_low_rank_mean_approx,thm:optimal_low_rank_update_mean_approx},
	\begin{align*}
		\mathbb{E}\left[D_{\hel}(\mu_\pos(Y),\mathcal{N}(A^{\opt,(i)}_rY,\mathcal{C}_\pos))\right]
		&= \mathbb{E}\left[g_\hel\left(D_{\ren,\frac{1}{2}}(\mu_\pos,\mathcal{N}(A^{\opt,(i)}_rY,\mathcal{C}_\pos))\right)\right]\\\
		&\leq g_\hel\left(\mathbb{E}\left[D_{\ren,\frac{1}{2}}(\mu_\pos,\mathcal{N}(A^{\opt,(i)}_rY,\mathcal{C}_\pos))\right]\right)\\\
		&= \sqrt{2\left(1-\exp\left(-\frac{1}{8}\sum_{j>r}^{}\left(\frac{-\lambda_j}{1+\lambda_j}\right)^{\gamma(i)}\right)\right)}.
	\end{align*}
\end{proof}

\interpretationOptimalMeans*
\begin{proof}[Proof of \Cref{lemma:interpretation_optimal_means}]
	For any realisation $y$ of $Y$, it holds that $m_\pos(y)\in\mathscr{M}_n^{(1)}\cap\mathscr{M}_n^{(2)}$, as discussed at the end of \Cref{sec:formulation}. Hence $A^{\opt,(i)}_ny=m_\pos(y)$ for $i=1,2$.
	Applying \Cref{thm:optimal_low_rank_mean_approx} with $r\leftarrow n$, we see that 
	\begin{align*}
		m_\pos(y) = A^{\opt,(2)}_n y = \mathcal{C}_\pr^{1/2}\biggr(\sum_{i=1}^{n}\sqrt{-\lambda_i(1+\lambda_i)}w_i\otimes \varphi_i\biggr)\mathcal{C}_\obs^{-1/2}y.
	\end{align*}
	For fixed $r\leq n$, it follows that for any $j\leq r$,
	\begin{align*}
		\langle A^{\opt,(2)}_ry,\mathcal{C}_\pr^{-1/2}w_j\rangle 
		&= \sum_{i=1}^{r} \sqrt{-\lambda_i(1+\lambda_i)} \langle \mathcal{C}_\pr^{1/2}w_i,\mathcal{C}_\pr^{-1/2}w_j\rangle\langle \varphi_i,\mathcal{C}_{\obs}^{-1/2}y\rangle \\
		&= \sum_{i=1}^{n} \sqrt{-\lambda_i(1+\lambda_i)} \langle \mathcal{C}_\pr^{1/2}w_i,\mathcal{C}_\pr^{-1/2}w_j\rangle\langle \varphi_i,\mathcal{C}_{\obs}^{-1/2}y\rangle \\
		&= \langle m_\pos(y),\mathcal{C}_\pr^{-1/2}w_j\rangle.
	\end{align*}
	Furthermore, $\langle A^{\opt,(2)}_ry,\mathcal{C}_\pr^{-1/2}w_j\rangle =0=\langle m_\pr,\mathcal{C}_\pr^{-1/2}w_j\rangle$ for $j>r$, since $m_\pr=0$. 
	Hence, $\langle A^{\opt,(2)}_ry,h\rangle=\langle m_\pos(y),h\rangle$ for all $h\in W_r$ and $\langle A^{\opt,(2)}_ry,h\rangle=\langle m_\pr,h\rangle$ for all $h\in\Span{\mathcal{C}_\pr^{-1/2}w_j,\ j>r}$, which is dense in $W_{-r}$. Thus, we have that $P_{W_r}A^{\opt,(2)}_ry = P_{W_r}m_\pos(y)$, and also that $P_{W_{-r}}A^{\opt,(2)}_ry = P_{W_{-r}}m_\pr$ by continuity of $h\mapsto\langle k,h\rangle$ for any $k\in\H$.

	Next, we note that $\mathcal{C}^\opt_n=\mathcal{C}_\pos$ by \Cref{rmk:interpretation_covariance_approx}. It follows from \Cref{thm:optimal_low_rank_update_mean_approx} with $r\leftarrow n$,
	\begin{align*}
		m_\pos(y) = A^{\opt,(1)}_n y = \mathcal{C}^\opt_n G^*\mathcal{C}_\obs^{-1}y = \mathcal{C}_\pos G^*\mathcal{C}_\obs^{-1}y.
	\end{align*}
	Hence, for $j\leq r$,
	\begin{align*}
		\langle A^{\opt,(1)}_ry,\mathcal{C}_\pr^{-1/2}w_j\rangle 
		&= \langle \mathcal{C}^\opt_r G^*\mathcal{C}_\obs^{-1}y, \mathcal{C}_\pr^{-1/2}w_j\rangle
		= \langle G^*\mathcal{C}_\obs^{-1}y, \mathcal{C}^\opt_r \mathcal{C}_\pr^{-1/2}w_j\rangle\\
		&= \langle G^*\mathcal{C}_\obs^{-1}y, \mathcal{C}_\pos \mathcal{C}_\pr^{-1/2}w_j\rangle
		= \langle \mathcal{C}_\pos G^*\mathcal{C}_\obs^{-1}y, \mathcal{C}_\pr^{-1/2}w_j\rangle
		= \langle m_\pos(y), \mathcal{C}_\pr^{-1/2}w_j\rangle,
	\end{align*}
	where we use consecutively the definition of $A^{\opt,(1)}_r$ of \Cref{thm:optimal_low_rank_update_mean_approx}, the self-adjoint property of $\mathcal{C}^\pos_r$, the fact that $\mathcal{C}^\opt_r \mathcal{C}_\pr^{-1/2}w_j=\mathcal{C}_\pos\mathcal{C}_\pr^{-1/2}w_j$ for $j\leq r$ by \Cref{rmk:interpretation_covariance_approx}, the self-adjoint property of $\mathcal{C}_\pos$, and the above expression of $m_\pos(y)$. 
	Using that  $\mathcal{C}^\opt_r\mathcal{C}_\pr^{-1/2}w_j = \mathcal{C}_\pr\mathcal{C}_\pr^{-1/2}w_j$ for $j>r$ by \Cref{rmk:interpretation_covariance_approx}, a similar computation for $j>r$ shows that $ \langle A^{\opt,(1)}_ry,\mathcal{C}_\pr^{-1/2}w_j\rangle =\langle \mathcal{C}_\pr G^*\mathcal{C}_\obs^{-1}y,\mathcal{C}_\pr^{-1/2}w_j \rangle$.
\end{proof}

\subsection{Proofs of Section \ref{sec:optimal_projector}}
\label{subsec:proofs_for_optimal_projector}

\optimalProjector*

\begin{proof}[Proof of \Cref{prop:optimal_projector}]
	Since ${P^\opt_r}\mathcal{C}_\pr^{1/2}w_i=\mathcal{C}_\pr^{1/2}w_i$ for $i\leq r$ and $\ran{P^\opt_r}=\Span{\mathcal{C}_\pr^{1/2}w_i,\ i\leq r}$, it holds that $(P^\opt_r)^2=P^\opt_r$, so that $P^\opt_r$ is indeed a projector of rank at most $r$. Let $(\widetilde{A}_ry,\widetilde{\mathcal{C}}_r)$ denote the posterior mean and covariance for the model \eqref{eqn:projected_observation_model} with $P_r\leftarrow P^\opt_r$. We first show that $\mathcal{C}^\opt_r=\widetilde{\mathcal{C}_r}$ by showing that $\widetilde{\mathcal{C}}_r^{-1}=(\mathcal{C}^\opt_r)^{-1}$. We then use this to show that $\widetilde{A}_r = A^{\opt,(2)}_r$.
	Since $P^\opt_r=\sum_{i=1}^{r}(\mathcal{C}_\pr^{1/2} w_i)\otimes (\mathcal{C}_\pr^{-1/2}w_i) = \mathcal{C}_\pr^{1/2} \sum_{i=1}^{r}w_i\otimes (\mathcal{C}_\pr^{-1/2}w_i)$, we have $(P^\opt_r)^* = \left(\sum_{i=1}^{r}(\mathcal{C}_\pr^{-1/2}w_i)\otimes w_i\right) \mathcal{C}_\pr^{1/2}$. Let $\varphi_i$ be the right eigenvector corresponding to $(\lambda_i,w_i)$ in \eqref{eqn:svd_preconditioned_hessian}.
	Using \eqref{eqn:svd_preconditioned_hessian} and the orthonormality of $(w_i)_i$, it follows that
	\begin{align}
		\nonumber
		(P^\opt_r)^*G^*\mathcal{C}_\obs^{-1/2} 
		&= \left(\sum_{i=1}^{r}(\mathcal{C}_\pr^{-1/2}w_i)\otimes w_i\right) \mathcal{C}_\pr^{1/2}G^* \mathcal{C}_\obs^{-1/2} \\
		\nonumber
		&= \left(\sum_{i=1}^{r}(\mathcal{C}_\pr^{-1/2}w_i)\otimes w_i\right)\left(\sum_{i}^{}\sqrt{\frac{-\lambda_i}{1+\lambda_i}} w_i\otimes \varphi_i \right) \\
		\label{eqn:projected_model_svd}
		&= \sum_{i=1}^{r}\sqrt{\frac{-\lambda_i}{1+\lambda_i}}(\mathcal{C}_\pr^{-1/2}w_i)\otimes \varphi_i.
	\end{align}
	Recall that ${H}$ defined in \eqref{eqn:hessian} is the Hessian of the negative log-likelihood of \eqref{eqn:observation_model}. Analogously, let $\widetilde{H}$ denote the Hessian of the negative log-likelihood of \eqref{eqn:projected_observation_model} with $P_r\leftarrow P^\opt_r$. 
	That is, upon replacement of $G$ with $GP^\opt_r$ in \eqref{eqn:hessian}, we obtain $\widetilde{H}$.
	Hence, orthonormality of $(\varphi_i)_i$ implies
	\begin{align*}
		\widetilde{H} &= (GP^\opt_r)^*\mathcal{C}_\obs^{-1}GP^\opt_r = (P^\opt_r)^*G^*\mathcal{C}_\obs^{-1}GP^\opt_r \\
		&=\left( \sum_{i=1}^{r}\sqrt{\frac{-\lambda_i}{1+\lambda_i}}(\mathcal{C}_\pr^{-1/2}w_i)\otimes \varphi_i)\right) \left(\sum_{i=1}^{r}\sqrt{\frac{-\lambda_i}{1+\lambda_i}}\varphi_i\otimes (\mathcal{C}_\pr^{-1/2}w_i) \right)\\
		&= \sum_{i=1}^{r}\frac{-\lambda_i}{1+\lambda_i}(\mathcal{C}_\pr^{-1/2}w_i)\otimes(\mathcal{C}_\pr^{-1/2}w_i).
	\end{align*}
	The analogue of the update \eqref{eqn:pos_precision} applied to the model \eqref{eqn:projected_observation_model} with $P_r\leftarrow P^\opt_r$, that is, \eqref{eqn:pos_precision} with $G$ replaced by $GP^\opt_r$, then implies $\ran{\widetilde{\mathcal{C}}_r}=\ran{\mathcal{C}_\pr}$ and $\widetilde{\mathcal{C}}_r^{-1} = \mathcal{C}_\pr^{-1}+\widetilde{H}$. 
	By \Cref{thm:opt_covariance_and_precision}, $\ran{\mathcal{C}^\opt_r=\ran{\mathcal{C}_\pr}}$.
	Hence $\ran{\widetilde{\mathcal{C}}_r}=\ran{\mathcal{C}^\opt_r}$. By the above expression of $\widetilde{H}$ and the expression of $(\mathcal{C}^\opt_r)^{-1}$ in \Cref{thm:opt_covariance_and_precision}, 
	\begin{align*}
		\widetilde{\mathcal{C}}_r^{-1} &= \mathcal{C}_\pr^{-1} + \widetilde{H}
		= \mathcal{C}_\pr^{-1} + \sum_{i=1}^{r}\frac{-\lambda_i}{1+\lambda_i}(\mathcal{C}_\pr^{-1/2}w_i)\otimes (\mathcal{C}_\pr^{-1/2}w_i)
		= (\mathcal{C}^\opt_r)^{-1}.
	\end{align*}
	Taking inverses shows that $\widetilde{\mathcal{C}}_r = \mathcal{C}^\opt_r$. The analogue of \eqref{eqn:pos_mean} applied to model \eqref{eqn:projected_observation_model} with $P_r\leftarrow P^\opt_r$, i.e.\ with $G$ replaced by $GP^\opt_r$, shows $\widetilde{A}_r = \widetilde{\mathcal{C}}_r (GP^\opt_r)^*\mathcal{C}_\obs^{-1}=\mathcal{C}^\opt_r (P^\opt_r)^*G^*\mathcal{C}_{\obs}^{-1}$. By \eqref{eqn:optimal_covariance} and \eqref{eqn:projected_model_svd},
	\begin{align*}
		\widetilde{A}_r
		&= \left( \mathcal{C}_\pr - \sum_{i=1}^{r}-\lambda_i (\mathcal{C}_\pr^{1/2} w_i)\otimes (\mathcal{C}_\pr^{1/2} w_i) \right)(P^\opt_r)^*G^*\mathcal{C}_\obs^{-1} \\
		&= \mathcal{C}_\pr^{1/2}\left(I - \sum_{i=1}^{r}-\lambda_i w_i\otimes w_i \right)\mathcal{C}_\pr^{1/2} (P^\opt_r)^*G^*\mathcal{C}_\obs^{-1} \\
		&= \mathcal{C}_\pr^{1/2}\left(I - \sum_{i=1}^{r}-\lambda_i w_i\otimes w_i \right)\left( \sum_{i=1}^{r}\sqrt{\frac{-\lambda_i}{1+\lambda_i}}w_i\otimes \varphi_i\right)\mathcal{C}_\obs^{-1/2}.
	\end{align*}
	Since $\left(I-\sum_{i=1}^{r}-\lambda_iw_i\otimes w_i\right)h=\sum_{i=1}^{r}(1+\lambda_i)\langle h,w_i\rangle w_i$ by the fact that $h=\sum_{i}^{}\langle w_i,h\rangle w_i$, we obtain
		\begin{align*}
			\widetilde{A}_r= \mathcal{C}_\pr^{1/2}\sum_{i=1}^{r}\sqrt{-\lambda_i(1+\lambda_i)}w_i\otimes \varphi_i \mathcal{C}_\obs^{-1/2} = A^{\opt,(2)}_r,
		\end{align*}
		where the last equality follows from \Cref{thm:optimal_low_rank_mean_approx}.
\end{proof}

\section{Examples}
\label{sec:examples}

In this section we consider the two examples of the linear Gaussian inverse problems given in \Cref{sec:examples_short} in detail. 
In both examples, $(\H,\langle\cdot,\cdot\rangle)=L^2([0,1])\simeq L^2( (0,1))$. We identify the operators in the formulation of \Cref{sec:formulation}. We also describe the prior-preconditioned Hessian $\mathcal{C}_\pr^{1/2}G^*\mathcal{C}_\obs^{-1}G\mathcal{C}_\pr^{1/2}$ and its square root $\mathcal{C}_\pr^{1/2}G^*\mathcal{C}_\obs^{-1/2}$ in \eqref{eqn:svd_preconditioned_hessian}. The eigendecomposition of the prior-preconditioned Hessian can be used in the construction of the optimal projector in \Cref{sec:optimal_projector}, and the SVD of \eqref{eqn:svd_preconditioned_hessian} can be used to form the optimal posterior mean approximations. If $(\frac{-\lambda_i}{1+\lambda_i},w_i)$ is an eigenpair of $\mathcal{C}_\pr^{1/2}G^*\mathcal{C}_\obs^{-1}G\mathcal{C}_\pr^{1/2}$, then $(\frac{-\lambda_i}{1+\lambda_i},\mathcal{C}_\obs^{-1/2}G\mathcal{C}_\pr^{1/2}w_i)$ is an eigenpair of $\mathcal{C}_\obs^{-1/2}G\mathcal{C}_\pr G^*\mathcal{C}_\obs^{-1/2}$, c.f.\ \Cref{lemma:eigenpairs_of_squares_relation}, so that the $(\varphi_i)_i$ occurring in \Cref{thm:optimal_low_rank_mean_approx} can be computed using the eigenpairs of the prior-preconditioned Hessian. Alternatively, they can be obtained by forming \eqref{eqn:svd_preconditioned_hessian}.

\begin{example}[Deconvolution]
	Let $\H=L^2([0,1])$ and let $\kappa:[0,1]^2\rightarrow\R$ be square integrable. We consider the convolution of functions in $L^2([0,1])$ with kernel $\kappa$, and hence define the convolution operator $T_\kappa\in\B(\H)$ by, for almost every $t\in[0,1]$,
	\begin{align*}
		(T_\kappa h )(t) = \int_0^1 \kappa (t,s) h(s)\d s,\quad h\in\H.
	\end{align*}
	Note that $T_\kappa$ is continuous by the integrability assumption on $\kappa$. 
	We consider the inverse problem in which the unknown parameter $x^\dagger\in L^2([0,1])$ is convolved by $T_\kappa\in\B(\H)$, and the goal is to recover $x^\dagger$. We take the Bayesian perspective and put a centered Gaussian prior $\mu_\pr$ on $\mathcal{H}$. We specify the prior covariance below. The parameter is now denoted by $X\sim\mu_\pr$.

	We assume the data $y$ is obtained by observing weighted averages of $T_\kappa X$ on the $n$ intervals in $[0,1]$ separated by $t_1<\cdots <t_{n+1}$, that are corrupted with standard Gaussian noise. That is, $y_i = \int_{t_i}^{t_{i+1}} (T_\kappa X)(s) \gamma (s)\d s+ \zeta_i = \langle T_\kappa X,1_{[t_i,t_{i+1}]}\gamma\rangle + \zeta_i$ for some known weighting function $\gamma\in \mathcal{H}$ and for $\zeta_i\sim\mathcal{N}(0,1)$. 

	Let $\mathcal{O}\in\B(\H,\R^n)$ be defined by $\mathcal{O}h = (\langle h, 1_{[t_i,t_{i+1}]}\gamma\rangle)_{i=1}^n$. 
	Defining $G\coloneqq\mathcal{O}T_\kappa$, we can write the deconvolution problem in the formulation \eqref{eqn:observation_model}, with $\mathcal{C}_\obs=I$.

	We construct the prior distribution $\mu_\pr$ of $X$ by using the Karhunen--Lo\`{e}ve expansion $X\coloneqq\sum_{i=1}^{\infty}c_i\xi_i e_i$. Here, $c\in\ell^2( (0,\infty))$, $(e_i)_i$ forms an ONB of $\H$, and $(\xi_i)_i$ is a sequence of independent $\mathcal{N}(0,1)$-distributed random variables. Then $\mu_\pr=\mathcal{N}(0,\mathcal{C}_\pr)$ with injective covariance $\mathcal{C}_\pr = \sum_{i}^{}c_i^2 e_i\otimes e_i\in L_1(\H)$. 	

	To compute the Hessian $H=G^*\mathcal{C}_\obs^{-1}G=G^*G$, we compute $G^*\in\B(\R^n,\H)$ by observing that $G^*=T_\kappa^*\mathcal{O}^*$ and
	\begin{align*}
		T_\kappa^* k &= \int_0^1 \kappa (t,\cdot)k(t)\d t,\quad k\in\H,
		\qquad\mathcal{O}^*z =  \sum_{i=1}^{n}1_{[t_i,t_{i+1}]}\gamma z_i,\quad z\in\R^n.
	\end{align*}
	Hence $G^* z =  \sum_{i=1}^{n} z_i\int_{}^{}\kappa(t,\cdot)1_{[t_i,t_{i+1}]}(t) \gamma(t)\d t.$
	In this way, we can formulate the deconvolution problem as a linear Gaussian inverse problem with observation model \eqref{eqn:observation_model}, and compute the Hessian $H$ defined in \eqref{eqn:hessian} by $Hh = G^*Gh = \sum_{i=1}^{n}\langle T_\kappa h,1_{[t_i,t_{i+1}]}\gamma\rangle\int\kappa(t,\cdot)1_{[t_i,t_{i+1}]}(t)\gamma(t)\d{t}$.

	Let us now assume that $\kappa$ is bounded and symmetric, and satisfies $\int \kappa(s,t)h(s)h(t)\geq 0$ for all $h\in\H$. Hence, $T_\kappa$ is self-adjoint and nonnegative. Then by Mercer's theorem, \cite[Theorem 3.a.1]{Konig1986}, we have $\kappa(s,t) = \sum_{i=1}^{\infty}b_i f_i(s)f_i(t)$, where the series converges absolutely and uniformly for almost every $(t,s)$. Here, $(b_i)_i$ is a nonnegative sequence converging to zero and $(f_i)_i$ is an ONB of $\H$ consisting of bounded functions. Furthermore, we may write $T_\kappa = \sum_{i}^{}b_i f_i\otimes f_i$. 
	For simplicity, we assume that the eigenvectors $(e_i)_i$ of the prior covariance and the eigenfunctions $(f_i)_i$ of the kernel are the same.
	One can verify that, with $a_{k,j}\coloneqq \langle f_k,1_{[t_j,t_{j+1}]}\gamma\rangle$, we have $\langle T_\kappa h,1_{[t_i,t_{i+1}]}\gamma\rangle = \sum_{j}^{}b_ja_{j,i}\langle f_j,h\rangle$ and $ \int\kappa(t,\cdot)1_{[t_i,t_{i+1}]}(t)\gamma(t)\d{t} = \sum_{k}^{}b_ka_{k,i}f_k$.
	Thus, $\mathcal{C}_\pr^{1/2}G^*\mathcal{C}_\obs^{-1/2}z=\sum_{i=1}^{n}\sum_{j}^{}z_ib_jc_ja_{j,i}f_j$ for $z\in\R^n$. Furthermore, $G^*G= \sum_{i=1}^{n}\sum_{j,k}^{} b_jb_k a_{j,i}a_{k,i} f_k\otimes f_j$ and hence the prior-preconditioned Hessian now takes the form 
	\begin{align*}
		\mathcal{C}_\pr^{1/2} H\mathcal{C}_\pr^{1/2} = \sum_{i=1}^{n}\sum_{j,k}^{} b_jc_jb_kc_k a_{j,i}a_{k,i} f_k\otimes f_j = \sum_{j,k}^{}d_{k,j}f_k\otimes f_j,
	\end{align*}
	where the coefficients $d_{k,j}=b_jc_jb_kc_k\sum_{i=1}^{n}a_{j,i}a_{k,i}$ and orthonormal sequence $(f_j)_j$ are explicitly known and depend on the choice of prior via $(c_i)_i$, on the kernel via $(f_k)_k$ and $(b_i)_i$, and on the observation model via $\gamma$.
\end{example}

\begin{example}[Inferring the initial condition of the heat equation]
	Let $u$ denote the solution of the heat equation on the one-dimensional spatial domain $(0,1)$ with boundary $\{0,1\}$ and time domain $[0,T]$. Thus, the temperature field $(x,t)\mapsto u(x,t)$ on $(0,1)\times[0,T]$ solves,
	\begin{align*}
		\partial_t u - \partial_{xx} u &=  0, &&\text{in } (0,1)\times(0,T),\\
		u(\cdot,0) &= x^\dagger, \quad &&\text{on } (0,1),\\
		u(0,\cdot) =u(1,\cdot) &= 0, && \text{on }(0,T],
	\end{align*}
	where the true initial condition $x^\dagger$ is unknown and where we impose
	a homogenous Dirichlet spatial boundary condition. We assume that the data consists of a noisy observation of $u$ at the observation coordinates $(x_i,t_i)_{i=1}^n\subset(0,1)\times(0,T]$, where we assume i.i.d.\ standard Gaussian noise. The aim is to reconstruct the initial condition $x^\dagger$ from the data $y$. This problem is similar to \cite[Example 3.5]{Stuart2010} and \cite[Section 4.2]{Flath2011}, but in this example we do not observe the temperature field over the entire spatial domain at finitely many times. Instead, we observe the temperature only at finitely many space-time points $(x_i,t_i)_{i=1}^{n}$. Furthermore, \cite[Section 4.2]{Flath2011} considers periodic boundary conditions instead of Dirichlet boundary conditions. We take the Bayesian perspective by considering $x^\dagger$ as an $\H$-valued random variable $X$ with centered Gaussian distribution $\mu_\pr$. Below, we choose an explicit form of the prior covariance $\mathcal{C}_\pr$ as a negative power of the Laplacian.

	To write this problem in the formulation of \Cref{sec:formulation}, we define $\H\coloneqq L^2((0,1))$. Let us denote by $H^1((0,1))$ the Sobolev space of square-integrable functions $h$ on $(0,1)$ that have a square-integrable weak derivative $\partial_x h$, which is a Hilbert space with the inner product $\langle h_1,h_2\rangle_1 \coloneqq \langle h_1,h_2 \rangle + \langle \partial_x h_1,\partial_x h_2 \rangle$, $h_1,h_2\in H^1((0,1))$. By \cite[Theorem 5.6.5]{Evans2010}, we have the continuous embedding $H^1((0,1))\subset C([0,1])$, where $C([0,1])$ denotes the space of continuous functions on $[0,1]$ with the supremum norm. Hence, for any $h\in H^1((0,1))$ and $x\in [0,1]$, we have $\abs{h(x)}\leq \norm{h}_{C([0,1])}\leq c\norm{h}_1$ for some $c>0$, so that pointwise evaluation is well-defined, linear and continuous on $H^1((0,1))$. Thus, $H^1((0,1))$ is a reproducing kernel Hilbert space. We denote the Riesz representatives of the pointwise evaluation functionals, or `features', by $\{\phi(x)\in H^1((0,1)),\ x\in[0,1]\}$. Hence, $h(x)=\langle h,\phi(x)\rangle_1$ for all $x\in[0,1]$ and $h\in H^1((0,1))$. For our choice of spatial domain $(0,1)$, we have the following explicit form for the features, by \cite[Corollary 2]{Thomas-Agnan1996}:
	\begin{align*}
		\phi(x)(x') = \frac{\cosh(x-1)\cosh(x')}{\sinh(1)},\quad 0\leq x'\leq x\leq 1,\\
		\phi(x)(x') = \frac{\cosh(x'-1)\cosh(x)}{\sinh(1)},\quad 0\leq x\leq x'\leq 1.
	\end{align*}
	We also define $H^1_0((0,1))\coloneqq \{h\in H^1((0,1)):\ h(0)=0=h(1)\}$, the space of functions $h\in H^1((0,1))$ which vanish on the boundary $\{0,1\}$.

	We use certain properties of $\Delta\coloneqq \partial_{xx}$, the one-dimensional Laplacian. We describe these briefly, and refer to \cite[Section 5.3]{Kretschmann2019} for a comprehensive treatment of these properties and their relation to the heat equation. By \cite[Theorem 8.22]{Brezis2011}, we can write $\Delta h=-\sum_{i}^{}a_i \langle h,e_i\rangle e_i$ for $h\in\dom{\Delta}=\{h\in L_2((0,1)):\ \sum_{i}^{}a_i^2\langle h,e_i\rangle^2<\infty\}$, where $\lim_i a_i=\infty$ and $(e_i)_i$ is an ONB on $\H$. In fact, by the example on \cite[p. 232]{Brezis2011}, we have $a_i=i^2\pi^2$ and $e_i(x)=\sqrt{2}\sin(i\pi x)$ for our choices of spatial domain $(0,1)$ and boundary conditions. 
	Now, one can define the self-adjoint operator $\exp(t\Delta)\in\B_0(\H)$ by $\exp(t\Delta)=\sum_i\exp(-ta_i)e_i\otimes e_i$. It holds that $\ker{\exp(t\Delta)}=\{0\}$ and $\ran{\exp(t\Delta)=\H}$. The diagonalisation of the Laplacian is compatible with $H^1_0((0,1))$ in the sense that $H^1_0((0,1))=\{h\in\H:\ \sum_{i}^{}a_i\langle h,e_i\rangle^2<\infty\}$ and $\langle h,k\rangle_1=\sum_{i}^{}(1+a_i)\langle h,e_i\rangle\langle e_i,k\rangle$ for $h,k\in H^1_0((0,1))$. Since for any $t\in(0,T]$ and $h\in \H$, we have $\langle \exp(t\Delta)h,e_i\rangle=\exp(-a_i t)\langle h,e_i\rangle$, it follows that $\sum_{i}^{}a_i\langle\exp(t\Delta)h,e_i\rangle^2\leq C(t)\sum_{i}^{}\langle h,e_i\rangle^2$ for some $C(t)>0$, so that $\exp(t\Delta)h\in\H^1_0((0,1))$. Therefore, the map $h\mapsto \exp(\Delta t)h$, $\H\rightarrow H^1_0((0,1))$ is linear and continuous for $t\in(0,T]$. Furthermore, $\exp(\Delta t)\in\B(H^1_0((0,1)))$ for each $t\in[0,T]$, and $\exp(\Delta t)$ is a self-adjoint element of $\B(H^1_0((0,1)))$, because
	\begin{align*}
		\langle \exp(t\Delta)h,k\rangle_1 =  \sum_{i}^{}(1+a_i)\langle \exp(t\Delta)h,e_i\rangle\langle e_i,k\rangle
		= \sum_{i}^{}(1+a_i)\exp(-ta_i)\langle h,e_i\rangle\langle e_i,k\rangle,
	\end{align*}
	is symmetric in $h,k\in H^1_0((0,1))$.

	By \cite[Theorem 10.1]{Brezis2011}, the solution $u$ of the heat equation above lies in $C( (0,T];H^1_0((0,1)))$, and in fact $u(\cdot,t)$ has infinitely many continuous derivatives for each $t\in(0,T]$. By \cite[Section 4.1]{Pazy1983}, the solution can be written as $t\mapsto\exp(t\Delta)x^\dagger$.
	Let us define the linear map $g_i:\H\rightarrow\R$ by $g_i(h) = (\exp(t_i\Delta)h)(x_i)$ for each $i$. Since $g_i$ is the composition of the linear and continuous maps $u\mapsto u(\cdot,t_i)$, $C( (0,T]; H^1_0((0,1)))\rightarrow H^1_0((0,1))$ and $f\mapsto f(x_i)$, $H^1_0((0,1))\rightarrow \R$, it follows that $g_i$ is linear and continuous.
	Then, with $G\in\B(\H,\R^n)$ defined by $Gh\coloneqq (g_ih)_{i=1}^n$, and with $\zeta\sim\mathcal{N}(0,\mathcal{C}_{\obs})$ where $\mathcal{C}_\obs=I$, this inverse problem is of the form \eqref{eqn:observation_model}.

	For the prior $\mu_\pr$ on $\H$, we take $\mathcal{N}(0,\mathcal{C}_\pr)$ with $\mathcal{C}_\pr=(-\Delta)^{-s}$ for some $s>\frac{1}{2}$. Thus, $\mathcal{C}_\pr=\sum_{i}^{}a_i^{-s}{e}_i\otimes {e}_i$, which is injective and satisfies $\dom{\mathcal{C}_\pr}=\H$. Furthermore, $\mathcal{C}_\pr\in L_1(\H)$, since $\sum_{i}^{}a_i^{-s} = \pi^{-2s}\sum_{i}^{}i^{-2s}<\infty$.

	Next, we compute $G^*$, $H$ and $\mathcal{C}_\pr^{1/2}H\mathcal{C}_\pr^{1/2}$. Since $\langle \exp(t\Delta)h,k\rangle_1 = \langle h,\exp(t\Delta)k\rangle_1 $ for $h,k\in H^1_0((0,1))$ as shown above, we have for $z\in\R$ and $h\in H^1_0((0,1))$,
	\begin{align}
		\label{eqn:forward_map_component_adjoint}
		\begin{split}
			\langle z,g_i(h)\rangle_\R&=  z(\exp(t_i\Delta)h)(x_i) = z\langle \exp(t_i\Delta)h,\phi(x_i)\rangle_1 = z\langle h,\exp(t_i\Delta)\phi(x_i)\rangle_1\\
			&= z\langle h,\exp(t_i\Delta)\phi(x_i)\rangle + z\langle  \partial_x{h},\partial_x\exp(t_i\Delta)\phi(x_i)\rangle\\
			&= z\langle h,\exp(t_i\Delta)\phi(x_i) - \Delta\exp(t_i\Delta)\phi(x_i)\rangle,
		\end{split}
	\end{align}
	where we use consecutively the definition of the inner product on $\R$, the definition of $g_i$, the definition of $\phi(x_i)$, the fact that $\exp(t\Delta)$ is self-adjoint on $H^1_0((0,1))$, the definition of the $H^1((0,1))$ inner product and integration by parts.
	Hence,
	\begin{align*}
		g_i^*z &=  z\left(\exp(t_i\Delta)(\phi(x_i))-\Delta\exp(t_i\Delta)(\phi(x_i))\right),& &z\in\R,\\
		G^* z =  \sum_{i=1}^{n}g_i^*(z_i) &=  \sum_{i=1}^{n} z_i\left(\exp(t_i\Delta)(\phi(x_i))-\Delta\exp(t_i\Delta)(\phi(x_i))\right), & &z\in\R^n,\\
		H h = G^*G h &= \sum_{i=1}^{n} \left(\exp(t_i\Delta)h\right)(x_i) \bigg(\exp(t_i\Delta)(\phi(x_i))-\Delta\exp(t_i\Delta)(\phi(x_i))\bigg), & &h\in\H.
	\end{align*}
	The term $\exp(t_i\Delta)(\phi(x_i))$ is the solution of the heat equation in which the initial condition is given by the feature $\phi(x_i)\in\H$. 
	We have $\exp(t_i\Delta)e_j = \exp(-a_jt_i)e_j$. Thus, with $b_{i,j}\coloneqq a_j^{-s/2}\exp(-t_ia_j)$, we can write
	\begin{equation*}
    H\mathcal{C}_\pr^{1/2}h = \sum_{i=1}^{n}\sum_{j}^{} b_{i,j}\langle e_j,h\rangle e_j(x_i)\left(\exp(t_i\Delta)(\phi(x_i))-\Delta\exp(t_i\Delta)(\phi(x_i))\right).
    \end{equation*}
	By \eqref{eqn:forward_map_component_adjoint}, it holds for $z\in\R$ and $h\in H^1_0((0,1))$,
	\begin{align*}
		z\langle h,\exp(t_i\Delta)\phi(x_i) - \Delta\exp(t_i\Delta)\phi(x_i)\rangle = z(\exp(t_i\Delta)h)(x_i).
	\end{align*}
	Now, $e_k(x)=\sqrt{2}\sin{(k\pi x)}$ for each $k$, so that $e_k\in H^1_0((0,1))$. Substituting $z\leftarrow 1$ and $h\leftarrow e_k$ in the previous display, we obtain,
	\begin{align*}
		\langle e_k,\exp(t_i\Delta)\phi(x_i) - \Delta\exp(t_i\Delta)\phi(x_i)\rangle = (\exp(t_i\Delta)e_k)(x_i)  = \exp(-t_ia_k) e_k(x_i).
	\end{align*}
	It follows that
	\begin{align*}
		\mathcal{C}_\pr^{1/2}G^*\mathcal{C}_\obs^{-1/2}z &= \mathcal{C}_\pr^{1/2}\sum_{i=1}^{n}z_i\sum_{k}^{}\langle \exp(t_i\Delta)(\phi(x_i))-\Delta\exp(t_i\Delta)(\phi(x_i)),e_k\rangle e_k\\
		&= \mathcal{C}_\pr^{1/2}\sum_{i=1}^{n}z_i\sum_{k}^{}\exp{(-t_i a_k)}e_k(x_i)e_k\\
		&= \sum_{i=1}^{n}\sum_{k}^{}z_ia_k^{-s/2}\exp(-t_ia_k)e_k(x_i)e_k,\quad z\in\R^n,
	\end{align*}
where in the first step we use $\mathcal{C}_\obs=I$, the expression of $G^*$ above, and an expansion of $\exp(t_i\Delta)(\phi(x_i))-\Delta\exp(t_i\Delta)(\phi(x_i))$ in the ONB $(e_k)_k$. Furthermore,
	\begin{align*}
		\mathcal{C}_\pr^{1/2}H\mathcal{C}_\pr^{1/2}h = \sum_{i=1}^{n}\sum_{j,k}^{}b_{i,j}b_{i,k}\langle e_j,h\rangle e_j(x_i)e_k(x_i)e_k= \left( \sum_{j,k}^{}d_{j,k}e_k\otimes e_j\right) h,\quad h\in\H,
	\end{align*}
	where $d_{j,k}=\sum_{i=1}^{n}b_{i,j}b_{i,k}e_j(x_i)e_k(x_i)=\sum_{i=1}^{n}a_j^{-s/2}\exp(-t_ia_j)a_k^{-s/2}\exp(-t_ia_k)e_j(x_i)e_k(x_i)$. The coefficients $(d_{j,k})_{j,k}$ are explicitly available, since $a_i=i^2\pi^2$, $e_i(x)=\sqrt{2}\sin(i\pi x)$ and the observation coordinates $(x_i,t_i)_{i=1}^n$ are all known.

\end{example}

\bibliographystyle{abbrv}
\bibliography{references}

@article{Cui2016b,
  title = {Dimension-Independent Likelihood-Informed {{MCMC}}},
  author = {Cui, T. and Law, K. J. H. and Marzouk, Y. M.},
  year = {2016},
  journal = {J. Comput. Phys.},
  volume = {304},
  pages = {109--137},
  doi = {10.1016/j.jcp.2015.10.008},
}

@article{Cui2014,
doi = {10.1088/0266-5611/30/11/114015},
year = {2014},
volume = {30},
number = {11},
pages = {114015}, 
author = {T. Cui and J. Martin and Y. Marzouk and A. Solonen and A. Spantini},
title = {Likelihood-informed dimension reduction for nonlinear inverse problems},
journal = {Inverse Problems},
fjournal = {Inverse Problems}, 
}

@article{Cui2022,
doi = {10.1088/1361-6420/ac9582},
year = {2022},
volume = {38},
number = {12},
pages = {124002},
author = {T. Cui and X. T. Tong and O. Zahm},
title = {Prior normalization for certified likelihood-informed subspace detection of {B}ayesian inverse problems},
journal = {Inverse Problems},
}

@article{CuiTong2022,
author = {T. Cui and X. T. Tong},
title = {{A unified performance analysis of likelihood-informed subspace methods}},
volume = {28},
journal = {Bernoulli},
number = {4},
year = {2022},
pages ={2788-2815},
doi = {10.3150/21-BEJ1437},
}

@article{Flath2011,
author = {Flath, H. P. and Wilcox, L. C. and Akcelik, V. and Hill, J. and Van Bloemen Waanders, B. and Ghattas, O.},
title = {Fast algorithms for {B}ayesian uncertainty quantification in large-scale linear inverse problems based on low-rank partial {H}essian approximations},
fjournal = {SIAM Journal on Scientific Computing},
journal = {SIAM J. Sci. Comput.},
volume = {33},
number = {1},
pages = {407-342},
year = {2011},
doi = {10.1137/090780717},
}

@article{Stuart2010,
  author={Stuart, A. M.},
  year={2010},
  title={Inverse problems: A {B}ayesian perspective},
  volume={19}, 
  DOI={10.1017/S0962492910000061},
  fjournal={Acta Numerica},
  journal={Acta Numer.}, 
  pages={451–559},
}

@article{Spantini2017,
author = {Spantini, A. and Cui, T. and Willcox, K. and Tenorio, L. and Marzouk, Y.},
title = {Goal-Oriented Optimal Approximations of {B}ayesian Linear Inverse Problems},
fjournal = {SIAM Journal on Scientific Computing},
journal = {SIAM J. Sci. Comput.},
volume = {39},
number = {5},
pages = {S167-S196},
year = {2017},
doi = {10.1137/16M1082123},
}

@article {Zahm2022,
    AUTHOR = {Zahm, O. and Cui, T. and Law, K. and Spantini, A. and Marzouk, Y.},
     TITLE = {Certified dimension reduction in nonlinear {B}ayesian inverse problems},
   JOURNAL = {Math. Comp.},
  FJOURNAL = {Mathematics of Computation},
    VOLUME = {91},
      YEAR = {2022},
    NUMBER = {336},
     PAGES = {1789--1835},
       DOI = {10.1090/mcom/3737},
}

@article{Li2024a,
author = {M. T. C. Li and Y. Marzouk and O. Zahm},
title = {{Principal feature detection via $\phi$-Sobolev inequalities}},
volume = {30},
journal = {Bernoulli},
number = {4},
pages = {2979 -- 3003},
year = {2024},
doi = {10.3150/23-BEJ1702},
}

@book {Amari2016,
    AUTHOR = {Amari, S.},
     TITLE = {Information geometry and its applications},
    SERIES = {Applied Mathematical Sciences},
    VOLUME = {194},
 PUBLISHER = {Springer, Tokyo},
      YEAR = {2016},
     PAGES = {xiii+374},
       DOI = {10.1007/978-4-431-55978-8},
}

@article {Nielsen2022,
    AUTHOR = {Nielsen, F.},
     TITLE = {The many faces of information geometry},
   JOURNAL = {Notices Amer. Math. Soc.},
  FJOURNAL = {Notices of the American Mathematical Society},
    VOLUME = {69},
      YEAR = {2022},
    NUMBER = {1},
     PAGES = {36--45},
       DOI = {10.1090/noti2403},
}

@eprint{CarereLie2024,
    title = {Generalised rank-constrained approximations of {Hilbert-Schmidt} operators on separable {H}ilbert spaces and applications},
    year={2024},
    author={G. Carere and H. C. Lie},
    eprint={2408.05104},
}

@article{Spantini2015,
author = {Spantini, A. and Solonen, A. and Cui, T. and Martin, J. and Tenorio, L. and Marzouk, Y.},
title = {Optimal Low-rank Approximations of {B}ayesian Linear Inverse Problems},
fjournal = {SIAM Journal on Scientific Computing},
journal = {SIAM J. Sci. Comput.},
volume = {37},
number = {6},
pages = {A2451-A2487},
year = {2015},
doi = {10.1137/140977308},
}

@eprint{Li2024b,
    title = {Sharp detection of low-dimensional structure in probability measures via dimensional logarithmic {S}obolev inequalities},
    year={2024},
    author = {Li, M. T. C. and Cui, T. and Li, F. and Marzouk, Y. and Zahm, O.},
    eprint={2406.13036},
}

@book{Simon2005,
  title = {Trace {{Ideals}} and {{Their Applications}}},
  author = {Simon, B.},
  year = {2005},
  series = {Mathematical {{Surveys}} and {{Monographs}}},
  edition = {second},
  volume = {120},
  publisher = {American Mathematical Society, Providence}, 
  pages = {viii+150},
  doi = {10.1090/surv/120},
}

@book{Hsing2015,
  title = {Theoretical {{Foundations}} of {{Functional Data Analysis}}, with an {{Introduction}} to {{Linear Operators}}},
  author = {Hsing, T. and Eubank, R.},
  year = {2015},
  series = {Wiley {{Series}} in {{Probability}} and {{Statistics}}},
  publisher = {John Wiley \& Sons, Ltd, Hoboken},
  pages = {xiv+334},
  doi = {10.1002/9781118762547},
}

@eprint{Hairer2023,
  title = {An {{Introduction}} to {{Stochastic PDEs}}},
  author = {Hairer, Martin},
  year = {2023},
  eprint = {0907.4178},
}

@book{da_prato_stochastic_2014,
  title = {Stochastic {{Equations}} in {{Infinite Dimensions}}},
  author = {Da Prato, G. and Zabczyk, J.},
  year = {2014},
  series = {Encyclopedia of {{Mathematics}} and Its {{Applications}}},
  edition = {second},
  publisher = {Cambridge University Press},
  doi = {10.1017/CBO9781107295513},
}

@article{Pinski2015,
	author = {Pinski, F. J. and Simpson, G. and Stuart, A. M. and Weber, H.},
	title = {Kullback--{Leibler} Approximation for Probability Measures on Infinite Dimensional Spaces},
	fjournal = {SIAM Journal on Mathematical Analysis},
	journal = {SIAM J. Math. Anal.},
	volume = {47},
	number = {6},
	pages = {4091-4122},
	year = {2015},
	doi = {10.1137/140962802},
}

@article{minh_regularized_2021,
  title = {Regularized {{Divergences Between Covariance Operators}} and {{Gaussian Measures}} on {{Hilbert Spaces}}},
  author = {Minh, H. Q.},
  year = {2021},
  journal = {J. Theor. Probab.},
  volume = {34},
  number = {2},
  pages = {580--643},
  doi = {10.1007/s10959-020-01003-2},
}

@article{simon_notes_1977,
  title = {Notes on Infinite Determinants of {{Hilbert}} Space Operators},
  author = {Simon, B.},
  year = {1977},
  journal = {Adv. Math.},
  fjournal = {Advances in Mathematics},
  volume = {24},
  number = {3},
  pages = {244--273},
  doi = {10.1016/0001-8708(77)90057-3},
}

@book{conway_course_2007,
  title = {A {{Course}} in {{Functional Analysis}}},
  author = {Conway, J. B.},
  year = {2007},
  series = {Graduate {{Texts}} in {{Mathematics}}},
  volume = {96},
  publisher = {Springer},
  url = {http://link.springer.com/10.1007/978-1-4757-4383-8},
}

@book{bogachev_gaussian_1998,
  title = {Gaussian Measures},
  author = {Bogachev, V.},
  year = {1998},
  series = {Mathematical {{Surveys}} and {{Monographs}}},
  volume = {62},
  publisher = {American Mathematical Society},
  doi = {10.1090/surv/062},
}

@book{engl_regularization_1996,
  title = {Regularization of {{Inverse Problems}}},
  author = {Engl, H. Werner and Hanke, M. and Neubauer, A.},
  year = {1996},
  fseries = {Mathematics and its Applications (Dordrecht)},
  series = {Math. Appl., Dordr.},
  edition = {first},
  volume = {375},
  year = {1996},
  publisher = {Dordrecht: Kluwer Academic Publishers},
}

@article{vanErven2014,
  title = {R{\'e}nyi {{Divergence}} and {{Kullback-Leibler Divergence}}},
  author = {{van Erven}, Tim and Harremos, Peter},
  year = {2014},
  journal = {IEEE Trans. Inform. Theory},
  fjournal = {Transactions on Information Theory},
  volume = {60},
  number = {7},
  pages = {3797--3820},
  doi = {10.1109/TIT.2014.2320500},
}

@article{Ray2022,
  title = {Variational {{Bayes}} for {{High-Dimensional Linear Regression With Sparse Priors}}},
  author = {Ray, Kolyan and Szab{\'o}, Botond},
  year = {2022},
  journal = {J. Amer. Statist. Assoc.},
  fjournal = {Journal of the American Statistical Association},
  volume = {117},
  number = {539},
  pages = {1270--1281},
  doi = {10.1080/01621459.2020.1847121},
}

@article{Sondermann1986,
  title = {Best Approximate Solutions to Matrix Equations under Rank Restrictions},
  author = {Sondermann, D.},
  year = {1986},
  journal = {Statistische Hefte},
  volume = {27},
  number = {1},
  pages = {57--66},
  doi = {10.1007/BF02932555},
}

@article{Friedland2007,
  title = {Generalized {{Rank-Constrained Matrix Approximations}}},
  author = {Friedland, Shmuel and Torokhti, Anatoli},
  year = {2007},
  journal = {SIAM J. Matrix Anal. Appl.},
  volume = {29},
  number = {2},
  pages = {656--659},
  doi = {10.1137/06065551},
}

@eprint{PartI,
  title = {Optimal low-rank posterior covariance approximation in linear {Gaussian} inverse problems on {Hilbert} spaces},
  author = {Carere, G. and Lie, H. C.},
  year = {2025},
  eprint = {2411.01112},
}

@book{Konig1986,
  title = {Eigenvalue {{Distribution}} of {{Compact Operators}}},
  author = {K{\"o}nig, Hermann},
  year = {1986},
  series = {Operator {{Theory}}: {{Advances}} and {{Applications}}},
  volume = {16},
  publisher = {Birkh{\"a}user},
  doi = {10.1007/978-3-0348-6278-3}
}

@book{Brezis2011,
  title = {Functional {{Analysis}}, {{Sobolev Spaces}} and {{Partial Differential Equations}}},
  author = {Brezis, Haim},
  year = {2011},
  publisher = {Springer},
  doi = {10.1007/978-0-387-70914-7}
}

@book{Evans2010,
  title = {Partial Differential Equations},
  author = {Evans, Lawrence C.},
  year = {2010},
  series = {Graduate {{Studies}} in {{Mathematics}}},
  edition = {second},
  volume = {19},
  publisher = {American Mathematical Society}
}

@article{Thomas-Agnan1996,
  title = {Computing a Family of Reproducing Kernels for Statistical Applications},
  author = {{Thomas-Agnan}, Christine},
  year = {1996},
  journal = {Numer. Algor.},
  fjournal = {Numerical Algorithms},
  volume = {13},
  number = {1},
  pages = {21--32},
  doi = {10.1007/BF02143124}
}

@book{Pazy1983,
  title = {Semigroups of {{Linear Operators}} and {{Applications}} to {{Partial Differential Equations}}},
  author = {Pazy, A.},
  year = {1983},
  series = {Applied {{Mathematical Sciences}}},
  volume = {44},
  publisher = {Springer},
  doi = {10.1007/978-1-4612-5561-1}
}

@phdthesis{Kretschmann2019,
  title = {Nonparametric {{Bayesian Inverse Problems}} with {{Laplacian Noise}}},
  author = {Kretschmann, Remo},
  year = {2019},
  school = {University of Duisburg-Essen},
  doi = {10.17185/duepublico/70452}
}

@book{ben-israel2003,
  title = {Generalized {{Inverses}}},
  author = {{Ben-Israel}, A. and Greville, T. N.E.},
  year = {2003},
  series = {{{CMS Books}} in {{Mathematics}}},
  publisher = {Springer-Verlag},
  doi = {10.1007/b97366}
}

@book{Reed1980,
  title = {Methods of Modern Mathematical Physics. {{I}}: {{Functional}} Analysis. {{Rev}}. and Enl. Ed},
  author = {Reed, M. and Simon, B.},
  year = {1980},
  series = {Methods of {{Modern Mathematical Physics}}},
  volume = {1},
  publisher = {Academic Press},
}

@article{Bui-Thanh2013,
  title = {A {{Computational Framework}} for {{Infinite-Dimensional Bayesian Inverse Problems Part I}}: {{The Linearized Case}}, with {{Application}} to {{Global Seismic Inversion}}},
  shorttitle = {A {{Computational Framework}} for {{Infinite-Dimensional Bayesian Inverse Problems Part I}}},
  author = {{Bui-Thanh}, Tan and Ghattas, Omar and Martin, James and Stadler, Georg},
  year = {2013},
  journal = {SIAM J. Sci. Comput.},
  volume = {35},
  number = {6},
  pages = {A2494-A2523},
  doi = {10.1137/12089586X}
}

@inproceedings{Bui-Thanh2012b,
  title = {Extreme-Scale {{UQ}} for {{Bayesian}} Inverse Problems Governed by {{PDEs}}},
  booktitle = {2012 {{Int}}. {{Conf}}. {{High Perform}}. {{Comput}}. {{Netw}}. {{Storage Anal}}.},
  author = {{Bui-Thanh}, Tan and Burstedde, Carsten and Ghattas, Omar and Martin, James and Stadler, Georg and Wilcox, Lucas C.},
  year = {2012},
  pages = {1--11},
  publisher = {IEEE},
  doi = {10.1109/SC.2012.56}
}

@book{Saad2003,
  title = {Iterative {{Methods}} for {{Sparse Linear Systems}} {\textbar} {{SIAM Publications Library}}},
  author = {Saad, Yousef},
  year = {2003},
  edition = {2nd ed.},
  publisher = {SIAM Society for Industrial {and} Applied Mathematics},
  url = {https://epubs.siam.org/doi/book/10.1137/1.9780898718003}
}

@article{Beskos2011,
  title = {Hybrid {{Monte Carlo}} on {{Hilbert}} Spaces},
  author = {Beskos, A. and Pinski, F. J. and {Sanz-Serna}, J. M. and Stuart, A. M.},
  year = {2011},
  fjournal = {Stochastic Processes and their Applications},
  journal = {Stoch. Proc. Appl.},
  volume = {121},
  number = {10},
  pages = {2201--2230},
  doi = {10.1016/j.spa.2011.06.003}
}

@article{Cotter2013,
  title = {{{MCMC Methods}} for {{Functions}}: {{Modifying Old Algorithms}} to {{Make Them Faster}}},
  shorttitle = {{{MCMC Methods}} for {{Functions}}},
  author = {Cotter, S. L. and Roberts, G. O. and Stuart, A. M. and White, D.},
  year = {2013},
  journal = {Statist. Sci.},
  volume = {28},
  number = {3},
  pages = {424-446},
  doi = {10.1214/13-sts421}
}

@article{Bui-Thanh2016,
  title = {{{FEM-based}} Discretization-Invariant {{MCMC}} Methods for {{PDE-constrained Bayesian}} Inverse Problems},
  author = {{Bui-Thanh}, Tan and Nguyen, Quoc P.},
  year = {Sat Oct 01 00:00:00 UTC 2016},
  journal = {Inverse Probl. Imaging},
  volume = {10},
  number = {4},
  pages = {943--975},
  doi = {10.3934/ipi.2016028}
}

@article{Ubaru2017,
  author = {Ubaru, Shashanka and Chen, Jie and Saad, Yousef},
  year = {2017},
  journal = {SIAM J. Matrix Anal. Appl.},
  volume = {38},
  number = {4},
  pages = {1075--1099},
  doi = {10.1137/16M1104974},
  title = {Fast {{Estimation}} of {$\textup{tr}(f(A))$} via {{Stochastic Lanczos Quadrature}}}
}

@article{Sanz-Alonso2024,
  title = {Analysis of a {{Computational Framework}} for {{Bayesian Inverse Problems}}: {{Ensemble Kalman Updates}} and {{MAP Estimators}} under {{Mesh Refinement}}},
  shorttitle = {Analysis of a {{Computational Framework}} for {{Bayesian Inverse Problems}}},
  author = {{Sanz-Alonso}, Daniel and Waniorek, Nathan},
  year = 2024,
  journal = {SIAM/ASA J. Uncertain. Quantif.},
  volume = {12},
number = {1},
pages = {30-68},
year = {2024},
doi = {10.1137/23M1567035},
}

@book{Quarteroni1994,
  title = {Numerical {{Approximation}} of {{Partial Differential Equations}}},
  author = {Quarteroni, Alfio and Valli, Alberto},
  year = 1994,
  series = {Springer {{Series}} in {{Computational Mathematics}}},
  volume = {23},
  publisher = {Springer},
  doi = {10.1007/978-3-540-85268-1}
}

@book{Thomee2006,
  title = {Galerkin {{Finite Element Methods}} for {{Parabolic Problems}}},
  author = {Thom{\'e}e, Vidar},
  year = 2006,
  series = {Springer {{Series}} in {{Computational Mathematics}}},
  edition = {Second},
  volume = {25},
  publisher = {Springer},
  doi = {10.1007/3-540-33122-0}
}

@book{Kaipio2005,
  title = {Statistical and {{Computational Inverse Problems}}},
  author = {Kaipio, Jari P. and Somersalo, Erkki},
  year = 2005,
  series = {Applied {{Mathematical Sciences}}},
  volume = {160},
  publisher = {Springer},
  doi = {10.1007/b138659}
}

@article{Hernandez2005,
  author = {Hernandez, Vicente and Roman, Jose E. and Vidal, Vicente},
  title = {{SLEPc}: A scalable and flexible toolkit for the solution of eigenvalue problems},
  publisher = {Association for Computing Machinery},
  year = {2005},
  journal = {ACM Trans. Math. Softw.},
  fjournal ={ACM Transactions on Mathematical Software},
  volume = {31},
  number = {3},
  pages = {351–362},
  doi = {10.1145/1089014.1089019},
}

@article{Dalcin2011,
  author = {Dalcin, Lisandro D. Dalcin and Paz, Rodrigo R. and Kler, Pablo A. and Cosimo, Alejandro},
  title = {Parallel distributed computing using Python},
  journal = {Adv. Water Resour.},
  fjournal = {Advances in Water Resources},
  volume = {34},
  number = {9},
  pages = {1124-1139},
  year = {2011},
  doi = {10.1016/j.advwatres.2011.04.013},
}

@techreport{Balay2025,
  author       = {Balay, S. and Abhyankar, S. and Adams, M. F. and Brown, J. and Brune, P. and Buschelman, K. and Constantinescu, E. M. and Dalcin, L. and Benson, S. and Dener, A. and others},
  title        = {{PETSc/TAO} Users Manual Revision 3.24},
  institution  = {Argonne National Laboratory (ANL), Argonne, IL (United States)},
  doi          = {10.2172/2998643},
  place        = {United States},
  year         = {2025},
  month        = {09}
}

@eprint{Baratta2023,
  title     = {{DOLFINx}: the next generation {FEniCS} problem solving environment},
  author    = {Baratta, Igor A. and Dean, Joseph P. and Dokken, J{\o}rgen S. and Habera, Michal and Hale, Jack S. and Richardson, Chris N. and Rognes, Marie E. and Scroggs, Matthew W. and Sime, Nathan and Wells, Garth N.},
  doi       = {10.5281/zenodo.10447666},
  year      = {2023},
  howpublished = {preprint}
}

@article{Alnaes2014,
  title     = {Unified Form Language: A domain-specific language for weak formulations of partial differential equations},
  author    = {Alnaes, Martin S. and Logg, Anders and {\O}lgaard, Kristian B. and Rognes, Marie E. and Wells, Garth N.},
  journal = {ACM Trans. Math. Softw.},
  fjournal   = {{ACM} Transactions on Mathematical Software},
  year      = {2014},
  volume = {40},
  number       = {2},
  pages        = {9:1--9:37},
  doi       = {10.1145/2566630},
}

@InCollection{ParaView,
  author    = {Ahrens, James and Geveci, Berk and Law, Charles},
  title     = {{ParaView}: An End-User Tool for Large Data Visualization},
  year      = {2005},
  booktitle    = {The Visualization Handbook},
  pages        = {717--731},
  publisher    = {Academic Press / Elsevier},
  year         = {2005},
  doi          = {10.1016/B978-012387582-2/50038-1},
}

@book{LeVeque2007,
  title = {Finite Difference Methods for Ordinary and Partial Differential Equations},
  author = {LeVeque, Randall J.},
  year = {2007},
  publisher = {{Society for Industrial and Applied Mathematics}},
  doi = {10.1137/1.9780898717839}
}

@article{Osterby2003,
  title = {Five {{Ways}} of {{Reducing}} the {{Crank}}--{{Nicolson Oscillations}}},
  author = {{\O}sterby, Ole},
  year = 2003,
  journal = {BIT},
  volume = {43},
  number = {4},
  pages = {811--822},
  doi = {10.1023/B:BITN.0000009942.00540.94}
}

\end{document}